\documentclass[11pt]{article}
\usepackage{xcolor}
\usepackage{mathrsfs}
\usepackage{amsmath,amsthm,amssymb,amscd}
\usepackage{booktabs} 
\usepackage{float}
\usepackage[hyperfootnotes=false, colorlinks, linkcolor={blue}, citecolor={magenta}, filecolor={blue}, urlcolor={blue}, plainpages=false, pdfpagelabels]{hyperref}
\usepackage[overload]{textcase} 
\usepackage{mathtools}
\usepackage[numbers,sort&compress]{natbib} 
\usepackage{etoolbox} 
\apptocmd{\sloppy}{\hbadness 10000\relax}{}{} 
\usepackage[left=2.5cm, right=2.5cm, top=3cm]{geometry}
\allowdisplaybreaks
\usepackage{enumerate}
\usepackage{enumitem}  
\usepackage{bm}
\usepackage{url}
\usepackage[multiple]{footmisc}

\usepackage{colonequals}
\usepackage{longtable} 
\usepackage[ampersand]{easylist}
\usepackage{ltablex,booktabs}
\usepackage{tikz}
\usetikzlibrary{shapes, backgrounds,mindmap, trees, arrows.meta}

\newcommand{\A}{{\mathbb A}}
\newcommand{\B}{B}
\newcommand{\Q}{{\mathbb Q}}
\newcommand{\Z}{{\mathbb Z}}
\newcommand{\R}{{\mathbb R}}
\newcommand{\C}{{\mathbb C}}
\newcommand{\N}{{\mathbb N}}
\newcommand{\F}{{\mathbb F}}
\newcommand{\bs}{\backslash}

\newcommand{\p}{\mathfrak p}
\newcommand{\OF}{{\mathfrak o}}
\newcommand{\GL}{{\rm GL}}
\newcommand{\PGL}{{\rm PGL}}
\newcommand{\SL}{{\rm SL}}

\newcommand{\Sp}{{\rm Sp}}
\newcommand{\GSp}{{\rm GSp}}
\newcommand{\PGSp}{{\rm PGSp}}

\newcommand{\St}{{\rm St}}
\renewcommand{\sc}{{\sf sc}}
\newcommand{\triv}{{\mathbf1}}

\newcommand{\K}[1]{{\rm K}(\p^{#1})}

\newcommand{\Ind}{{\rm Ind}}

\newcommand{\mat}[4]{{\setlength{\arraycolsep}{0.5mm}\begin{bsmallmatrix}#1&#2\\#3&#4\end{bsmallmatrix}}}
\newcommand{\forget}[1]{}
\def\qdots{\mathinner{\mkern1mu\raise0pt\vbox{\kern7pt\hbox{.}}\mkern2mu
		\raise3.4pt\hbox{.}\mkern2mu\raise7pt\hbox{.}\mkern1mu}}

\newtheorem{lemma}{Lemma}[section]
\newtheorem{theorem}[lemma]{Theorem}
\newtheorem{corollary}[lemma]{Corollary}
\newtheorem{proposition}[lemma]{Proposition}
\newtheorem{definition}[lemma]{Definition}

\newcommand\blfootnote[1]{%
	\begingroup
	\renewcommand\thefootnote{}\footnote{#1}%
	\addtocounter{footnote}{-1}%
	\endgroup
}
\usepackage[toc,page, title, titletoc]{appendix}
\makeatother
\makeatletter
\newcommand\appendix@section[1]{%
	\refstepcounter{section}%
	\orig@section*{Appendix \@Alph\c@section: #1}%
	\addcontentsline{toc}{section}{Appendix \@Alph\c@section: #1}%
}
\g@addto@macro\appendix{\let\section\appendix@section}
\let\orig@section\section
\makeatother
\makeatletter
\def\thickhline{%
	\noalign{\ifnum0=`}\fi\hrule \@height \thickarrayrulewidth \futurelet
	\reserved@a\@xthickhline}
\def\@xthickhline{\ifx\reserved@a\thickhline
	\vskip\doublerulesep
	\vskip-\thickarrayrulewidth
	\fi
	\ifnum0=`{\fi}}
\makeatother

\newcommand\KLfour{\Gamma_0'(4)}
\newcommand\KLN{\Gamma_0'(N)}
\newcommand\middlefour{M(4)}
\newcommand\Ekone{E_{k,1}}
\newcommand\Efone{E_{4,1}}

\renewcommand\N{{\mathbb N}}
\newcommand\Span{\operatorname{Span}}
\newcommand\Jac[2]{{\rm J}_{#1,#2}}
\newcommand\Jmero[2]{{\rm J}_{#1,#2}^{\rm mero}}
\newcommand\UHP{{\mathcal H}} 

\newcommand\smallpmat[4]{\left(\begin{smallmatrix}
{#1}&{#2}\\{#3}&{#4}\end{smallmatrix}\right)}
\def\Grit{\operatorname{Grit}}
\newcommand\inv{^{-1}}
\newcommand{\zzeta}{y}

\title{Dimension formulas for Siegel modular forms of level $4$}
\author{Manami Roy, Ralf Schmidt, and Shaoyun Yi}
\date{}

\begin{document}
	
\maketitle
\blfootnote{\hspace*{0.2in } Key words and phrases. Cuspidal automorphic representations, Siegel modular forms, dimension formulas, Satake compactification.}

\vspace{-10ex}
\begin{center}
 with an appendix by Cris Poor and David S.\ Yuen
\end{center}

\vspace{2ex}
\begin{abstract}
We prove several dimension formulas for spaces of scalar-valued Siegel modular forms of degree $2$ with respect to certain congruence subgroups of level $4$. In case of cusp forms, all modular forms considered originate from cuspidal automorphic representations of $\GSp(4,\A)$ whose local component at $p=2$ admits non-zero fixed vectors under the principal congruence subgroup of level $2$. Using known dimension formulas combined with dimensions of spaces of fixed vectors in local representations at $p=2$, we obtain formulas for the number of relevant automorphic representations. These in turn lead to new dimension formulas, in particular for Siegel modular forms with respect to the Klingen congruence subgroup of level~$4$.
\end{abstract}

\tableofcontents
\section{Introduction}
In this work we consider dimension formulas for spaces of scalar-valued Siegel modular forms of degree $2$, weight $k$ and level dividing $4$. The notion of level is ambiguous; for example, level $4$ could refer to modular forms with respect to the paramodular group $K(4)$, the Siegel congruence subgroup $\Gamma_0(4)$, the Klingen congruence subgroup $\Gamma_0'(4)$, or others. We consider the following 11 congruence subgroups of $\Sp(4,\Q)$, all of which are in some sense level $1$, $2$ or $4$:

\vspace{-0.5ex}
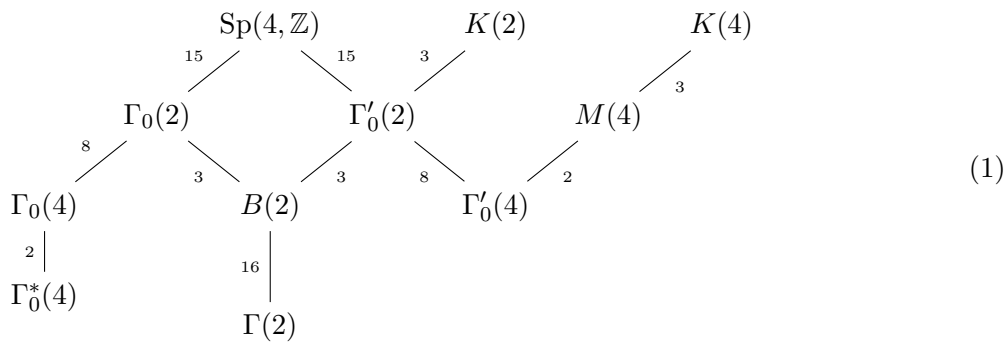
\begin{equation}\label{congruencesubgroupseq}
\raisebox{-14ex}{
\begin{tikzpicture}[auto]
%
 \node (Sp4Z)      at (0,0) {$\Sp(4,\Z)$};
 \node (Gamma02)   at (-1.5,-1.2) {$\Gamma_0(2)$};
 \node (Gamma04)   at (-3,-2.4) {$\Gamma_0(4)$};
 \node (Gamma04s)  at (-3,-3.6) {$\Gamma^*_0(4)$};
 \node (Gammap02)  at (1.5,-1.2) {$\Gamma_0'(2)$};
 \node (Gammap04)  at (3,-2.4) {$\Gamma_0'(4)$};
 \node (B2)        at (0,-2.4) {$B(2)$};
 \node (Gamma2)    at (0,-4) {$\Gamma(2)$};
 \node (K2)        at (3,0) {$K(2)$};
 \node (K4)        at (6,0) {$K(4)$};
 \node (M4)        at (4.5,-1.2) {$M(4)$};

 \draw (Gamma04s) to node {$\scriptscriptstyle2$} (Gamma04) to node {$\scriptscriptstyle8$} (Gamma02) to node {$\scriptscriptstyle15$} (Sp4Z);
 \draw (Gamma02) to node [swap] {$\scriptscriptstyle3$} (B2) to node [swap] {$\scriptscriptstyle16$} (Gamma2);
 \draw (Sp4Z) to node {$\scriptscriptstyle15$} (Gammap02) to node {$\scriptscriptstyle3$} (B2);
 \draw (K2) to node [swap] {$\scriptscriptstyle3$} (Gammap02) to node [swap] {$\scriptscriptstyle8$} (Gammap04);
 \draw (K4) to node {$\scriptscriptstyle3$} (M4) to node {$\scriptscriptstyle2$} (Gammap04);
\end{tikzpicture}
}
\end{equation}
The connecting lines indicate inclusions (with the bigger group on top), and their labels show indices. The group $\Gamma(2)$ is the principal congruence subgroup of level $2$, $B(2)$ is the Borel congruence subgroup of level $2$, $\Gamma_0^*(4)$ is a certain subgroup of index $2$ in $\Gamma_0(4)$, and $M(4)$ is the ``middle'' group, which lies between $\Gamma_0'(4)$ and $K(4)$. For precise definitions, see Table~\ref{congruencesubgroupstable} in the notations section.

For many of the subgroups $\Gamma$ in \eqref{congruencesubgroupseq} the dimension of the space of Siegel modular forms $M_k(\Gamma)$ and the subspace of cusp forms $S_k(\Gamma)$ is known; Table~\ref{historytable} gives some references. In case a conjugate of $\Gamma$ lies between $\Gamma(2)$ and $\Sp(4,\Z)$, there is a well-known method based on Igusa's classic paper \cite{Igusa1964}. Theorem~2 of \cite{Igusa1964} gives the character of the representation of $\Sp(4,\Z/2\Z)\cong S_6$ on $M_k(\Gamma(2))$. Using standard character theory, one can thus easily calculate $\dim M_k(\Gamma)$. This method works for all groups in \eqref{congruencesubgroupseq} except $K(2)$ and $\Gamma_0'(4)$. The results are summarized in Table~\ref{dimMktable}. All the dimension formulas in this paper are packaged into generating series like $\sum_{k=0}^\infty\dim M_k(\Gamma)t^k$.

The result for $K(2)$ in Table~\ref{dimMktable} is taken from the literature. The result for $\Gamma_0'(4)$, which follows from our considerations using automorphic representation theory, is new. In fact, calculating $\dim M_k(\Gamma_0'(4))$ and $\dim S_k(\Gamma_0'(4))$ provided the original motivation for the present work.

At least for $k\geq6$ the codimension $\dim M_k(\Gamma)-\dim S_k(\Gamma)$ can be determined from the cusp structure of the Satake compactification and Satake's characterization of the image of the global $\Phi$-map. In degree $2$ the method distills down to a simple formula, which we record in Theorem~\ref{thmcodim}. We thus obtain codimension formulas for all the groups in \eqref{congruencesubgroupseq}; see Table~\ref{codimensionstable}. Together with the information about $\dim M_k(\Gamma)$ in Table~\ref{dimMktable}, we get the dimension formulas for $\dim S_k(\Gamma)$ in Table~\ref{dimSktable} for all $\Gamma$ except $\Gamma_0'(4)$.

To obtain further results we consider the automorphic representations $\pi$ generated by the eigenforms in $S_k(\Gamma)$ for $\Gamma$ in \eqref{congruencesubgroupseq}. If we factor an irreducible such $\pi$ into local representations $\pi\cong\otimes\pi_v$ with irreducible, admissible representations $\pi_v$ of $\GSp(4,\Q_v)$, then $\pi_\infty$ is a ``holomorphic'' representation of lowest weight $k$, and $\pi_p$ is unramified for all primes $p\geq3$. If $\Gamma$ is not equal to $\Gamma_0'(4)$, then $\Gamma$ contains a conjugate of $\Gamma(2)$, and consequently $\pi_2$ will have non-zero fixed vectors under the local principal congruence subgroup $\Gamma(\p)$, where $\p=2\Z_2$. A complete determination of such $\pi_2$, which are also known as representations with non-zero hyperspecial parahoric restriction, has been achieved in \cite{Rosner2016}. We reproduce the list of irreducible, admissible representations of $\GSp(4,\Q_2)$ with non-zero hyperspecial parahoric restriction in Table~\ref{parahoricrestrictiontable}. They are organized into types I, IIa, IIb,~\ldots. We let $S_k(\Omega)$ be the set of cuspidal, automorphic representations which are holomorphic of weight~$k$ at the archimedean place, unramified outside $2$, and are of type $\Omega$ with non-zero hyperspecial parahoric restriction at $p=2$; see Definition~\ref{Skomegadef} for more details. We note that $S_k(\Omega)$ is a finite set; see~\cite{BorelJacquet1979}.

A key result which allows us to get information about $\Gamma_0'(4)$ is \cite[Lemma 4]{Yi2021}. It implies that if an irreducible, admissible representation of $\GSp(4,\Q_2)$ has non-zero $\Gamma_0'(\p^2)$-invariant vectors, then it also has non-zero $\Gamma(\p)$-invariant vectors, and hence appears in Table~\ref{parahoricrestrictiontable}. Therefore, eigenforms in $S_k(\Gamma_0'(4))$ also generate elements of $S_k(\Omega)$ for some $\Omega$. Conversely, given an automorphic representation $\pi\cong\otimes\pi_v$ in $S_k(\Omega)$, and a non-zero vector in $\pi_2$ invariant under the local congruence subgroup $C$ analogous to $\Gamma$ for some $\Gamma$ in \eqref{congruencesubgroupseq}, we can construct an element of $S_k(\Gamma)$ by ``descending'' to the Siegel upper half space $\UHP_2$. We thus get the relation
\begin{equation}\label{dimSkdCOmegaeq2}
 \dim S_k(\Gamma)=\sum_\Omega s_k(\Omega)d_{C,\Omega},
\end{equation}
where $s_k(\Omega)=|S_k(\Omega)|$ and $d_{C,\Omega}$ is the common dimension of the space of $C$-invariant vectors in representations of type $\Omega$ occuring in Table~\ref{parahoricrestrictiontable}. The numbers $d_{C,\Omega}$ can all be calculated and are listed in Table~\ref{localdimtable}.

Observe that \eqref{dimSkdCOmegaeq2} is a system of linear equations relating the $\dim S_k(\Gamma)$ for all $\Gamma$ and the $s_k(\Omega)$ for all $\Omega$. Recall that the $\dim S_k(\Gamma)$ are already known for all $\Gamma$ except $\Gamma_0'(4)$. Essentially now what happens is that as $\Gamma$ runs through the subgroups in \eqref{congruencesubgroupseq} except $\Gamma_0'(4)$ the system \eqref{dimSkdCOmegaeq2} provides enough equations in order to determine the $s_k(\Omega)$. Once these are known we use \eqref{dimSkdCOmegaeq2} again, this time for $\Gamma=\Gamma_0'(4)$, to determine $\dim S_k(\Gamma_0'(4))$.

In full detail the situation is slightly more complicated because the system \eqref{dimSkdCOmegaeq2} has more unknowns than equations. This hurdle is overcome by exploiting that automorphic representations of $\GSp(4,\A)$ are categorized into six different kinds of Arthur packets. In Proposition~\ref{YQBzeroprop} we will prove that only packets of ``general type'', denoted by \textbf{(G)}, and packets of Saito-Kurokawa type, denoted by \textbf{(P)}, are relevant. Considering the \textbf{(G)} version and the \textbf{(P)} version of \eqref{dimSkdCOmegaeq2} separately reduces the number of equations and makes the method work. As a byproduct we obtain refined dimension formulas for the spaces of non-lifts $S^{\text{\bf(G)}}_k(\Gamma)$ and the spaces of lifts $S^{\text{\bf(P)}}_k(\Gamma)$ (see Sect.~\ref{Arthursec} for a more precise definition of these spaces). We remark that we included the groups $\Gamma_0^*(4)$ and $M(4)$ in our list \eqref{congruencesubgroupseq} in order to obtain two more linear equations; without these the system 
\eqref{dimSkdCOmegaeq2} would still be underdetermined.

There are only two spaces which are not accessible with the above methods, namely $M_2(\Gamma_0'(4))$ and $M_4(\Gamma_0'(4))$. Their dimensions have been determined by Cris Poor and David S.\ Yuen in Appendix~\ref{CDappendix}.

As mentioned above, many dimension formulas for Siegel modular forms are already contained in the literature. The new contributions of the present work are as follows.
\begin{itemize}
 \item Siegel modular forms for the groups $\Gamma_0^*(4)$ and $M(4)$ have not previously received much attention in the literature. (See however the ``paramodular groups with level'' defined in \cite{Fabien2009thesis}.)
 \item The dimension formulas for $\Gamma_0'(4)$ are new. Until now, the literature only contains dimension formulas for $\Gamma_0'(p)$ where $p$ is prime; see \cite{Ibukiyama1984}, \cite{HashimotoIbukiyama1985}, \cite{Ibukiyama2007}, \cite{Wakatsuki2013}.
 
 \item We obtain the refined dimension formulas for the spaces of lifts $S^{\text{\bf(P)}}_k(\Gamma)$ and non-lifts $S^{\text{\bf(G)}}_k(\Gamma)$.
 \item We obtain formulas for $s_k(\Omega)$, the number of cuspidal automorphic representations of \linebreak $\PGSp(4,\A)$ of weight $k$, unramified outside $2$, and with a representation of type $\Omega$ at $p=2$ admitting non-zero $\Gamma(\p)$-invariant vectors.
\end{itemize}
The paper is organized as follows. In Sect.~\ref{localdimsec} we collect the necessary facts from local representation theory. The main outcomes are Table~\ref{parahoricrestrictiontable}, the complete list of all relevant representations, and Table~\ref{localdimtable}, which contains the dimensions of the spaces of fixed vectors in these representations under all relevant local congruence subgroups. In Sect.~\ref{dimcodimsec}, which is largely independent from Sect.~\ref{localdimsec}, we first utilize Satake's method to obtain codimension formulas for all $\Gamma$ in \eqref{congruencesubgroupseq}. Combined with Igusa's result we thus obtain dimension formulas for $M_k(\Gamma)$ and $S_k(\Gamma)$ for all $\Gamma$ in \eqref{congruencesubgroupseq} except $K(2)$ and $\Gamma_0'(4)$. The formulas for $K(2)$ are already known, and the ones for $\Gamma_0'(4)$ will follow as a consequence of our other results. Sect.~\ref{Sect:countingskpO} begins with a review of Arthur packets for $\GSp(4)$. We make the connection between Siegel modular forms and representations, resulting in the system of linear equations \eqref{dimSkdCOmegaeq2}. We then derive the numbers $s_k(\Omega)$, first for Saito-Kurokawa lifts, then for representations of general type. Finally, as an application, we obtain the desired dimension formulas for $\Gamma_0'(4)$.

Most of our results are summarized in table form in Appendix~\ref{TablesAppendix}. More precisely, Table~\ref{dimMktable} and~\ref{dimSktable} contain dimension formulas for $M_k(\Gamma)$ and $S_k(\Gamma)$, respectively. Table~\ref{dimSkPtable} and \ref{dimSkGtable} are for dimension formulas of $S^{\text{\bf(P)}}_k(\Gamma)$ and $S^{\text{\bf(G)}}_k(\Gamma)$, respectively. Table~\ref{skPOmegatable} and \ref{skGOmegatable} contain formulas for $s^{\text{\bf(P)}}_k(\Omega)$ and $s^{\text{\bf(G)}}_k(\Omega)$, respectively. Here, $s_k^{(*)}(\Omega)=|S_k^{(*)}(\Omega)|$ (see Sect.~\ref{Arthursec}).
 Tables~\ref{Mklowweightstable}-\ref{SkGlowweightstable} and Table~\ref{SkOmegalowweightstable} provide numerical examples for weight $k\le 20$. Appendix~\ref{CDappendix}, provided by Cris Poor and David S.~Yuen, fills the final gap by calculating $\dim M_k(\Gamma_0'(4))$ for $k=2$ and $k=4$.

\vspace*{0.1in}
\textbf{Acknowledgements.} We would like to thank Tomoyoshi Ibukiyama and Cris Poor for providing helpful comments. We would also like to thank the referee for the careful revision and a number of insightful suggestions. Shaoyun Yi is supported by the Fundamental Research Funds for the Central Universities (No. 20720230025).

\section*{Notation and preliminaries}

The symbols $\Z$, $\Q$, $\R$ have the usual meaning. The symbol $\Q_p$ stands for the field of $p$-adic numbers. We will write $\mathbb{F}_p$ for the field with $p$ elements; only $\mathbb{F}_2$ is needed in this work.

Let $J$ be a $4\times4$ anti-symmetric matrix over a field $F$. We consider the symplectic similitude group
\begin{equation}\label{GSp4_symmetric_version}
 G=\GSp(4) \colonequals \{g\in\GL(4)\colon \:^tgJg=\lambda(g)J,\:\lambda(g)\in \GL(1)\},
\end{equation}
which is an algebraic $F$-group. The function $\lambda$ is called the multiplier homomorphism. The kernel of this function is the symplectic group $\Sp(4)$. Let $Z$ be the center of $\GSp(4)$ and $\PGSp(4):=\GSp(4)/Z$.

While all choices of $J$ lead to isomorphic groups, one or the other choice might be more convenient depending on the context. When working with classical Siegel modular forms, the usual choice for $J$ is\footnote{Empty entries in matrices mean zeros.}
\begin{equation}\label{J1defeq}
 J_1=\begin{bsmallmatrix}&&1\\&&&1\\-1\\&-1\end{bsmallmatrix},
\end{equation}
leading to the ``classical'' version of the symplectic group. When working with local representations it is often more convenient to use
\begin{equation}\label{J2defeq}
 J_2=\begin{bsmallmatrix}&&&1\\&&1\\&-1\\-1\end{bsmallmatrix},
\end{equation}
resulting in the ``symmetric'' version of the symplectic group. For example, the standard Borel subgroup in the second version consists of upper triangular matrices. We will allow ourselves to use both versions of $\GSp(4)$. An isomorphism between them is obtained by switching the first two rows and columns.

We will utilize the following representatives for elements of the Weyl group,
\begin{equation}\label{s1s2defeq}
 s_1=\begin{bsmallmatrix}&1\\1\\&&&1\\&&1\end{bsmallmatrix},\qquad s_2=\begin{bsmallmatrix}&&1\\&1\\-1\\&&&1\end{bsmallmatrix},
\end{equation}
given in the $J_1$ version of $\Sp(4)$.

\textbf{Local and global congruence subgroups.}
For a positive integer $N$, we define $\Gamma_0^{(1)}(N):=\SL(2,\Z)\cap\mat{\Z}{\Z}{N\Z}{\Z}$. In degree $2$, we will work with a number of congruence subgroups of ``level $2$'' and ``level $4$''. The global subgroups are contained in $\Sp(4,\Q)$, and except for the paramodular group $K(2)$ all of them can be conjugated into the full modular group $\Sp(4,\Z)$. Locally, we work over the field $\Q_2$ and denote by $\OF$ its ring of integers $\Z_2$ and by $\p$ the maximal ideal $2\Z_2$ of $\OF$. All our subgroups will be contained in $G^1:=\{g\in\GSp(4,\Q_2):\lambda(g)\in\OF^\times\}$, and except for the paramodular group $K(\p)$ all of them can be conjugated into the hyperspecial maximal compact subgroup $K:=\GSp(4,\Z_2)$. Table~\ref{congruencesubgroupstable} shows the notations we use for various congruence subgroups. Note that for the global groups we use the symplectic form $J_1$, and for the local groups we use the symplectic form $J_2$.

\begin{table}
 \caption{Global and local congruence subgroups.}
 \label{congruencesubgroupstable}
\begin{equation*}\renewcommand{\arraycolsep}{0.9ex}
 \begin{array}{cc|cc|c}
  \toprule
   \multicolumn{2}{c|}{\text{global groups}}&\multicolumn{2}{c|}{\text{local groups}}&\\
   \text{symbol}&\text{definition}&\text{symbol}&\text{definition}&\text{name}\\
  \toprule
   \Gamma(N)&\Sp(4,\Z)\cap\begin{bsmallmatrix}1+N\Z&N\Z&N\Z&N\Z\\N\Z&1+N\Z&N\Z&N\Z\\N\Z&N\Z&1+N\Z&N\Z\\N\Z&N\Z&N\Z&1+N\Z\end{bsmallmatrix}&\Gamma(\p^n)&K\cap\begin{bsmallmatrix}1+\p^n&\p^n&\p^n&\p^n\\\p^n&1+\p^n&\p^n&\p^n\\\p^n&\p^n&1+\p^n&\p^n\\\p^n&\p^n&\p^n&1+\p^n\end{bsmallmatrix}&\begin{minipage}{12ex}\small principal congruence subgroup\end{minipage}\\
  \midrule
   K(N)&\Sp(4,\Q)\cap\begin{bsmallmatrix}\Z&N\Z&\Z&\Z\\\Z&\Z&\Z&N^{-1}\Z\\\Z&N\Z&\Z&\Z\\N\Z&N\Z&N\Z&\Z\end{bsmallmatrix}&K(\p^n)&G^1\cap\begin{bsmallmatrix}\OF&\OF&\OF&\p^{-n}\\\p^n&\OF&\OF&\OF\\\p^n&\OF&\OF&\OF\\\p^n&\p^n&\p^n&\OF\end{bsmallmatrix}&\begin{minipage}{12ex}\small paramodular group\end{minipage}\\
  \midrule
   \Gamma_0(N)&\Sp(4,\Z)\cap\begin{bsmallmatrix}\Z&\Z&\Z&\Z\\\Z&\Z&\Z&\Z\\N\Z&N\Z&\Z&\Z\\N\Z&N\Z&\Z&\Z\end{bsmallmatrix}&\Gamma_0(\p^n)&K\cap\begin{bsmallmatrix}\OF&\OF&\OF&\OF\\\OF&\OF&\OF&\OF\\\p^n&\p^n&\OF&\OF\\\p^n&\p^n&\OF&\OF\end{bsmallmatrix}&\begin{minipage}{12ex}\small Siegel\\ congruence subgroup\end{minipage}\\
  \midrule
   \Gamma_0'(N)&\Sp(4,\Z)\cap\begin{bsmallmatrix}\Z&N\Z&\Z&\Z\\\Z&\Z&\Z&\Z\\\Z&N\Z&\Z&\Z\\N\Z&N\Z&N\Z&\Z\end{bsmallmatrix}&\Gamma_0'(\p^n)&K\cap\begin{bsmallmatrix}\OF&\OF&\OF&\OF\\\p^n&\OF&\OF&\OF\\\p^n&\OF&\OF&\OF\\\p^n&\p^n&\p^n&\OF\end{bsmallmatrix}&\begin{minipage}{12ex}\small Klingen\\ congruence subgroup\end{minipage}\\
  \midrule
   B(N)&\Sp(4,\Z)\cap\begin{bsmallmatrix}\Z&N\Z&\Z&\Z\\\Z&\Z&\Z&\Z\\N\Z&N\Z&\Z&\Z\\N\Z&N\Z&N\Z&\Z\end{bsmallmatrix}&B(\p^n)&K\cap\begin{bsmallmatrix}\OF&\OF&\OF&\OF\\\p^n&\OF&\OF&\OF\\\p^n&\p^n&\OF&\OF\\\p^n&\p^n&\p^n&\OF\end{bsmallmatrix}&\begin{minipage}{12ex}\small Borel\\ congruence subgroup\end{minipage}\\
  \midrule
   M(4)&\Sp(4,\Q)\cap\begin{bsmallmatrix}\Z&4\Z&\Z&\Z\\\Z&\Z&\Z&2^{-1}\Z\\\Z&4\Z&\Z&\Z\\4\Z&4\Z&4\Z&\Z\end{bsmallmatrix}&M(\p^2)&G^1\cap\begin{bsmallmatrix}\OF&\OF&\OF&\p^{-1}\\\p^2&\OF&\OF&\OF\\\p^2&\OF&\OF&\OF\\\p^2&\p^2&\p^2&\OF\end{bsmallmatrix}&\begin{minipage}{12ex}\small The\\ ``middle'' group\end{minipage}\\
  \midrule
   \Gamma_0^*(4)&\begin{minipage}{30ex}\Big\{$\begin{bsmallmatrix}A&B\\C&D\end{bsmallmatrix}\in\Gamma_0(4):D\bmod2$\\\phantom{xx}$\in\{\begin{bsmallmatrix}1&0\\0&1\end{bsmallmatrix},\begin{bsmallmatrix}1&1\\1&0\end{bsmallmatrix},\begin{bsmallmatrix}0&1\\1&1\end{bsmallmatrix}\}\Big\}$\end{minipage}&\Gamma_0^*(\p^2)&\begin{minipage}{28ex}\Big\{$\begin{bsmallmatrix}A&B\\C&D\end{bsmallmatrix}\in\Gamma_0(\p^2):D\bmod \p$\\\phantom{xx}$\in\{\begin{bsmallmatrix}1&0\\0&1\end{bsmallmatrix},\begin{bsmallmatrix}1&1\\1&0\end{bsmallmatrix},\begin{bsmallmatrix}0&1\\1&1\end{bsmallmatrix}\}\Big\}$\end{minipage}&\\
  \midrule
 \end{array}
\end{equation*}
\end{table}

\textbf{Siegel modular forms.}
Let $\UHP_n$ be the Siegel upper half space of degree $n$, i.e., $\UHP_n$ consists of all symmetric complex $n\times n$ matrices whose imaginary part is positive definite. The principal congruence subgroup $\Gamma(N)$ of $\Sp(2n,\Z)$ is the kernel of the reduction map $\Sp(2n,\Z)\to\Sp(2n,\Z/N\Z)$. By a congruence subgroup of $\Sp(2n,\Q)$ we mean a subgroup of $\Sp(2n,\Q)$ which, for some $N$, contains $\Gamma(N)$ with finite index.
\begin{definition}
 A Siegel modular form of degree $n$ and weight $k$ with respect to a congruence subgroup $\Gamma$ of $\Sp(2n,\Q)$ is a holomorphic function $f:\UHP_n\rightarrow \C$ with the transformation property
 \begin{equation*}
  (f|_{k}g)(Z)=j(g,Z)^{-k}f((AZ+B)(CZ+D)^{-1})=f(Z)\ \text{for }  g=\begin{bsmallmatrix}A&B\\C&D\end{bsmallmatrix}  \in \Gamma,
 \end{equation*}
 where $j(g,Z)=\det(CZ+D)$, and which satisfies the usual moderate growth condition if $n=1$.
\end{definition}
We call a Siegel modular form $f$ a \emph{cusp form} if 
\begin{equation*}
 \lim\limits_{\lambda \rightarrow \infty} (f|_{k}g)\left(\begin{bsmallmatrix} \tau&\\& i\lambda\end{bsmallmatrix} \right)=0\quad \text{for all } g \in \Sp(2n,\Q) \text{ and } \tau \in \UHP_{n-1}.
\end{equation*}
In this work we will primarily consider Siegel modular forms of degree $2$, and occasionally modular forms of degree $1$. We denote by $M_{k}(\Gamma)$ the space of Siegel modular forms of degree $2$ and weight $k$ with respect to the congruence subgroup $\Gamma$ of $\Sp(4,\Q)$, and by $S_{k}(\Gamma)$ its subspace of cusp forms. We denote by $M_k^{(1)}(\Gamma)$ the space of modular forms of degree $1$ and weight $k$ with respect to the congruence subgroup $\Gamma$ of $\SL(2,\Q)$, and by $S_k^{(1)}(\Gamma)$ its subspace of cusp forms.

\textbf{A lemma on rational points and integral points.} For lack of a good reference, we include a proof of the following result. It will be used in Sect.~\ref{gencodimsec}.
\begin{lemma}\label{SpnQlemma}
 Let $n$ be a positive integer. Let $R$ be any standard parabolic subgroup of $\Sp(2n)$. Then
 $$
  \Sp(2n,\Q)=R(\Q)\Sp(2n,\Z).
 $$
\end{lemma}
\begin{proof}
Let $g\in\Sp(2n,\Q)$. For any place $p$, let $K_p$ be the standard maximal compact subgroup of $\Sp(2n,\Q_p)$. Let $K=\prod K_p$. Use the Iwasawa decomposition to write $g=r_p\kappa_p$, with $r_p\in R(\Q_p)$ and $\kappa_p\in K_p$. Then $g=r\kappa$, where $r=(r_p)$ and $\kappa=(\kappa_p)$. Let $R=MN$ be the Levi decomposition of $R$. Write $r=mn$ with $m\in M(\A)$ and $n\in N(\A)$. By strong approximation, we may write
$$
 m=m_\Q m_\R m_K\qquad\text{with }m_\Q\in M(\Q),\;m_\R\in M(\R),\;m_K\in M(\A)\cap K.
$$
The element $m_K$ may be absorbed into $K$ (possibly modifying $n$), and may therefore assumed to be~$1$. Using $\A=\Q+\R+\prod_{p<\infty}\Z_p$, we can write
$$
 n=n_\Q n_\R n_K,\qquad n_\Q\in N(\Q),\;n_\R\in N(\R),\;n_K\in N(\A)\cap K.
$$
The element $n_K$ may be absorbed into $K$, and therefore assumed to be $1$. Summarizing, we see that we can write
$$
 g=r_\Q r_\R\kappa,\qquad r_\Q\in R(\Q),\;r_\R\in R(\R),\;\kappa\in K.
$$
The matrix $r_\Q^{-1}g$ lies in $\Sp(2n,\Z_p)$ for all finite $p$, and hence in $\Sp(2n,\Z)$. This concludes the proof.
\end{proof}

\section{Local dimensions}\label{localdimsec}
We will start this section by collecting some facts about the symmetric group $S_6$. In subsections \ref{parahoricrestrictionsec} and \ref{localfixedvectorssec} we work over the $p$-adic field $\Q_2$ and write $\OF$ for its ring of integers and $\p$ for the maximal ideal of $\OF$. In this local context it is convenient to work with the ``symmetric'' version of the symplectic group, defined by the symplectic form $J_2$ given in \eqref{J2defeq}.
\subsection{Preliminaries on \texorpdfstring{$S_6$}{}}
Let $(\rho, U)$ be a representation of a finite group $G$, and let $H$ be a subgroup of $G$. Then $p=\frac{1}{|H|}\sum_{h\in H}\rho(h)$ is a projector onto the subspace $U^H$ of $H$-fixed vectors. Hence
\begin{equation}\label{dimUH}
 \dim U^H = \text{Tr}(p)=\frac{1}{|H|}\sum_{h\in H}\chi_\rho(h),
\end{equation}
where $\chi_\rho$ is the character of $\rho$. More generally, if $\tau$ is a representation of $H$, then the multiplicity
of $\tau$ in $\rho|_{H}$ is
\begin{equation}\label{multirepH}
 \text{mult}_{\rho}(\tau)=\frac{1}{|H|}\sum_{h\in H}\overline{\chi_\tau(h)}\chi_\rho(h).
\end{equation}
We will apply this principle to the finite group $\Sp(4,\Z)/\Gamma(2)\cong \Sp(4,\F_2)$. In order to do so, we will exhibit an explicit isomorphism with the symmetric group~$S_6$.

Consider the natural permutation action of $S_6$ on the space of column vectors $(\mathbb{F}_2)^6$. Let $W$ be the $5$-dimensional subspace of vectors whose coordinates add up to zero. Let $U$ be the subspace of $W$ spanned by $u=\,^t(1,1,1,1,1,1)$. Then $W$ and $U$ are both invariant under the action of $S_6$, so that we get an action on the four-dimensional space $W/U$. There is a symplectic (and symmetric) form on $W$ given by
$$
 \langle x,y\rangle=\sum_{i=1}^6x_iy_i,\qquad x=\,^t(x_1,\ldots,x_6),\;y=\,^t(y_1,\ldots,y_6).
$$
This form is degenerate with radical $U$, thus inducing a non-degenerate symplectic form on the quotient $W/U$. Evidently, this form is invariant under the action of $S_6$. We thus obtain a non-trivial homomorphism $S_6\to\Sp(4,\mathbb{F}_2)$. Since the image of this map has more than two elements, and since $A_6$ is the only non-trivial, proper, normal subgroup of $S_6$, the map is injective. Since both groups have the same number of elements, it is an isomorphism. To make the isomorphism more explicit, let
$$
 e_1=\begin{bsmallmatrix}1\\1\\0\\0\\0\\0\end{bsmallmatrix},\qquad
 e_2=\begin{bsmallmatrix}0\\0\\0\\0\\1\\1\end{bsmallmatrix},\qquad
  f_2=\begin{bsmallmatrix}0\\0\\0\\1\\1\\0\end{bsmallmatrix},\qquad
 f_1=\begin{bsmallmatrix}0\\1\\1\\0\\0\\0\end{bsmallmatrix}.
$$
Then $W=\langle e_1,e_2,f_2,f_1,u\rangle$. The images of $e_1,e_2,f_2,f_1$ form a basis of $W/U$, with respect to which the form $\langle\,,\,\rangle$ has matrix $J_2$ defined in \eqref{J2defeq}. Easy calculations then show that on certain elements the isomorphism $S_6\to\Sp(4,\mathbb{F}_2)=\{g\in\GL(4,\mathbb{F}_2)\,:\,^tgJ_2g=J_2\}$ has the following explicit description.
\begin{xalignat}{2}\label{S6Sp4expleq1}
 (16)(25)(34)&\longmapsto\begin{bsmallmatrix}&1\\1\\&&&1\\&&1\end{bsmallmatrix},&
 (46)&\longmapsto\begin{bsmallmatrix}1\\&&1&\\&1&&\\&&&1\end{bsmallmatrix},\\
 (13)(46)&\longmapsto\begin{bsmallmatrix}&&&1\\&&1&\\&1\\1\end{bsmallmatrix},&
 (12)(36)(45)&\longmapsto\begin{bsmallmatrix}1&1\\&1\\&&1&1\\&&&1\end{bsmallmatrix},\\
 (12)(34)(56)&\longmapsto\begin{bsmallmatrix}1&&1&\\&1&&1\\&&1\\&&&1\end{bsmallmatrix},&
 (12)&\longmapsto\begin{bsmallmatrix}1&&&1\\&1\\&&1\\&&&1\end{bsmallmatrix},\\
 (135)(246)&\longmapsto\begin{bsmallmatrix}1&1\\1\\&&1&1\\&&1&\end{bsmallmatrix},&
 (153)(264)&\longmapsto\begin{bsmallmatrix}&1\\1&1\\&&&1\\&&1&1\end{bsmallmatrix}.
\end{xalignat}
Using such a description, it is easy to determine the number of elements of a given cycle type in certain subgroups of $\Sp(4,\mathbb{F}_2)\cong S_6$. Table~\ref{conjugacyclassestable} shows such data for a number of subgroups of $\Sp(4,\mathbb{F}_2)$ (the first one of which is the trivial and the second one of which is the full subgroup). All these subgroups are obtained as the image of a conjugate of a congruence subgroup $\Gamma$ of $\Sp(4,\Q)$, this conjugate lying between $\Gamma(2)$ and $\Sp(4,\Z)$, under the projection map $\Sp(4,\Z)\to\Sp(4,\mathbb{F}_2)$; the first column of Table~\ref{conjugacyclassestable} shows the congruence subgroup~$\Gamma$.

\begin{table}
 \caption{The number of elements of a given cycle type in some subgroups of $\Sp(4,\mathbb{F}_2)\cong S_6$. Here, we use the ``symmetric'' form of $\Sp(4)$, i.e., the one defined with the symplectic form $J_2$ as in \eqref{J2defeq}. The third column shows the cardinality of the subgroup.}
 \label{conjugacyclassestable}
\begin{equation*}\setlength{\arraycolsep}{0.5ex}\renewcommand{\arraystretch}{2}
 \begin{array}{cccccccccccccc}
  \toprule
   &&&\;\;\;{\scriptstyle1}\;\;\;&&\!\!\!{\scriptstyle(12)(34)}\!\!\!&&\scriptstyle(123)&&\!\!\!\!\!\!\!{\scriptstyle(123)(456)}\!\!\!\!&&\!\!\!\!{\scriptstyle(1234)(56)}\!\!\!\!&&\!\!\!\scriptstyle(123456)\\[-1ex]
   \Gamma&\subset\Sp(4,\mathbb{F}_2)&\;\;\#\;\;&&\scriptstyle(12)&&\!\!\!\!\!\!\!{\scriptstyle(12)(34)(56)}\!\!\!\!\!\!\!&&\!\!{\scriptstyle(123)(45)}\!\!\!\!\!\!&&\!\!{\scriptstyle(1234)}\!\!\!&&\!\!\!{\scriptstyle(12345)}\!\!\!&\\
  \toprule
   \Gamma(2)&\begin{bsmallmatrix}*\\&*\\&&*\\&&&*\end{bsmallmatrix}&1&1&0&0&0&0&0&0&0&0&0&0\\
  \midrule
   \Sp(4,\Z)&\begin{bsmallmatrix}*&*&*&*\\ *&*&*&*\\ *&*&*&*\\ *&*&*&*\end{bsmallmatrix}&720&1&15&45&15&40&\!120\!\!\!&\!40\!&90&90&144&120\\
  \midrule
   K(4)&\begin{bsmallmatrix}*&&&*\\&*&*&\\&*&*&\\ *&&&*\end{bsmallmatrix}&36&1&6&9&0&4&12&4&0&0&0&0\\
  \midrule
   \Gamma_0(2)&\begin{bsmallmatrix}*&*&*&*\\ *&*&*&*\\&&*&*\\&&*&*\end{bsmallmatrix}&48&1&3&9&7&0&0&8&6&6&0&8\\
  \midrule
   \Gamma_0(4)&\begin{bsmallmatrix}*&*\\ *&*\\&&*&*\\&&*&*\end{bsmallmatrix}&6&1&0&0&3&0&0&2&0&0&0&0\\
  \midrule
   \Gamma_0^*(4)&\text{index $2$ in }\begin{bsmallmatrix}*&*\\ *&*\\&&*&*\\&&*&*\end{bsmallmatrix}&3&1&0&0&0&0&0&2&0&0&0&0\\
  \midrule
   \Gamma_0'(2)&\begin{bsmallmatrix}*&*&*&*\\&*&*&*\\&*&*&*\\&&&*\end{bsmallmatrix}&48&1&7&9&3&8&8&0&6&6&0&0\\
  \midrule
   M(4)&\begin{bsmallmatrix}*\\&*&*&\\&*&*&\\ *&&&*\end{bsmallmatrix}&12&1&4&3&0&2&2&0&0&0&0&0\\
  \midrule
   B(2)&\begin{bsmallmatrix}*&*&*&*\\&*&*&*\\&&*&*\\&&&*\end{bsmallmatrix}&16&1&3&5&3&0&0&0&2&2&0&0\\
  \bottomrule
 \end{array}
\end{equation*}
\end{table}
Both the conjugacy classes and the irreducible characters of $S_6$ (also referred to as $S_6$-types) are parametrized by partitions of~$6$. We write $[n_1,\ldots,n_r]$ for the irreducible character of $S_6$ corresponding to the partition $6=n_1+\ldots+n_r$. For example, $[6]$ is the trivial character and $[1,1,1,1,1,1]$ is the sign character of $S_6$. The character table of $S_6$ is given in \cite[p.~400]{Igusa1964}. Using formula \eqref{dimUH}, the data in Table~\ref{conjugacyclassestable}, and the character table, we obtain the dimension of the space of fixed vectors in each $S_6$-type under the subgroups listed in Table~\ref{conjugacyclassestable}. The results are summarized in Table~\ref{S6typesfixedtable}. The last two rows of Table~\ref{S6typesfixedtable} indicate the generic representations (i.e., those which admit a non-zero Whittaker functional) and the cuspidal representations (i.e., those with no non-zero fixed vectors under the unipotent radicals of the parabolic subgroups).

\begin{table}
 \caption{The dimensions of the spaces of fixed vectors in each $S_6$-type under some subgroups of $\Sp(4,\mathbb{F}_2)\cong S_6$. Here, we use the ``symmetric'' form of $\Sp(4)$, i.e., the one defined with the symplectic form $J_2$ as in \eqref{J2defeq}. The ``$\Gamma(2)$'' row shows the dimensions of each $S_6$-type.}
 \label{S6typesfixedtable}
\begin{equation*}\setlength{\arraycolsep}{0.5ex}\renewcommand{\arraystretch}{2}
 \begin{array}{cccccccccccccc}
  \toprule
   &&\;{\scriptstyle[6]}\;\;&&{\scriptstyle[4,2]}&&\scriptstyle[3,3]&&\!\!\!\!{\scriptstyle[3,1,1,1]}\!\!&&\!\!\!\!{\scriptstyle[2,2,1,1]}\!\!&&\!\!\!\!\!\!\!\scriptstyle[1,1,1,1,1,1]\!\!\!\\[-1ex]
   \Gamma&\subset\Sp(4,\mathbb{F}_2)&&\scriptstyle[5,1]&&\!{\scriptstyle[4,1,1]}\!&&{\scriptstyle[3,2,1]}\!\!\!&&\!\!{\scriptstyle[2,2,2]}\!\!&&\!\!\!\!\!\!{\scriptstyle[2,1,1,1,1]}\!\!\!\!\!\!\!&\\
  \toprule
   \Gamma(2)&\begin{bsmallmatrix}*\\&*\\&&*\\&&&*\end{bsmallmatrix}&1&5&9&10&5&16&10&5&9&5&1\\
  \midrule
   \Sp(4,\Z)&\begin{bsmallmatrix}*&*&*&*\\ *&*&*&*\\ *&*&*&*\\ *&*&*&*\end{bsmallmatrix}&1&0&0&0&0&0&0&0&0&0&0\\
  \midrule
   K(4)&\begin{bsmallmatrix}*&&&*\\&*&*&\\&*&*&\\ *&&&*\end{bsmallmatrix}&1&1&1&0&1&0&0&0&0&0&0\\
  \midrule
   \Gamma_0(2)&\begin{bsmallmatrix}*&*&*&*\\ *&*&*&*\\&&*&*\\&&*&*\end{bsmallmatrix}&1&0&1&0&0&0&0&1&0&0&0\\
  \midrule
   \Gamma_0(4)&\begin{bsmallmatrix}*&*\\ *&*\\&&*&*\\&&*&*\end{bsmallmatrix}&1&0&3&1&0&2&3&3&0&1&0\\
  \midrule
   \Gamma_0^*(4)&\text{index $2$ in }\begin{bsmallmatrix}*&*\\ *&*\\&&*&*\\&&*&*\end{bsmallmatrix}&1&1&3&4&3&4&4&3&3&1&1\\
  \midrule
   \Gamma_0'(2)&\begin{bsmallmatrix}*&*&*&*\\&*&*&*\\&*&*&*\\&&&*\end{bsmallmatrix}&1&1&1&0&0&0&0&0&0&0&0\\
  \midrule
   M(4)&\begin{bsmallmatrix}*\\&*&*&\\&*&*&\\ *&&&*\end{bsmallmatrix}&1&2&2&1&1&1&0&0&0&0&0\\
  \midrule
   B(2)&\begin{bsmallmatrix}*&*&*&*\\&*&*&*\\&&*&*\\&&&*\end{bsmallmatrix}&1&1&2&0&0&1&0&1&0&0&0\\
  \midrule
   \multicolumn{2}{c}{\text{generic}}&&&&\bullet&&\bullet&\bullet&&\bullet\\
  \midrule
   \multicolumn{2}{c}{\text{cuspidal}}&&&&&&&&&\bullet&&\bullet\\
  \bottomrule
 \end{array}
\end{equation*}
\end{table}
\subsection{Parahoric restriction}\label{parahoricrestrictionsec}
In this section we consider irreducible, admissible representations of $\GSp(4,\Q_2)$ with trivial central character that have non-zero fixed vectors under the principal congruence subgroup $\Gamma(2)$. Table~\ref{parahoricrestrictiontable} below contains a complete list of such representations. Their central characters are necessarily unramified, so after a twist we may assume that the central character is trivial. All characters in the representations appearing in Table~\ref{parahoricrestrictiontable} are assumed to be unramified.

\begin{table}
		\caption{Hyperspecial parahoric restriction for $\GSp(4,\Q_2)$. All characters $\chi,\chi_1,\chi_2,\sigma,\xi$ are assumed to be unramified, and the supercuspidal representation $\pi$ of $\GL(2,\Q_2)$ has depth $0$.}\vspace{-4.1ex}\label{parahoricrestrictiontable}
		$$\renewcommand{\arraystretch}{1.1}
		\renewcommand{\arraycolsep}{0.15cm}
		\begin{array}{cccccccccccc}
			\toprule
			&&\text{representation}&a&\varepsilon&\text{temp}&\text{para}&\text{parahoric restriction}&\text{\bf{(G)}}&\text{\bf{(P)}}\\
			\toprule
			{\rm I}&&\chi_1\times\chi_2\rtimes\sigma\quad\text{(irred.)}&0&+&\bullet&\bullet&\scriptstyle [6]+[5,1]+2[4,2]+[3,2,1]+[2,2,2]&\bullet&\\
			\midrule
			{\rm II}&{\rm a}&\chi\St_{\GL(2)}\rtimes\sigma&1&\pm&\bullet&\bullet&[5,1]+[4,2]+[3,2,1]&\bullet&\\
			\cmidrule{2-10}
			&{\rm b}&\chi\triv_{\GL(2)}\rtimes\sigma&0&+&&\bullet&[6]+[4,2]+[2,2,2]&&\bullet\\
			\midrule
			{\rm III}&{\rm a}&\chi\rtimes\sigma\St_{\GSp(2)}&2&+&\bullet&\bullet&[4,2]+[3,2,1]+[2,2,2]&\bullet&\\
			\cmidrule{2-10}
			&{\rm b}&\chi\rtimes\sigma\triv_{\GSp(2)}&0&+&&\bullet&[6]+[5,1]+[4,2]&&\\
			\midrule
			{\rm IV}&{\rm a}&\sigma\St_{\GSp(4)}&3&\pm&\bullet&\bullet&[3,2,1]&\bullet&\\
			\cmidrule{2-10}
			&{\rm b}&L(\nu^2,\nu^{-1}\sigma\St_{\GSp(2)})&2&+&&\bullet&[4,2]+[2,2,2]&\\
			\cmidrule{2-10}
			&{\rm c}&L(\nu^{3/2}\St_{\GL(2)},\nu^{-3/2}\sigma)&1&\pm&&\bullet&[4,2]+[5,1]&\\
			\cmidrule{2-10}
			&{\rm d}&\sigma\triv_{\GSp(4)}&0&+&&\bullet&[6]&\\
			\midrule
			{\rm V}&{\rm a}&\delta([\xi,\nu\xi],\nu^{-1/2}\sigma)&2&-&\bullet&\bullet&[5,1]+[3,2,1]&\bullet&\\
			\cmidrule{2-10}
			&{\rm b}&L(\nu^{1/2}\xi\St_{\GL(2)},\nu^{-1/2}\sigma)&1&\pm&&\bullet&[4,2]&&\bullet\\
			\cmidrule{2-10}
			&{\rm c}&L(\nu^{1/2}\xi\St_{\GL(2)},\xi\nu^{-1/2}\sigma)&1&\pm&&\bullet&[4,2]&&\bullet\\
			\cmidrule{2-10}
			&{\rm d}&L(\nu\xi,\xi\rtimes\nu^{-1/2}\sigma)&0&+&&\bullet&[6]+[2,2,2]&\\
			\midrule
			{\rm VI}&{\rm a}&\tau(S,\nu^{-1/2}\sigma)&2&+&\bullet&\bullet&[4,2]+[3,2,1]&\bullet&\\
			\cmidrule{2-10}
			&{\rm b}&\tau(T,\nu^{-1/2}\sigma)&2&+&\bullet&&[2,2,2]&\bullet&\bullet\\
			\cmidrule{2-10}
			&{\rm c}&L(\nu^{1/2}\St_{\GL(2)},\nu^{-1/2}\sigma)&1&\pm&&\bullet&[5,1]&&\bullet\\
			\cmidrule{2-10}
			&{\rm d}&L(\nu,1_{F^\times}\rtimes\nu^{-1/2}\sigma)&0&+&&\bullet&[6]+[4,2]&\\
			\toprule
			{\rm VII}&&\chi\rtimes\pi&4&+&\bullet&\bullet&[3,1,1,1]+[2,1,1,1,1]&\bullet&\\
			\midrule
			{\rm VIII}&{\rm a}&\tau(S,\pi)&4&+&\bullet&\bullet&[3,1,1,1]&\bullet&\\
			\cmidrule{2-10}
			&{\rm b}&\tau(T,\pi)&4&+&\bullet&&[2,1,1,1,1]&\bullet&\\
			\midrule
			{\rm IX}&{\rm a}&\delta(\nu\xi,\nu^{-1/2}\pi)&4&+&\bullet&\bullet&[3,1,1,1]&\bullet&\\
			\cmidrule{2-10}
			&{\rm b}&L(\nu\xi,\nu^{-1/2}\pi)&4&+&&&[2,1,1,1,1]&\\
			\toprule
			{\rm X}&&\pi\rtimes\sigma&2&-&\bullet&\bullet&[4,1,1]+[3,3]&\bullet&\\
			\midrule
			{\rm XI}&{\rm a}&\delta(\nu^{1/2}\pi,\nu^{-1/2}\sigma)&3&\pm&\bullet&\bullet&[4,1,1]&\bullet&\\
			\cmidrule{2-10}
			&{\rm b}&L(\nu^{1/2}\pi,\nu^{-1/2}\sigma)&2&-&&\bullet&[3,3]&&\bullet\\
			\toprule
			{\rm Va^*}&&\delta^*([\xi,\nu\xi],\nu^{-1/2}\sigma)&2&-&\bullet&&[1,1,1,1,1,1]&\bullet&\bullet\\
			\midrule
			\sc(16)&&&4&-&\bullet&\bullet&[2,2,1,1]&\bullet&\\
			\bottomrule
		\end{array}
		$$
	\end{table}

Let $(\pi, V)$ be an irreducible, admissible representation of $\GSp(4,\Q_2)$. The hyperspecial maximal compact subgroup $K=\GSp(4,\Z_2)$ of $\GSp(4,\Q_2)$ normalizes $\Gamma(\p)$. Hence $K$ acts on the space $V^{\Gamma(\p)}$ of $\Gamma(\p)$-fixed vectors. The resulting representation of $K/\Gamma(\p)\cong \Sp(4,\F_2)$ is called the hyperspecial parahoric restriction of $\pi$ and denoted by $r_K(\pi)$. It has been calculated for all $\pi$ in \cite{Rosner2016,Rosner2018}.

Table~\ref{parahoricrestrictiontable} below contains a list of all irreducible, admissible representations of $\PGSp(4,\Q_2)$ for which $r_K(\pi)\neq0$, using notations as in \cite{SallyTadic1993, RobertsSchmidt2007}. 
Since hyperspecial parahoric restriction commutes with induction by \cite[Thm.~2.19]{Rosner2016}, all the parameters in Table~\ref{parahoricrestrictiontable} must have non-zero parahoric restriction on $\GL(1)$ or $\GL(2)$. This means the characters $\chi,\chi_1,\chi_2,\sigma,\xi$ of $\Q_2^\times$ are assumed to be unramified, and the supercuspidal representation $\pi$ in types VII - XIb is an unramified twist of the unique supercuspidal representation $\pi$ of $\PGL(2,\Q_2)$ of conductor exponent $2$. The ``$a$'' column shows the (exponent of the) conductor of the representation. The ``$\varepsilon$'' column shows the possibilities for the value of the $\varepsilon$-factor at $1/2$. The ``temp'' column indicates the tempered representations, under the assumption that the inducing data is unitary. The ``para'' column indicates the representations that have non-zero paramodular vectors (in which case the minimal paramodular level coincides with the conductor). The hyperspecial parahoric restriction for non-supercuspidal representations is given in \cite[Table~3]{Rosner2018} and \cite[Table~3.1]{Rosner2016}; see \cite[p.~103]{Rosner2016} for the translation of Enomoto's notation to standard $S_6$ notation. 

There are two supercuspidal representations in Table~\ref{parahoricrestrictiontable}, the non-generic $\delta^*([\xi,\nu\xi],\nu^{-1/2}\sigma)$ of type $\rm{Va}^*$ and the generic $\sc(16)$. The representation $\delta^*([\xi,\nu\xi],\nu^{-1/2}\sigma)$ is invariant under twisting by $\xi$, so there is only one representation of type  $\rm{Va}^*$. It shares an $L$-packet with the unique representation of type $\rm{Va}$; both have $L$-parameter $\rm{st}_2 \oplus \xi \rm{st}_2$, where $\rm{st}_2$ is the $L$-parameter of the Steinberg representation $\rm{St}_{\GL(2)}$ of $\GL(2,\Q_2)$. So, the parahoric restriction information for $\rm{Va}^*$ comes from \cite[Tables~4.2~and~5.2]{Rosner2016}. By Frobenius reciprocity and the proposition in \cite[Sect. 1.4]{Morris1996}, it follows that
\begin{equation}\label{Va*czk}
 \delta^*([\xi,\nu\xi],\nu^{-1/2}\sigma)=\text{c-}\Ind^{G}_{ZK}([1,1,1,1,1,1]),
\end{equation}
where we inflate $[1,1,1,1,1,1]$, the sign character of  $K/\Gamma(2)\cong \Sp(4,\F_2)$ to $K$ and then extend it to $ZK$ by having $Z$ act trivially. By \cite[Table~4.2]{Rosner2016}, there are no non-generic supercuspidals of $\PGSp(4,\Q_2)$ with non-zero hyperspecial parahoric restriction besides $\rm{Va}^*$.

By \cite{PoorSchmidtYuen2019} and \cite[Prop.~2.16]{Rosner2016}, there are no generic supercuspidals of $\PGSp(4,\Q_2)$ with non-zero hyperspecial parahoric restriction besides $\sc(16)$. In this case we have
\begin{equation}\label{sc16czk}
 \sc(16)=\text{c-}\Ind^{G}_{ZK}([2,2,1,1]).
\end{equation}
The parahoric restriction for $\sc(16)$ follows from \cite[Lemma~2.18]{Rosner2016}.
\subsection{Local fixed vectors}\label{localfixedvectorssec}
Table~\ref{localdimtable} lists the dimensions of the space of fixed vectors under various congruence subgroups for the same class of representations as in Table~\ref{parahoricrestrictiontable}. These are the irreducible, admissible representations $\pi$ of $\GSp(4,\Q_2)$ for which the hyperspecial parahoric restriction $r_K(\pi)$ is non-zero, i.e., which have non-zero vectors fixed under the principal congruence subgroup $\Gamma(\p)$.

\begin{table}
 \caption{The dimensions of the spaces of fixed vectors under various congruence subgroups of the irreducible, admissible representations of $\GSp(4,\Q_2)$ with non-zero hyperspecial parahoric restriction. (See Table~\ref{parahoricrestrictiontable} for the precise notation for these representations.)\vspace{2ex}} \label{localdimtable}
 \centering
 \scalebox{0.75}{$
 \renewcommand{\arraystretch}{1.5}\setlength{\arraycolsep}{0.4cm}
 \begin{array}{ccccccccccccccccc}
  \toprule
				\Omega &&\Gamma(\p)&K&\mathrm{K}(\p)&\K{2}&\Gamma_0(\p)&\Gamma_0(\p^2)&\Gamma_0^*(\p^2)&\Gamma_0'(\p)&\Gamma_0'(\p^2)&M(\p^2)&B(\p)\\
  \toprule
				\rm{I}&&45&1&2&4&4&12&15&4&11&8&8\\\midrule
				\rm{II}&\rm{a}&30&0&1&2&1&5&8&2&7&5&4\\\cmidrule{2-13}
				&\rm{b}&15&1&1&2&3&7&7&2&4&3&4\\\midrule
				\rm{III}&\rm{a}&30&0&0&1&2&8&10&1&5&3&4\\\cmidrule{2-13}
				&\rm{b}&15&1&2&3&2&4&5&3&6&5&4\\\midrule
				\rm{IV}&\rm{a}&16&0&0&0&0&2&4&0&2&1&1\\\cmidrule{2-13}
				&\rm{b}&14&0&0&1&2&6&6&1&3&2&3\\\cmidrule{2-13}
				&\rm{c}&14&0&1&2&1&3&4&2&5&4&3\\\cmidrule{2-13}
				&\rm{d}&1&1&1&1&1&1&1&1&1&1&1\\\midrule
				\rm{V}&\rm{a}&21&0&0&1&0&2&5&1&5&3&2\\\cmidrule{2-13}
				&\rm{b}&9&0&1&1&1&3&3&1&2&2&2\\\cmidrule{2-13}
				&\rm{c}&9&0&1&1&1&3&3&1&2&2&2\\\cmidrule{2-13}
				&\rm{d}&6&1&0&1&2&4&4&1&2&1&2\\\midrule
				\rm{VI}&\rm{a}&25&0&0&1&1&5&7&1&5&3&3\\\cmidrule{2-13}
				&\rm{b}&5&0&0&0&1&3&3&0&0&0&1\\\cmidrule{2-13}
				&\rm{c}&5&0&1&1&0&0&1&1&2&2&1\\\cmidrule{2-13}
				&\rm{d}&10&1&1&2&2&4&4&2&4&3&3\\\midrule
				\rm{VII}&&15&0&0&0&0&4&5&0&2&0&0\\\midrule
				\rm{VIII}&\rm{a}&10&0&0&0&0&3&4&0&2&0&0\\\cmidrule{2-13}
				&\rm{b}&5&0&0&0&0&1&1&0&0&0&0\\\midrule
				\rm{IX}&\rm{a}&10&0&0&0&0&3&4&0&1&0&0\\\cmidrule{2-13}
				&\rm{b}&5&0&0&0&0&1&1&0&1&0&0\\\midrule
				\rm{X}&&15&0&0&1&0&1&7&0&3&2&0\\\midrule
				\rm{XI}&\rm{a}&10&0&0&0&0&1&4&0&2&1&0\\\cmidrule{2-13}
				&\rm{b}&5&0&0&1&0&0&3&0&1&1&0\\\midrule
				\rm{Va}^*&&1&0&0&0&0&0&1&0&0&0&0\\\midrule
				\sc(16)&&9&0&0&0&0&0&3&0&1&0&0\\
  \bottomrule
 \end{array}
 $}
\end{table}

\begin{theorem}\label{localdimtheorem}
 Let $(\pi,V)$ be an irreducible, admissible representation of $\GSp(4,\Q_2)$. Let $H$ be one of the congruence subgroups listed in the first row of Table~\ref{localdimtable}.
 \begin{enumerate}
  \item If $r_K(\pi)\neq0$, so that $\pi$ occurs among the representations in Tables \ref{parahoricrestrictiontable}, then $\dim V^H$ is given as in Table~\ref{localdimtable}.
  \item If $\dim V^H\neq0$, then $r_K(\pi)\neq0$, so that $\pi$ occurs among the representations in Tables \ref{parahoricrestrictiontable}. 
 \end{enumerate}
\end{theorem}
\begin{proof}
i) For $H=\mathrm{K}(\p)$, see \cite[Sect.~A.8]{RobertsSchmidt2007}. For $H=\Gamma_0'(\p^2)$, see \cite[Table~1]{Yi2021}. For every other $H$, there exists a conjugate subgroup $\tilde H$ such that $\Gamma(\p)\subset\tilde H\subset K$. If $r_K(\pi)=\rho_1\oplus\ldots\oplus\rho_m$ with $S_6$-types $(\rho_i,U_i)$, then
\begin{equation}\label{localdimtheoremeq1}
 \dim V^H=\dim  r_K(\pi)^{\tilde H}=\sum_{i=1}^m\dim U_i^{\bar H},
\end{equation}
where $\bar H$ is the image of $\tilde H$ in $K/\Gamma(\p)\cong\Sp(4,\mathbb{F}_2)$. The dimensions $U^{\bar H}$ are listed in Table~\ref{S6typesfixedtable}, for each $S_6$-type $(\rho,U)$. We thus get the desired dimensions from the $r_K(\pi)$ listed in Table~\ref{parahoricrestrictiontable}.

ii) If $H\neq \Gamma_0'(\p^2)$, then a conjugate of $H$ contains $\Gamma(\p)$, so that $r_K(\pi)=V^{\Gamma(\p)}\supset V^H\neq0$. If $H=\Gamma_0'(\p^2)$, then \cite[Lemma~4]{Yi2021} shows that $V^{\Gamma(\p)}\neq0$.
\end{proof}

We remark that for most of the congruence subgroups the dimensions in Table~\ref{localdimtable} appear elsewhere in the literature. For all the subgroups containing $B(\p)$, see \cite{Schmidt2005}. For the paramodular groups, see \cite{RobertsSchmidt2007}. For $M(\p^2)$ and $\Gamma_0'(\p^2)$, see \cite{Yi2021}.

\section{Global dimensions and codimensions}\label{dimcodimsec}
The goal of this section is to derive the dimension formulas in Tables~\ref{dimMktable} and~\ref{dimSktable} for all congruence subgroups except $\Gamma_0'(4)$. Most of these formulas are already contained in the literature, but we give a unified approach. First, we derive a general formula in degree $2$ for the codimension $\dim M_k(\Gamma)-\dim S_k(\Gamma)$ based on the Satake compactification and the global $\Phi$ map; see Theorem~\ref{thmcodim}. We thus obtain the codimensions summarized in Table~\ref{codimensionstable} for all congruence subgroups of interest to us.

To obtain the actual dimensions, we note that most of our congruence subgroups $\Gamma$, after an appropriate conjugation, lie between $\Gamma(2)$ and $\Sp(4,\Z)$. One can thus use Igusa's result \cite[Thm.~2]{Igusa1964} to calculate $\dim M_k(\Gamma)$; see Sect.~\ref{Igusasec}. The only subgroup other than $\Gamma_0'(4)$ for which this does not work is $K(2)$, for which the result is already contained in the literature.
\subsection{A general codimension formula}\label{gencodimsec}
In this section we will find a general formula for calculating the codimension of $S_k(\Gamma)$ in $M_k(\Gamma)$ for a congruence subgroup $\Gamma$ of $\Sp(4,\Q)$. A summary of the method for any degree is given in \cite[Sect.~3]{PoorYuen2013}. It is based on \cite{Satake1957-1958a} and the surjectivity of the global $\Phi$ operator proven in \cite{Satake1957-1958}. We specialize to the degree $2$ case, resulting in the formula in Theorem~\ref{thmcodim} below.

We define the symplectic group $\Sp(4)$ with respect to the form $J_1$ given in \eqref{J1defeq}, and use the following parabolic subgroups,
\begin{equation}\label{BPQeq}
		B=\begin{bsmallmatrix}
			*&0&*&*\\
			*&*&*&*\\
			0&0&*&*\\
			0&0&0&*
		\end{bsmallmatrix} \cap \Sp(4),\quad 
		P=\begin{bsmallmatrix}
			*&*&*&*\\
			*&*&*&*\\
			0&0&*&*\\
			0&0&*&*
		\end{bsmallmatrix} \cap \Sp(4),\quad 
		Q=\begin{bsmallmatrix}
			*&0&*&*\\
			*&*&*&*\\
			*&0&*&*\\
			0&0&0&*
		\end{bsmallmatrix} \cap \Sp(4).
\end{equation}
Consider the homomorphisms $\omega\colon Q(\R) \rightarrow \SL(2,\R)$ and $\iota\colon\SL(2,\R) \rightarrow Q(\R)$  given by 
\begin{align}\label{omegaiotadef}
		\omega\colon Q(\R) &\longrightarrow \SL(2,\R),  &  \iota\colon\SL(2,\R) &\longrightarrow Q(\R),\\
		\begin{bsmallmatrix}
			a&0&b&*\\
			*&*&*&*\\
			c&0&d&*\\
			0&0&0&*
		\end{bsmallmatrix}&\longmapsto \begin{bsmallmatrix}
			a&b\\
			c&d\\
		\end{bsmallmatrix},& \begin{bsmallmatrix}
			a&b\\
			c&d\\
		\end{bsmallmatrix}&\longmapsto \begin{bsmallmatrix}
			a&0&b&0\\
			0&1&0&0\\
			c&0&d&0\\
			0&0&0&1
		\end{bsmallmatrix}.\nonumber 
\end{align}
Let $\Gamma$ be a congruence subgroup of $\Sp(4,\Q)$. We will describe how the geometry of the Satake compactification $\mathcal{S}(\Gamma\backslash\UHP_2)$ is reflected algebraically via double cosets. 

Let $X$ be a fixed set of representatives for the double cosets $\Gamma\backslash \Sp(4,\Q) /P(\Q)$, and let $Y$ be a fixed set of representatives for the double cosets $\Gamma\backslash \Sp(4,\Q) /Q(\Q)$. (Note that the quotient $\Sp(4,\R)^{\rm pr}/H(\R)^{\rm pr}$ appearing in \cite[p.~451]{PoorYuen2013} simplifies to $\Sp(4,\Q)/H(\Q)$ for any of the subgroups $H$ in \eqref{BPQeq}). Since $\Sp(4,\Q)=\Sp(4,\Z)B(\Q)$ (by taking inverses in Lemma~\ref{SpnQlemma}), we may assume that $X,Y\subset\Sp(4,\Z)$. There is a bijection between $X$ and the zero-dimensional cusps of $\Gamma$. Similarly, there is a bijection between $Y$ and the one-dimensional cusps of $\Gamma$. For $y\in Y$, let $C_y$ be the $1$-cusp corresponding to $y$, and let
\begin{equation}\label{Gammaydefeq}
 \Gamma_{y}=\omega\left(y^{-1} \Gamma y \cap Q(\Q)\right),
\end{equation}
which is a congruence subgroup of $\SL(2,\Q)$. Let $R_y$ be a fixed set of representatives for the double cosets $\Gamma_y\bs\SL(2, \Q)/B_1(\Q)$, where $B_1$ is the upper triangular subgroup of $\SL(2)$. As is well-known, there is a bijection between $R_y$ and the set of cusps of $\Gamma_y$, which are points in the Satake compactification $\mathcal{S}(\Gamma_y\backslash\UHP_1)$. There is an embedding $\Gamma_y\backslash\UHP_1\to\Gamma\backslash\UHP_2$, which extends to a continuous map $\mathcal{S}(\Gamma_y\backslash\UHP_1)\to\mathcal{S}(\Gamma\backslash\UHP_2)$. Let $C_{y,\rho}$ be the image of the cusp corresponding to $\rho\in R_y$ under this map. It is a $0$-cusp of $\mathcal{S}(\Gamma\backslash\UHP_2)$ lying on the $1$-cusp $C_y$. 
The double coset corresponding to $C_{y,\rho}$  is $\Gamma y\iota(\rho)P(\Q)$. We see:
 \begin{itemize}
  \item If $\Gamma y_1\iota(\rho_1)P(\Q)=\Gamma y_2\iota(\rho_2)P(\Q)$ for two distinct $y_1,y_2\in Y$ and some $\rho_1\in R_{y_1}$, $\rho_2\in R_{y_2}$, then it means that $C_{y_1}$ and $C_{y_2}$ intersect at $C_{y_1,\rho_1}=C_{y_2,\rho_2}$.
  \item If  $\Gamma y\iota(\rho_1)P(\Q)=\Gamma y\iota(\rho_2)P(\Q)$ for $y\in Y$ and distinct $\rho_1,\rho_2\in R_y$, then it means that $C_y$ has a self-intersection at $C_{y,\rho_1}=C_{y,\rho_2}$.
 \end{itemize}
 In this way we find the cusp structure diagram for $\Gamma$. It consists of $|Y|$ curves representing the $1$-cusps $C_y$, and $|X|$ points representing the $0$-cusps $C_{y,\rho}$ for $\rho\in R_y$, with $C_{y,\rho}$ lying on $C_y$ indicating the intersections and self-intersections.
 
For $f\in M_k(\Gamma)$, the Siegel $\Phi$-operator produces a function $\Phi f$ on $\UHP_1$ defined by
\begin{equation}\label{Phimap}
 (\Phi f)(\tau)=\lim_{\lambda\rightarrow\infty}f(\mat{\tau}{}{}{i\lambda}),\quad \tau\in \mathcal{H}_{1}.
\end{equation}
It follows from the Fourier expansion of $f$ that in fact
\begin{equation}\label{Phimap2}
 (\Phi f)(\tau)=\lim_{\lambda\rightarrow\infty}f(\mat{\tau}{z}{z}{i\lambda})\quad\text{for any }z\in\C.
\end{equation}
We also define a $\Phi$-operator on modular forms on $\UHP_1$. If $f$ is such a modular form, then $\Phi f$ is simply the number $\lim_{\lambda\rightarrow\infty}f(i\lambda)$. 
\begin{lemma}\label{Phicharacterlemma}
 Let
 \begin{equation}\label{Phicharacterlemmaeq1}
  u=\begin{bsmallmatrix}
			*&0&*&*\\
			*&*&*&*\\
			*&0&*&*\\
			0&0&0&r
  \end{bsmallmatrix}\in Q(\Q)
 \end{equation}
 and $f$ be a modular form of weight $k$ on $\UHP_2$ with respect to some congruence subgroup. Then
 \begin{equation}\label{Phicharacterlemmaeq2}
  \Phi(f|u)=r^{-k}(\Phi f)|\omega(u).
 \end{equation}
\end{lemma}
\begin{proof}
See the calculation in \cite[p.~2464]{PoorShurmanYuen2020} to obtain \eqref{Phicharacterlemmaeq2}.
\end{proof}

\begin{lemma}\label{Phifylemma}
 Let $\Gamma$ be a congruence subgroup of $\Sp(4,\Q)$ and $f\in M_k(\Gamma)$. Let $y\in Y$.
 \begin{enumerate}\label{Remark for dim with char} 
  \item  If $k\ge 1$ is even, then $\Phi(f|y) \in M_k^{(1)}(\Gamma_y)$, where $\Gamma_y$ is the group defined in \eqref{Gammaydefeq}.
  \item  If $k\ge 1$ is odd and $\begin{bsmallmatrix}
				1&&&\\
				&-1&&\\
				&&1&\\
				&&&-1
			\end{bsmallmatrix} \in y^{-1}\Gamma y$, then $\Phi(f|y)=0$. 
 \end{enumerate}
\end{lemma}
\begin{proof}
Let
\begin{equation}\label{Phifylemmaeq1}
  u=\begin{bsmallmatrix}
			*&0&*&*\\
			*&*&*&*\\
			*&0&*&*\\
			0&0&0&r
  \end{bsmallmatrix}\in y^{-1}\Gamma y\cap Q(\Q).
\end{equation}
We claim that $r\in\{\pm1\}$. Indeed, the map $y^{-1}\Gamma y\cap Q(\Q_p)\rightarrow \Q_p^{\times}$ which sends any matrix to its $(4,4)$-coefficient is a continuous group homomorphism. Since $y^{-1}\Gamma y\cap Q(\Q_p)$ lies in a compact subset of $Q(\Q_p)$, the image of this homomorphism lies in $\Z_p^\times$. This is true for all $p$, and hence $r\in\{\pm1\}$.

Applying Lemma~\ref{Phicharacterlemma} to $g:=f|y$ instead of $f$ and $u$ as in \eqref{Phifylemmaeq1}, we see that
\begin{equation}\label{Phifylemmaeq2}
  \Phi g=r^{-k}(\Phi g)|\omega(u).
\end{equation}
Now both i) and ii) follow easily.
\end{proof}

\begin{theorem}\label{thmcodim}
 Let $\Gamma$ be a congruence subgroup of $\Sp(4,\Q)$. Let $X$, $Y$ and $\Gamma_{y}$ be as defined above. Then, for even $k\geq 6$, we have
 \begin{equation}\label{codim}
  \dim M_k(\Gamma)-\dim S_k(\Gamma)=|X|+\sum_{y\in Y}\dim  S_k^{(1)}(\Gamma_{y}).
 \end{equation}	
\end{theorem}
\begin{proof}
Observing Lemma~\ref{Phifylemma}~i), we define
  \begin{equation}\label{Phitlidemap}
   \widetilde{\Phi}\colon M_k(\Gamma)\longrightarrow\bigoplus_{y\in Y}M_k^{(1)}(\Gamma_y),\quad f\longmapsto(f_y)_{y\in Y},\quad\text{where }f_y=\Phi(f|y).
  \end{equation}
One may think of $f_y$ as the restriction of $f$ to $C_y$. Evidently, $\ker(\tilde\Phi)=S_k(\Gamma)$, so that we have an exact sequence
  \begin{equation}\label{thmcodimeq11}
   0\longrightarrow S_k(\Gamma)\longrightarrow M_k(\Gamma)\longrightarrow\mathrm{Im}(\tilde\Phi)\longrightarrow0.
  \end{equation}
Hence our desired codimension equals $\dim\mathrm{Im}(\tilde\Phi)$. To understand $\mathrm{Im}(\tilde\Phi)$, note that the $f_y$ satisfy the following compatibility condition: For all $y_1,y_2\in Y$, $\rho_1\in R_{y_1}$ and $\rho_2\in R_{y_2}$,
\begin{equation}\label{compatibility condition}
 \Gamma y_1\iota(\rho_1) P(\Q)=\Gamma y_2\iota(\rho_2) P(\Q) \quad \Longrightarrow \quad	\Phi(f_{y_1}|\rho_1)=\Phi(f_{y_2}|\rho_2).
\end{equation}
(This amounts to saying that $f_{y_1}$ and $f_{y_2}$ agree on the intersection points of the $1$-cusps $C_{y_1}$ and $C_{y_2}$; see \cite[(1)]{PoorYuen2013}.) Satake \cite{Satake1957-1958} proved that $\mathrm{Im}(\tilde\Phi)$ is characterized by this compatibility condition. In particular,
\begin{equation}\label{thmcodimeq10}
 \bigoplus_{y\in Y}S_k^{(1)}(\Gamma_y)\subset\mathrm{Im}(\tilde\Phi).
\end{equation}
To further study $\mathrm{Im}(\tilde\Phi)$, we choose for every $x\in X$ a $y_x\in Y$ and a $\rho_x\in R_{y_x}$ such that the $0$-cusp represented by $x$ equals $C_{y_x,\rho_x}$. In terms of double cosets, this means
\begin{equation}\label{definitionfyxrhox}
 \Gamma xP(\Q)=\Gamma y_x\iota(\rho_x)P(\Q).
\end{equation}
Then we define the map
 \begin{equation}\label{alphamap}
  \theta\colon\mathrm{Im}(\tilde\Phi)\longrightarrow\C^{|X|},\quad(f_y)_{y\in Y} \longmapsto  \left(\Phi(f_{y_x}|\rho_x) \right)_{x\in X}. 
 \end{equation}
The compatibility condition \eqref{compatibility condition} assures that $\theta$ is independent of the choices of $y_x$ and $\rho_x$.

It is clear that $\bigoplus_{y\in Y}S_k^{(1)}(\Gamma_{y})\subseteq\ker\theta$ by definition of $\theta$. To prove the reverse inclusion, suppose $(f_y)_{y\in Y} \in \mathrm{Im}(\tilde\Phi)$ lies in the kernel of $\theta$, i.e., $\Phi(f_{y_x}|\rho_x)=0$ for all $x\in X$. We want to show that $\Phi(f_y|\rho)=0$ for all $y\in Y$ and $\rho\in R_y$. Let $a\in X$ be such that $\Gamma y\iota(\rho)P(\Q)=\Gamma aP(\Q)$. Since also $\Gamma aP(\Q)=\Gamma y_a\iota(\rho_a)P(\Q)$, we have
\begin{equation*}
 \Phi(f_{y}|\rho)=\Phi(f_{y_a}|\rho_a)=0
\end{equation*}
by the compatibility condition \eqref{compatibility condition}. This proves $\ker \theta=\bigoplus_{y\in Y}S_k^{(1)}(\Gamma_{y})$.

Next we show that $\theta$ is surjective. Let $x\in X$. It follows from \cite[Thm.~3.5.1]{DiamondShurman2005} that we can find an $f_y\in M_k^{(1)}(\Gamma_y)$ such that, for all $\rho\in R_y$,
\begin{equation*}
 \Phi(f_y|\rho)=\begin{cases}
  1& \text{if } \Gamma xP(\Q)=\Gamma y\iota(\rho)P(\Q),\\
  0&\text{otherwise}.
 \end{cases} 
\end{equation*}
The family of $f_y$ thus defined satisfies the compatibility condition \eqref{compatibility condition}, so that $(f_y)_{y\in Y}$ lies in the image of $\tilde\Phi$. Hence we constructed an element of $\mathrm{Im}(\tilde\Phi)$ which does not vanish at the $0$-cusp corresponding to $x$, but vanishes at all other $0$-cusps. It follows that $\theta$ is surjective. 

We proved that there is an exact sequence
\begin{equation}\label{thmcodimeq12}
 0\longrightarrow \bigoplus_{y\in Y}S_k^{(1)}(\Gamma_{y})\longrightarrow \mathrm{Im}(\tilde\Phi)\longrightarrow\C^{|X|}\longrightarrow0.
\end{equation}
Our assertion now follows from \eqref{thmcodimeq11} and \eqref{thmcodimeq12}.
\end{proof}
\subsection{Codimension formulas for some congruence subgroups}
In this section we determine the codimension $\dim M_k(\Gamma)-\dim S_k(\Gamma)$ for the congruence subgroups $\Gamma$ listed below in Theorem~\ref{oddweightcodimtheorem}. One of them is a group $\Gamma_0^*(4)$ defined as follows. We consider the character of the group $\Gamma_0(4)$ obtained as the composition
\begin{equation}
 \Gamma_0(4)\longrightarrow\GL(2,\Z)\longrightarrow\SL(2,\Z/2\Z)\stackrel{\sim}{\longrightarrow} S_3\longrightarrow\{\pm1\},
\end{equation}
where the first map is given by $\mat{A}{B}{C}{D}\mapsto D$, the second map is reduction modulo $2$, the third map is any isomorphism, and the last map is the sign character of the symmetric group $S_3$. Let $\Gamma_0^*(4)$ be the kernel of this character. Explicitly,
\begin{equation}\label{Gamma0astdef}
 \Gamma_0^{\ast}(4)=\left\lbrace g=\mat{A}{B}{C}{D}\in\Gamma_0(4):D\equiv\mat{1}{0}{0}{1},\mat{0}{1}{1}{1},\mat{1}{1}{1}{0}\bmod{2} \right\rbrace.
\end{equation}
Evidently, $\Gamma_0(4)=\Gamma_0^{\ast}(4)\sqcup\Gamma_0^{\ast}(4)s_1$, where $s_1$ is defined in \eqref{s1s2defeq}.

For odd weights we have the following result.
\begin{theorem}\label{oddweightcodimtheorem}
 Suppose that $k\geq1$ is odd, and that $\Gamma$ is conjugate to one of the congruence subgroups in \eqref{congruencesubgroupseq}. Then $M_k(\Gamma)=S_k(\Gamma)$.
\end{theorem}
\begin{proof}
Let $Y\subset\Sp(4,\Z)$ be a fixed set of representatives for the double cosets $\Gamma\backslash\Sp(4,\Q)/Q(\Q)$. If we can verify the condition in Lemma~\ref{Phifylemma} ii) for all $y\in Y$, then $M_k(\Gamma)=S_k(\Gamma)$ will follow from \eqref{Phitlidemap} and \eqref{thmcodimeq11}; note that \eqref{Phitlidemap} and \eqref{thmcodimeq11} hold for both even and odd $k$.

If $\Gamma'$ is conjugate to $\Gamma$ by an element of $\Sp(4,\Q)$, then $M_k(\Gamma')=S_k(\Gamma')$ if and only if $M_k(\Gamma)=S_k(\Gamma)$. Hence we need only consider the groups in \eqref{congruencesubgroupseq}.

The condition in Lemma~\ref{Phifylemma} ii) is satisfied for the normal subgroup $\Gamma(2)$ of $\Sp(4,\Z)$, and then also for any subgroup containing $\Gamma(2)$. Up to conjugation, this covers all groups in \eqref{congruencesubgroupseq} except $\Gamma_0'(4)$. For $\Gamma_0'(4)$ one can verify the condition directly using the representatives $y$ given in Table~\ref{doublecosetsQtable}.
\end{proof}

We turn to even weights, considering the case $k\geq6$. Recall that for the codimension formula in Theorem~\ref{thmcodim} we need $|X|$, which is the cardinality of the double coset space $\Gamma\backslash\Sp(4,\Q)/P(\Q)$, and we need to know the groups $\Gamma_y$ defined in \eqref{Gammaydefeq}, where $y$ runs through a system of representatives for the double coset space $\Gamma\backslash\Sp(4,\Q)/Q(\Q)$.

For $\Gamma$ equal to the principal congruence subgroup $\Gamma(2)$, it is well known that both double coset spaces have 15 elements, and that each group $\Gamma_y$ equals $\Gamma^{(1)}(2)$. A quick derivation uses the fact that $\Sp(4,\mathbb{F}_2)\cong S_6$ has 720 elements (see Sect.~\ref{Igusasec}). Note $\Sp(4,\Q)/P(\Q)\cong\Sp(4,\Z)/P(\Z)$ and $\Sp(4,\Q)/Q(\Q)\cong\Sp(4,\Z)/Q(\Z)$. Hence
\begin{equation}\label{Gamma2cuspseq1}
 \Gamma(2)\backslash\Sp(4,\Q)/P(\Q)\cong\Sp(4,\mathbb{F}_2)/P(\mathbb{F}_2)\quad\text{and}\quad
 \Gamma(2)\backslash\Sp(4,\Q)/Q(\Q)\cong\Sp(4,\mathbb{F}_2)/Q(\mathbb{F}_2).
\end{equation}
Since $P(\mathbb{F}_2)$ and $Q(\mathbb{F}_2)$ both have 48 elements, it follows that both double coset spaces have cardinality 15. Moreover, since $\Gamma(2)$ is normal in $\Sp(4,\Z)$, each $\Gamma_y$ equals $\omega(\Gamma\cap Q(\Q))=\SL(2,\Z)\cap\mat{\Z}{2\Z}{2\Z}{\Z}$, which is conjugate (by an element of $\SL(2,\Q)$) to $\Gamma_0^{(1)}(4)$. From Theorem~\ref{thmcodim} we thus get $\dim M_k(\Gamma(2))-\dim S_k(\Gamma(2))=15+15\dim S_k(\Gamma_0^{(1)}(4))$ for even $k\geq6$.

For the congruence subgroups in \eqref{congruencesubgroupseq} other than $\Gamma(2)$, the last column of Table~\ref{doublecosetsPtable} shows the cardinality of $\Gamma\backslash\Sp(4,\Q)/P(\Q)$. The table also indicates representatives for this double coset space, using the following notations:
\begin{align}\label{xidef}
 x_1&=\mathbf{I}_4\qquad x_2=s_2=\begin{bsmallmatrix}&&1\\&1\\-1\\&&&1\end{bsmallmatrix}\qquad x_3=s_1s_2=\begin{bsmallmatrix}&1\\&&1\\&&&1\\-1\end{bsmallmatrix}\qquad x_4=s_2s_1s_2=\begin{bsmallmatrix}&&&1\\&&1\\&-1\\-1\end{bsmallmatrix}\nonumber\\
 x_5&=\begin{bsmallmatrix}1\\&1\\&&1\\&2&&1\end{bsmallmatrix}\qquad x_6=\begin{bsmallmatrix}1\\&1\\2&&1\\&2&&1\end{bsmallmatrix}\qquad x_7=\begin{bsmallmatrix}1\\&1\\&2&1\\2&&&1\end{bsmallmatrix}\qquad x_8=\begin{bsmallmatrix}&1\\&&1\\&2&&1\\-1\end{bsmallmatrix}.
\end{align}
For the group $\Sp(4,\Z)$ the information in Table~\ref{doublecosetsPtable} is trivial, for $K(2)$ and $K(4)$ see \cite[Thm.~1.3]{PoorYuen2013}. Representatives for $\Gamma_0(2)$ follow from
\begin{equation}\label{Gamma02Peq}
 \Gamma_0(2)\backslash\Sp(4,\Q)/P(\Q)\cong P(\mathbb{F}_2)\backslash\Sp(4,\mathbb{F}_2)/P(\mathbb{F}_2)
\end{equation}
and the Bruhat decomposition; similarly for $\Gamma_0'(2)$ and $B(2)$. For $\Gamma_0'(4)$ and $M(4)$ see \cite[Lemma~1, Lemma~2]{Yi2021}. For $\Gamma=\Gamma_0(4)$ see \cite[Prop.~2.6]{Tsushima2003}. It is an exercise to derive the result for $\Gamma_0^*(4)$ from that for $\Gamma_0(4)$, using that $\Gamma_0(4)=\Gamma_0^{\ast}(4)\sqcup\Gamma_0^{\ast}(4)s_1$.

\begin{table}
 \caption{Double coset representatives for $\Gamma\backslash\Sp(4,\Q)/P(\Q)$.}
 \label{doublecosetsPtable}
\begin{equation*}\setlength{\arraycolsep}{0.8ex}
 \begin{array}{cccccccccccc}
  \toprule
   \Gamma&x_1&x_2&x_3&x_4&x_5&x_6&x_7&x_8&\#\\
  \midrule
   \Sp(4,\Z)&\bullet&&&&&&&&1\\
  \midrule
   K(2)&\bullet&&&&&&&&1\\
  \midrule
   K(4)&\bullet&&&&&&\bullet&&2\\
  \midrule
   \Gamma_0(2)&\bullet&\bullet&&\bullet&&&&&3\\
  \midrule
   \Gamma_0(4)&\bullet&&\bullet&\bullet&\bullet&\bullet&\bullet&\bullet&7\\
  \midrule
   \Gamma_0^*(4)&\bullet&&\bullet&\bullet&\bullet&\bullet&\bullet&\bullet&7\\
  \midrule
   \Gamma_0'(2)&\bullet&&\bullet&&&&&&2\\
  \midrule
   \Gamma_0'(4)&\bullet&&\bullet&&\bullet&&\bullet&&4\\
  \midrule
   M(4)&\bullet&&\bullet&&\bullet&&&&3\\
  \midrule
   B(2)&\bullet&\bullet&\bullet&\bullet&&&&&4\\
  \bottomrule
 \end{array}
\end{equation*}
\end{table}

Table~\ref{doublecosetsQtable} gives double coset representatives for $\Gamma\backslash\Sp(4,\Q)/Q(\Q)$. The notation used is
\begin{align}\label{yidef}
 y_1&=\mathbf{I}_4\qquad y_2=s_1=\begin{bsmallmatrix}&1\\1\\&&&1\\&&1\end{bsmallmatrix}\qquad y_3=s_2s_1=\begin{bsmallmatrix}&&&1\\1\\&-1\\&&1\end{bsmallmatrix}\qquad y_4=s_1s_2s_1=\begin{bsmallmatrix}1\\&&&1\\&&1\\&-1\end{bsmallmatrix}\nonumber\\
 y_5&=\begin{bsmallmatrix}1\\&1\\&&1\\&2&&1\end{bsmallmatrix}\qquad y_6=\begin{bsmallmatrix}1&-2\\&1\\&&1\\&&2&1\end{bsmallmatrix}\qquad y_7=\begin{bsmallmatrix}1\\&1\\&2&1\\2&&&1\end{bsmallmatrix}\qquad y_8=\begin{bsmallmatrix}&1\\1\\&2&&1\\&&1\end{bsmallmatrix}\qquad y_9=\begin{bsmallmatrix}&1\\1\\2&&&1\\&2&1\end{bsmallmatrix}.
\end{align}
A non-empty entry in the row for $\Gamma$ and the column for $y_i$ indicates that $y_i$ is to be included in the set $Y$ of representatives for $\Gamma\backslash\Sp(4,\Q)/Q(\Q)$. The entry itself indicates the group $\Gamma_{y_i}$, obtained by intersecting the given set with $\SL(2,\Q)$. For most of the groups the references are the same as given above for Table~\ref{doublecosetsPtable}. For $\Gamma=\Gamma_0(4)$ see \cite[Prop.~2.5]{Tsushima2003}. Again the result for $\Gamma_0^*(4)$ can be derived from that for $\Gamma_0(4)$.

\begin{table}
 \caption{Double coset representatives for $\Gamma\backslash\Sp(4,\Q)/Q(\Q)$. The groups $\Gamma_y$ defined in \eqref{Gammaydefeq} are obtained by intersecting the given sets of $2\times2$ matrices with $\SL(2,\Q)$.}\label{doublecosetsQtable}
\begin{equation*}\setlength{\arraycolsep}{0.8ex}
 \begin{array}{cccccccccccc}
  \toprule
   \Gamma&y_1&y_2&y_3&y_4&y_5&y_6&y_7&y_8&y_9&\#\\
  \midrule
   \Sp(4,\Z)&\mat{\Z}{\Z}{\Z}{\Z}&&&&&&&&&1\\
  \midrule
   K(2)&\mat{\Z}{\Z}{\Z}{\Z}&\mat{\Z}{2^{-1}\Z}{2\Z}{\Z}&&&&&&&&2\\
  \midrule
   K(4)&\mat{\Z}{\Z}{\Z}{\Z}&\mat{\Z}{4^{-1}\Z}{4\Z}{\Z}&&&&\mat{\Z}{\Z}{2\Z}{\Z}&&&&3\\
  \midrule
   \Gamma_0(2)&\mat{\Z}{\Z}{2\Z}{\Z}&&\mat{\Z}{\Z}{2\Z}{\Z}&&&&&&&2\\
  \midrule
   \Gamma_0(4)&\mat{\Z}{\Z}{4\Z}{\Z}&&&\mat{\Z}{4\Z}{\Z}{\Z}&\mat{\Z}{\Z}{4\Z}{\Z}&&\mat{\Z}{\Z}{4\Z}{\Z}&&&4\\
  \midrule
   \Gamma_0^*(4)&\mat{\Z}{\Z}{4\Z}{\Z}&&&\mat{\Z}{4\Z}{\Z}{\Z}&\mat{\Z}{\Z}{4\Z}{\Z}&&\mat{\Z}{\Z}{4\Z}{\Z}&\mat{\Z}{\Z}{4\Z}{\Z}&&5\\
  \midrule
   \Gamma_0'(2)&\mat{\Z}{\Z}{\Z}{\Z}&\mat{\Z}{\Z}{2\Z}{\Z}&&\mat{\Z}{\Z}{\Z}{\Z}&&&&&&3\\
  \midrule
   \Gamma_0'(4)&\mat{\Z}{\Z}{\Z}{\Z}&\mat{\Z}{\Z}{4\Z}{\Z}&&\mat{\Z}{\Z}{\Z}{\Z}&\mat{\Z}{\Z}{\Z}{\Z}&\mat{\Z}{\Z}{2\Z}{\Z}&&&\mat{\Z}{\Z}{4\Z}{\Z}&6\\
  \midrule
   M(4)&\mat{\Z}{\Z}{\Z}{\Z}&\mat{\Z}{2^{-1}\Z}{4\Z}{\Z}&&\mat{\Z}{\Z}{\Z}{\Z}&&\mat{\Z}{\Z}{2\Z}{\Z}&&&\mat{\Z}{2^{-1}\Z}{4\Z}{\Z}&5\\
  \midrule
   B(2)&\mat{\Z}{\Z}{2\Z}{\Z}&\mat{\Z}{\Z}{2\Z}{\Z}&\mat{\Z}{\Z}{2\Z}{\Z}&\mat{\Z}{\Z}{2\Z}{\Z}&&&&&&4\\
  \bottomrule
 \end{array}
\end{equation*}
\end{table}

The following generating series can easily be derived from well-known dimension formulas; see, for example, \cite[Thm.~3.5.1]{DiamondShurman2005}.
\begin{align}
 \label{SkSL2Zdim}&\sum_{k=0}^\infty \dim  S_k(\SL(2,\Z))t^k=\sum_{\substack{k=6\\ k\text{ even}}}^\infty \dim  S_k(\SL(2,\Z))t^k=\frac{t^{12}}{(1-t^4)(1-t^6)}.\\
 \label{SkGamma02dim}&\sum_{k=0}^\infty \dim  S_k(\Gamma_0^{(1)}(2))t^k=\sum_{\substack{k=6\\ k\text{ even}}}^\infty \dim  S_k(\Gamma_0^{(1)}(2))t^k=\frac{t^{8}}{(1-t^2)(1-t^4)}.\\
 \label{SkGamma04dim}&\sum_{k=0}^\infty \dim  S_k(\Gamma_0^{(1)}(4))t^k=\sum_{\substack{k=6\\ k\text{ even}}}^\infty \dim  S_k(\Gamma_0^{(1)}(4))t^k=\frac{t^6}{(1-t^2)^2}.
\end{align}
Using these formulas, Theorem~\ref{thmcodim} and the information in Tables~\ref{doublecosetsPtable} and~\ref{doublecosetsQtable}, we now get the following result.

\begin{theorem}
For even $k\ge 6$ and a congruence subgroup $\Gamma$ as in \eqref{congruencesubgroupseq}, the quantity $\dim M_k(\Gamma)-\dim S_k(\Gamma)$ is given as in Table~\ref{codimensionstable}.	
\end{theorem}

\emph{Remark:} After we calculate $\dim M_4(\Gamma)$ and $\dim S_4(\Gamma)$ in the next section, it will turn out that the codimension formulas in Table~\ref{codimensionstable} also hold for $k=4$. See \cite{BochererIbukiyama2012} for other cases in which Satake's method still works for $k=4$.

\begin{table}
 \caption{Codimension formulas valid for even weight $k\geq6$, given in the form $\dim M_k(\Gamma)-\dim S_k(\Gamma)=\alpha\dim S_k(\SL(2,\Z))+\beta\dim S_k(\Gamma_0^{(1)}(2))+\gamma\dim S_k(\Gamma_0^{(1)}(4))+\delta$.}
 \label{codimensionstable}
\begin{equation*}\setlength{\arraycolsep}{1ex}\renewcommand{\arraystretch}{2}
 \begin{array}{ccccccc}
  \toprule
   \Gamma&\alpha&\beta&\gamma&\delta&\displaystyle\sum_{k=6}^\infty(\dim M_k(\Gamma)-\dim S_k(\Gamma))t^k\\
  \toprule
   \Gamma(2)&0&0&15&15&15\displaystyle\frac{t^6(2-t^2)}{(1-t^2)^2}\\
  \midrule
   \Sp(4,\Z)&1&0&0&1&\displaystyle\frac{t^6(1+t^2-t^8)}{(1-t^4)(1-t^6)}\\
  \midrule
   K(2)&2&0&0&1&\displaystyle\frac{t^6(1+t^2+t^6-t^8)}{(1-t^4)(1-t^6)}\\
  \midrule
   K(4)&2&1&0&2&\displaystyle\frac{t^6(2+3t^2+t^4+t^6-2t^8)}{(1-t^4)(1-t^6)}\\
  \midrule
   \Gamma_0(2)&0&2&0&3&\displaystyle\frac{t^6(3+2t^2-3t^4)}{(1-t^2)(1-t^4)}\\
  \midrule
   \Gamma_0(4)&0&0&4&7&\displaystyle\frac{t^6(11-7t^2)}{(1-t^2)^2}\\
  \midrule
   \Gamma_0^*(4)&0&0&5&7&\displaystyle\frac{t^6(12-7t^2)}{(1-t^2)^2}\\
  \midrule
   \Gamma_0'(2)&2&1&0&2&\displaystyle\frac{t^6(2+3t^2+t^4+t^6-2t^8)}{(1-t^4)(1-t^6)}\\
  \midrule
   \Gamma_0'(4)&3&1&2&4&\displaystyle\frac{t^6(6+9t^2+5t^4+2t^6-4t^8)}{(1-t^4)(1-t^6)}\\
  \midrule
   M(4)&2&3&0&3&\displaystyle\frac{t^6(3+6t^2+3t^4+2t^6-3t^8)}{(1-t^4)(1-t^6)}\\
  \midrule
   B(2)&0&4&0&4&\displaystyle4\frac{t^6(1+t^2-t^4)}{(1-t^2)(1-t^4)}\\
  \bottomrule
 \end{array}
\end{equation*}
\end{table}
\subsection{Dimension formulas for some congruence subgroups}\label{Igusasec}
In this section we determine $\dim M_k(\Gamma)$ and $\dim S_k(\Gamma)$ for all non-negative integers $k$ and all congruence subgroups $\Gamma$ in \eqref{congruencesubgroupseq} except $\Gamma_0'(4)$. Many of the dimension formulas for these groups have appeared before in the literature, but to the best of our knowledge the groups $\Gamma_0^*(4)$ and $M(4)$ have not been previously considered. References are contained in Tables~\ref{dimMktable} and~\ref{dimSktable}. Except for $K(2)$ and $\Gamma_0'(4)$, the dimension of $M_k(\Gamma)$ for $\Gamma$ in \eqref{congruencesubgroupseq} can be determined from \cite[Thm.~2]{Igusa1964}. The method is well known, but we summarize it for completeness.

Let $\Gamma$ be a congruence subgroup of $\Sp(4,\Q)$ for which there exists an element $g\in\Sp(4,\Q)$ such that the group $\tilde\Gamma:=g\Gamma g^{-1}$ satisfies $\Gamma(2)\subset\tilde\Gamma\subset\Sp(4,\Z)$. Evidently $\dim M_k(\Gamma)=\dim M_k(\tilde\Gamma)$. The group $\Sp(4,\Z)/\Gamma(2)\cong\Sp(4,\mathbb{F}_2)\cong S_6$ acts naturally on the space $M_k(\Gamma(2))$. The character of this action has been determined in \cite[Thm.~2]{Igusa1964}. The space $M_k(\tilde\Gamma)$ is the fixed space of this action under the subgroup $\tilde\Gamma/\Gamma(2)$. Hence we can use formula \eqref{dimUH} and \cite[Thm.~2]{Igusa1964} to calculate $\dim M_k(\tilde\Gamma)$. All we need to know is how many elements of each conjugacy class of $S_6$ are contained in $\tilde\Gamma/\Gamma(2)$. For our subgroups of interest, we have already summarized this information in Table~\ref{conjugacyclassestable}.

\begin{proposition}\label{MkSkdimprop}
 With the possible exception of $\Gamma=\Gamma_0'(4)$, the generating series for $\dim M_k(\Gamma)$ given in Table~\ref{dimMktable} and for $\dim S_k(\Gamma)$ given in Table~\ref{dimSktable} hold.
\end{proposition}
\begin{proof}
The dimensions of $M_k(K(2))$ and $S_k(K(2))$ are given in \cite[Prop.~2]{IbukiyamaOnodera1997}. (Note that the generating series for $\dim S_k(K(2))$ given in \cite[Prop.~2]{IbukiyamaOnodera1997} is missing the odd weights. The correct formula, given in Table~\ref{dimSktable}, can be derived from the original source \cite[Thm.~4]{Ibukiyama1985}.) We may therefore assume that a conjugate of $\Gamma$ lies between $\Gamma(2)$ and $\Sp(4,\Z)$. For such $\Gamma$ the dimension of $M_k(\Gamma)$ can be derived from \eqref{dimUH} and \cite[Thm.~2]{Igusa1964}, as explained above. Hence we obtain the information in Table~\ref{dimMktable}.

The quantity $\sum_{k=6}^\infty\dim S_k(\Gamma)$ follows from Tables~\ref{codimensionstable} and~\ref{dimMktable}. It remains to explain $\dim S_k(\Gamma)$ for $k\in\{1,2,3,4,5\}$. For odd $k$ we have $\dim M_k(\Gamma)=\dim S_k(\Gamma)$ by Theorem~\ref{oddweightcodimtheorem}. We have $S_4(\Gamma(2))=0$ by \cite[p.~882]{Tsushima1982}. Hence also $S_2(\Gamma(2))=0$, and $S_2(\Gamma)=S_4(\Gamma)=0$ for all $\Gamma$ that contain a conjugate of $\Gamma(2)$. This concludes the proof.
\end{proof}

For illustration we have listed $\dim M_k(\Gamma)$ and $\dim S_k(\Gamma)$ for weights $k\le 20$ in Tables~\ref{Mklowweightstable}~and~\ref{Sklowweightstable}.
\section{Counting automorphic representations}\label{Sect:countingskpO}
This section contains our main results. In essence, we will use the dimension formulas proven or quoted so far in order to count the number of certain automorphic representations. Then we will use these counts to derive more dimension formulas. It is essential to consider the packet structure of the discrete automorphic spectrum of $\GSp(4,\A)$, which we recall first.
\subsection{Arthur packets}\label{Arthursec}
We recall from \cite{Arthur2004} that there are six types of automorphic representations of $\GSp(4,\A)$ in the discrete spectrum. We are only interested in representations with trivial central character, for which the description simplifies as follows.
\begin{itemize}
 \item The general type \textbf{(G)}: These representations are characterized by the fact that they lift to cusp forms on $\GL(4,\A)$ with trivial central character. They consist of finite, tempered and stable packets. The latter means that if $\pi\cong\otimes\pi_v$ is such a representation, and if one of the local components $\pi_w$ is part of a local $L$-packet $\{\pi_w,\pi_w'\}$, then $\pi':=\pi_w'\otimes\big(\otimes_{v\neq w}\pi_v\big)$ is also an automorphic representation in the discrete spectrum. All these representations are cuspidal.
 \item The Yoshida type \textbf{(Y)}: These packets are parametrized by pairs of distinct, cuspidal automorphic representations $\mu_1,\mu_2$ of $\GL(2,\A)$ with trivial central character. The packets are tempered and finite, but they are not stable. If $\pi\cong\otimes\pi_v$ is parametrized by $\mu_1\cong\otimes\mu_{1,v}$ and $\mu_2\cong\otimes\mu_{2,v}$, then the local $L$-parameter of $\pi_v$ is the direct sum of the $L$-parameters of $\mu_{1,v}$ and $\mu_{2,v}$. Given $\mu_1$ and $\mu_2$, if the $\pi_v$ are chosen from the local $L$-packets parametrized by $\mu_{1,v}$ and $\mu_{2,v}$, then $\pi=\otimes\pi_v$ belongs to the discrete spectrum if and only if the number of non-generic $\pi_v$ is even.
 \item The Saito-Kurokawa type \textbf{(P)}: These packets are parametrized by pairs $(\mu,\sigma)$, where $\mu$ is a cuspidal, automorphic representation of $\GL(2,\A)$ with trivial central character, and $\sigma$ is a quadratic Hecke character. We will see in Lemma~\ref{Psigmatrivlemma} below that only the case $\sigma=1$ is relevant to us, in which case we say that $\pi$ is a Saito-Kurokawa lift of $\mu$. The packets are finite, non-tempered and not stable. Given $\mu$, if the $\pi_v$ are chosen from local Arthur packets (listed in \cite[Table~2]{Schmidt2020}) parametrized by $\mu_v$, then $\pi=\otimes\pi_v$ belongs to the discrete spectrum if and only if the parity condition
 \begin{equation}\label{parityconditionP}
  \varepsilon(1/2,\mu)=(-1)^n
 \end{equation}
 is satisfied, where $n$ is the number of places where $\pi_v$ is \emph{not} the base point in the local Arthur packet.
 \item The Soudry type \textbf{(Q)}: These are parametrized by self-dual, cuspidal, automorphic representations of $\GL(2,\A)$ with non-trivial central character. The packets are non-tempered, infinite and stable. The local Arthur packets are given in \cite[Table~3]{Schmidt2020}.
 \item The Howe--Piatetski-Shapiro type \textbf{(B)}: The packets are parametrized by pairs of distinct, quadratic Hecke characters. They are non-tempered, infinite and unstable. The local Arthur packets are given in \cite[Table~1]{Schmidt2020}.
 \item The finite type \textbf{(F)}: These are one-dimensional representations. They are not relevant for this work, since they are not cuspidal.
\end{itemize}

We next determine how these types intersect with the representations of interest to us. For the following definition let $\Omega$ be one of the representation types I, IIa, IIb, \ldots, XIb, Va$^*$, $\sc$ appearing in Table~\ref{parahoricrestrictiontable}.
\begin{definition}\label{Skomegadef}
 Let $k$ be a positive integer. Let $S_k(\Omega)$ be the set of cuspidal automorphic representations $\pi\cong\otimes_v\pi_v$ of $\GSp(4,\A)$ with the following properties:
 \begin{enumerate}
  \item $\pi$ has trivial central character.
  \item  $\pi_{\infty}$ is the lowest weight module with minimal $K$-type $(k,k)$; it is a holomorphic discrete series representation if $k\geq3$, a holomorphic limit of discrete series representation if $k=2$, and a non-tempered representation if $k=1$. (It was denoted by $\mathcal{B}_{k,0}$ in \cite[Sect.~3.5]{Schmidt2017}.)
  \item $\pi_p$ is unramified for each finite $p\neq 2$.
  \item $\pi_2$ is an irreducible, admissible representation of $\GSp(4,\Q_2)$ of type $\Omega$ with non-trivial $\Gamma(\p)$-invariant vectors.
 \end{enumerate}
\end{definition}
We note a peculiarity about representation types Vb and Vc. While these occupy two different rows in Table~\ref{parahoricrestrictiontable}, the resulting sets of representations are identical, if the parameters in Table~\ref{parahoricrestrictiontable} are allowed to vary over all possibilities. Therefore $S_k(\text{Vb})=S_k(\text{Vc})$. In the following we will work with Vb and ignore Vc.

\begin{proposition}\label{Skomegaprop}
 Let $\Gamma$ be one of the congruence subgroups of $\Sp(4,\Q)$ in~\eqref{congruencesubgroupseq}. Suppose that $\pi$ is one of the cuspidal, automorphic representations of $\GSp(4,\A)$ generated by the adelization of some non-zero $f\in S_k(\Gamma)$. Then $\pi\in S_k(\Omega)$ for some $\Omega$.
\end{proposition}
\begin{proof}
Recall from \cite[Sect.~4.2]{Schmidt2017} (among other places) that the adelization $\Phi$ of $f\in S_k(\Gamma)$ is the unique function $G(\A)\to\C$ which is left invariant under $G(\Q)$, invariant under the center of $G(\A)$, right invariant under
\begin{equation}\label{Skomegapropeq1}
 C_2\times\prod_{\substack{p<\infty\\p\neq2}}G(\Z_p),
\end{equation}
and satisfies
\begin{equation}\label{Skomegapropeq2}
 \Phi(g)=(f|_kg)(\begin{bsmallmatrix}
 	i&\\&i
 \end{bsmallmatrix})\qquad\text{for all }g\in\Sp(4,\R).
\end{equation}
Here, $C_2$ is the congruence subgroup of $G(\Q_2)$ analogous to $\Gamma$, or more precisely, the closure of $\Gamma$ in $\Sp(4,\Q_2)$ times the the group of ``multiplier matrices'' ${\rm diag}(1,1,x,x)$ with $x\in\Z_2^\times$. 

If $\pi\cong\otimes\pi_v$ is one of the irreducible components of the representation generated by $\Phi$ under right translation, then it is clear that each $\pi_p$ for primes $p\neq2$ is spherical, and that $\pi_2$ contains non-zero $C_2$-invariant vectors. By Theorem~\ref{localdimtheorem}~ii) the representation $\pi_2$ contains non-zero $\Gamma(\p)$-invariant vectors. Finally, it follows from the holomorphy of $f$ that $\pi_\infty$ is a lowest weight representation minimal $K$-type $(k,k)$; see \cite{AsgariSchmidt2001}. Hence $\pi\in S_k(\Omega)$ for some $\Omega$.
\end{proof}

The point of Proposition~\ref{Skomegaprop} is that, if $\Gamma$ is one of the congruence subgroups in~\eqref{congruencesubgroupseq}, then no cuspidal, automorphic representation of $\GSp(4,\A)$ besides those in $S_k(\Omega)$, where $\Omega$ runs through the types occurring in Table~\ref{parahoricrestrictiontable}, will contribute to $S_k(\Gamma)$. Of course, the $S_k(\Omega)$ contribute to $S_k(\Gamma')$ for many other congruence subgroups $\Gamma'$ (for example, subgroups or conjugates of any of the $\Gamma$'s in~\eqref{congruencesubgroupseq}).

Let $S_k^{\text{\bf(G)}}(\Omega)$ be the subset of $\pi\in S_k(\Omega)$ that are of type \textbf{(G)}, and similarly for the other Arthur types. Let $s_k(\Omega)$ be the cardinality of $S_k(\Omega)$, and $s_k^{(*)}(\Omega)$ be the cardinality of $S_k^{(*)}(\Omega)$. Evidently,
\begin{equation}\label{SkArthurdecompeq}
 S_k(\Omega)=S_k^{\text{\bf(G)}}(\Omega)\,\sqcup\,S_k^{\text{\bf(Y)}}(\Omega)\,\sqcup\,S_k^{\text{\bf(P)}}(\Omega)\,\sqcup\,S_k^{\text{\bf(Q)}}(\Omega)\,\sqcup\,S_k^{\text{\bf(B)}}(\Omega),
\end{equation}
so that
\begin{equation}\label{no. of cup form0}
 s_k(\Omega)=s_k^{\text{\bf(G)}}(\Omega)+s_k^{\text{\bf(Y)}}(\Omega)+s_k^{\text{\bf(P)}}(\Omega)+s_k^{\text{\bf(Q)}}(\Omega)+s_k^{\text{\bf(B)}}(\Omega).
\end{equation}
It follows from Proposition~\ref{Skomegaprop} that
\begin{equation}\label{SMFarthurdecomp}
 S_{k}(\Gamma)=S^{\text{\bf(G)}}_k(\Gamma)\oplus S^{\text{\bf(Y)}}_k(\Gamma) \oplus S^{\text{\bf(P)}}_k(\Gamma) \oplus S^{\text{\bf(Q)}}_k(\Gamma)\oplus S^{\text{\bf(B)}}_k(\Gamma)
\end{equation}
for any of the congruence subgroups $\Gamma$ in~\eqref{congruencesubgroupseq}, the obvious notation being that elements of $S_k^{(*)}(\Omega)$ (for any possible $\Omega$) give rise to elements of $S_k^{(*)}(\Gamma)$. Hence
\begin{equation}\label{Dimension formula}
 \dim  S_k(\Gamma)=\dim  S^{\text{\bf(G)}}_k(\Gamma)+\dim  S^{\text{\bf(Y)}}_k(\Gamma)+\dim  S^{\text{\bf(P)}}_k(\Gamma)+\dim  S^{\text{\bf(Q)}}_k(\Gamma)+\dim  S^{\text{\bf(B)}}_k(\Gamma).
\end{equation}
\begin{proposition}\label{YQBzeroprop}
 Let $k$ be a positive integer. Then
 \begin{equation}\label{YQBzeropropeq1}
  S^{\text{\bf(Y)}}_k(\Omega)=S^{\text{\bf(Q)}}_k(\Omega)=S^{\text{\bf(B)}}_k(\Omega)=\emptyset
 \end{equation}
 for any $\Omega$, and hence
 \begin{equation}\label{YQBzeropropeq2}
  S^{\text{\bf(Y)}}_k(\Gamma)=S^{\text{\bf(Q)}}_k(\Gamma)=S^{\text{\bf(B)}}_k(\Gamma)=0
 \end{equation}
 for any of the congruence subgroups $\Gamma$  in~\eqref{congruencesubgroupseq}.
\end{proposition}
\begin{proof}
For types \textbf{(Q)} or \textbf{(B)}, the proof is analogous to that of \cite[Prop.~2.1]{RoySchmidtYi2021}. If $\pi\cong\otimes\pi_v$ lies in an Arthur packet of type \textbf{(Q)} or \textbf{(B)}, then the characters parametrizing the packet are ramified at least at one prime $p$. A look at \cite[Table~1, Table~3]{Schmidt2020} shows that $\pi_p$ is not among the representations listed in Table~\ref{parahoricrestrictiontable}. (Recall that all the characters appearing in Table~\ref{parahoricrestrictiontable} are unramified.) Therefore $\pi\notin S_k(\Omega)$ for any $\Omega$.

Now consider a cuspidal, automorphic representation $\pi=\otimes\pi_v$ of type \textbf{(Y)}. Recall that the packet containing $\pi$ is parametrized by two distinct, cuspidal automorphic representations $\mu_1=\otimes\mu_{1,v}$ and $\mu_2=\otimes\mu_{2,v}$ of $\GL(2,\A)$ with trivial central character. Chasing through archimedean Langlands parameters, we see that in order for $\pi_\infty$ to be a lowest weight representation of weight $k$, the only possibility, up to order, is that $\mu_{1,\infty}$ is a discrete series representation of $\PGL(2,\R)$ of lowest weight $2k-2$, and $\mu_{2,\infty}$ is a discrete series representation of $\PGL(2,\R)$ of lowest weight $2$. Hence, $\mu_1$ corresponds to a newform $f_1\in S_{2k-2}(\Gamma_0^{(1)}(N_1))$ and $\mu_2$ corresponds to a newform $f_2\in S_2(\Gamma_0^{(1)}(N_2))$ for some levels $N_1,N_2$. If we want $\pi$ to be in $S_k(\Omega)$ for some $\Omega$, then $N_1$ and $N_2$ both have to be powers of $2$. Now $S_2(\Gamma_0^{(1)}(4))=0$, so that we would need $N_2=2^n$ for some $n\geq3$. But then the local component $\mu_{2,2}$, whose $L$-parameter is a direct summand of the $L$-parameter of $\pi_2$, is such that $\pi_2$ is not among the representations listed in Table~\ref{parahoricrestrictiontable}; see \cite[Eq.~(16)]{SahaSchmidt2013} for the possible local Yoshida packets. It follows that $\pi$ cannot be in $S_k(\Omega)$ for any $\Omega$.

Note that \eqref{YQBzeropropeq2} follows from \eqref{YQBzeropropeq1} in view of Proposition~\ref{Skomegaprop}.
\end{proof}

As a consequence of Proposition~\ref{YQBzeroprop},
\begin{equation}\label{SkArthurdecompeq2}
 S_k(\Omega)=S_k^{\text{\bf(G)}}(\Omega)\,\sqcup\,S_k^{\text{\bf(P)}}(\Omega),
\end{equation}
so that $s_k(\Omega)=s_k^{\text{\bf(G)}}(\Omega)+s_k^{\text{\bf(P)}}(\Omega)$, and
\begin{equation}\label{SMFarthurdecomp2}
 S_{k}(\Gamma)=S^{\text{\bf(G)}}_k(\Gamma)\oplus S^{\text{\bf(P)}}_k(\Gamma),
\end{equation}
so that $\dim  S_k(\Gamma)=\dim  S^{\text{\bf(G)}}_k(\Gamma)+\dim  S^{\text{\bf(P)}}_k(\Gamma)$. In Sect.~\ref{SKsec} we will determine the numbers $s_k^{\text{\bf(P)}}(\Omega)$.

The sets $S_k^{\text{\bf(G)}}(\Omega)$ are empty for certain $\Omega$, because the local Arthur packets (which are $L$-packets in this case) must contain a tempered element. The \textbf{(G)} column in Table~\ref{parahoricrestrictiontable} indicates which $\Omega$ can occur in packets of type \textbf{(G)}. Similarly, the sets $S_k^{\text{\bf(P)}}(\Omega)$ are empty for certain $\Omega$, because the local Arthur packets can only contain the representations listed in \cite[Table~2]{Schmidt2020}. The \textbf{(P)} column in Table~\ref{parahoricrestrictiontable} indicates which $\Omega$ can occur in packets of type \textbf{(P)}. We see that the only representations that can occur in packets of both Arthur types \textbf{(G)} and \textbf{(P)} or those of type VIb and Va$^*$.

Since Arthur packets of type {\bf(G)} are stable, one can switch within local $L$-packets and still retain the automorphic property. Most representations in Table~\ref{parahoricrestrictiontable} constitute singleton $L$-packets, except $\{\rm Va,Va^*\}$, $\{ \rm VIa, VIb\}$ and $\{ \rm VIIIa, VIIIb\}$, which constitute $2$-element $L$-packets ($\rm XIa$ is also part of an $L$-packet $\{ \rm XIa, XIa^*\}$, but $\rm XIa^*$ does not appear in Table~\ref{parahoricrestrictiontable}).

Hence $s_k^{\text{\bf(G)}}({\rm Va})=s_k^{\text{\bf(G)}}({\rm Va}^*)$, $s_k^{\text{\bf(G)}}({\rm VIa})=s_k^{\text{\bf(G)}}({\rm VIb})$, $s_k^{\text{\bf(G)}}({\rm VIIIa})=s_k^{\text{\bf(G)}}({\rm VIIIb})$ and we denote these common numbers as follows:
\begin{align}
 \label{skVaastar}
   s_k^{\text{\bf(G)}}({\rm Va/a}^*)&\colonequals s_k^{\text{\bf(G)}}({\rm Va})=s_k^{\text{\bf(G)}}({\rm Va}^*),\\
 \label{skVIab}
   s_k^{\text{\bf(G)}}({\rm VIa/b})&\colonequals s_k^{\text{\bf(G)}}({\rm VIa})=s_k^{\text{\bf(G)}}({\rm VIb}),\\
 \label{skVIIIab}
   s_k^{\text{\bf(G)}}({\rm VIIIa/b})&\colonequals s_k^{\text{\bf(G)}}({\rm VIIIa})=s_k^{\text{\bf(G)}}({\rm VIIIb}).
\end{align}
We observe from Table~\ref{localdimtable} that, for each of the congruence subgroups $H$ in this table, the dimension of the space of $H$-invariant vectors in a type IIIa (resp.\ VII) representation equals the sum of the dimensions of the spaces of $H$-invariant vectors for the $L$-packet VIa/b (resp.\ VIIIa/b). (The reason is that IIIa is a parabolically induced representation $\chi\rtimes\sigma\St_{\GSp(2)}$ for an unramified, non-trivial character $\chi$, and VIa/b are the two constitutents of the same induced representation when $\chi$ is trivial. Similarly, VII is $\chi\rtimes\pi$ for an unramified, non-trivial $\chi$ and VIIIa/b are the two constituents of the same induced representation with trivial $\chi$.) Since our methods cannot determine the numbers $s_k^{\text{\bf(G)}}({\rm IIIa})$ and $s_k^{\text{\bf(G)}}({\rm VIa/b})$ (resp. $s_k^{\text{\bf(G)}}({\rm VII})$ and $s_k^{\text{\bf(G)}}({\rm VIIIa/b})$) separately, we consider
\begin{align}
 \label{skGIIIaeq}
  s_k^{\text{\bf(G)}}({\rm IIIa+VIa/b})&\colonequals s_k^{\text{\bf(G)}}({\rm IIIa})+s_k^{\text{\bf(G)}}({\rm VIa/b}),\\
 \label{skGVIIeq}
  s_k^{\text{\bf(G)}}(\mathrm{VII+VIIIa/b})&\colonequals s_k^{\text{\bf(G)}}(\mathrm{VII})+s_k^{\text{\bf(G)}}(\mathrm{VIIIa/b}).
\end{align}
\subsection{Siegel modular forms and representations in \texorpdfstring{$S_k(\Omega)$}{}}
Consider $\pi\cong\bigotimes_{p\leq\infty}\pi_p\in S_k(\Omega)$. Recall that $\pi_2$ is an irreducible, admissible representation of $\PGSp(4,\Q_2)$ of type $\Omega$ with non-zero hyperspecial parahoric restriction $r_K(\pi_2)$. Let $C$ be one of the compact open subgroups in Table~\ref{parahoricrestrictiontable}, and let $\Gamma$ be the corresponding congruence subgroup of $\Sp(4, \Q)$. More precisely,
\begin{equation}\label{CGammaeq}
 \Gamma=\Sp(4,\Q)\cap\bigg(C\times\prod_{\substack{p<\infty\\p\neq2}}\GSp(4,\Z_p)\bigg).
\end{equation}
Then every eigenform (for the Hecke operators at all odd primes) $f\in S_k(\Gamma)$ arises from a vector in  $\pi_2^C$, for some $\pi\in S_k(\Omega)$, by a procedure similar to the one explained in \cite[Sect.~2.1]{RoySchmidtYi2021}. Thus we obtain the formula
\begin{equation}\label{dimSkdCOmegaeq1}
 \dim  S_k(\Gamma)=\sum_\Omega\sum_{\pi\in S_k(\Omega)}\dim \pi_2^C=\sum_\Omega s_k(\Omega)d_{C,\Omega},
\end{equation}
where $d_{C,\Omega}$ is the common dimension of the space of $C$-fixed vectors of the representations $\pi_2$ of type $\Omega$ with $r_K(\pi_2)\neq0$. The $d_{C,\Omega}$ are the numbers listed in Table~\ref{localdimtable}. Hence the equations \eqref{dimSkdCOmegaeq1} for all $\Gamma$ and all $\Omega$ are equivalent to the matrix equation

	\begin{equation}\label{mainmatrixeq1}
	\begin{bsmallmatrix}
		\dim S_k(\Gamma(2))\\
		\dim S_k(\Sp(4, \Z))\\
		\dim S_k(K(2))\\
		\dim S_k(K(4))\\
		\dim S_k(\Gamma_0(2))\\
		\dim S_k(\Gamma_0(4))\\
		\dim S_k(\Gamma_0^{\ast}(4))\\
		\dim S_k(\Gamma_0'(2))\\
		\dim S_k(\Gamma_0'(4))\\
		\dim S_k(M(2))\\
		\dim S_k(B(2))
	\end{bsmallmatrix}=
	\resizebox{0.6\textwidth}{!}{$
		\left[
		\begin{array}{cccccccccccccccccccc}
			45&30&15&30&16&21&9&25&5&5&15&10&5&10&15&10&5&1&9\\
			1&0&1&0&0&0&0&0&0&0&0&0&0&0&0&0&0&0&0\\
			2&1&1&0&0&0&1&0&0&1&0&0&0&0&0&0&0&0&0\\
			4&2&2&1&0&1&1&1&0&1&0&0&0&0&1&0&1&0&0\\
			4&1&3&2&0&0&1&1&1&0&0&0&0&0&0&0&0&0&0\\
			12&5&7&8&2&2&3&5&3&0&4&3&1&3&1&1&0&0&0\\
			15&8&7&10&4&5&3&7&3&1&5&4&1&4&7&4&3&1&3\\
			4&2&2&1&0&1&1&1&0&1&0&0&0&0&0&0&0&0&0\\
			11&7&4&5&2&5&2&5&0&2&2&2&0&1&3&2&1&0&1\\
			8&5&3&3&1&3&2&3&0&2&0&0&0&0&2&1&1&0&0\\
			8&4&4&4&1&2&2&3&1&1&0&0&0&0&0&0&0&0&0
		\end{array}\right]$}
	\begin{bsmallmatrix}
		s_k(\mathrm{I})\\
		s_k(\mathrm{IIa})\\
		s_k(\mathrm{IIb})\\
		s_k(\mathrm{IIIa})\\
		s_k(\mathrm{IVa})\\
		s_k(\mathrm{Va})\\
		s_k(\mathrm{Vb})\\
		s_k(\mathrm{VIa})\\
		s_k(\mathrm{VIb})\\
		s_k(\mathrm{VIc})\\
		s_k(\mathrm{VII})\\
		s_k(\mathrm{VIIIa})\\
		s_k(\mathrm{VIIIb})\\
		s_k(\mathrm{IXa})\\
		s_k(\mathrm{X})\\
		s_k(\mathrm{XIa})\\
		s_k(\mathrm{XIb})\\
		s_k(\mathrm{Va}^*)\\
		s_k(\sc(16))
	\end{bsmallmatrix}.
\end{equation}
Here, we have omitted those $\Omega$ which do not occur in packets of type \textbf{(G)} or \textbf{(P)}, because for these $s_k(\Omega)=0$ by Proposition~\ref{YQBzeroprop}. Note also that the class of representations of type Vb is the same as the class of representations of type Vc, since the parameter $\sigma$ in Table~\ref{parahoricrestrictiontable} runs through all possibilities. We therefore include only Vb in \eqref{mainmatrixeq1}.

The identity \eqref{mainmatrixeq1} still holds if we put a \textbf{(G)} or a \textbf{(P)} on all the $S_k(\Gamma)$ and all the $s_k(\Omega)$; this is the definition of the spaces $S^{\text{\bf(G)}}_k(\Gamma)$ and $S^{\text{\bf(P)}}_k(\Gamma)$. More of the $s_k^{(*)}(\Omega)$ will then be zero; see Table~\ref{parahoricrestrictiontable}. We will utilize the \textbf{(P)} version of \eqref{mainmatrixeq1} in the proof of Corollary~\ref{dimSkPcorollary}, and the \textbf{(G)} version in the proof of Theorem~\ref{skGtheorem}. More precisely, we will proceed as follows.
\begin{itemize}
 \item Exploiting the fact that packets of type \textbf{(P)} are parametrized by cuspidal, automorphic representations of $\GL(2,\A)$, the numbers $s^{\text{\bf(P)}}_k(\Omega)$ can be determined for all $\Omega$ from dimension formulas for elliptic modular forms. (Theorem~\ref{skPtheorem})
 \item We then use the \textbf{(P)} version of \eqref{mainmatrixeq1} to calculate $\dim S^{\text{\bf(P)}}_k(\Gamma)$ for all $\Gamma$ (Corollary~\ref{dimSkPcorollary}).
 \item Since we already determined $\dim S_k(\Gamma)$ for all $\Gamma$ except $\Gamma_0'(4)$, we can calculate $\dim S^{\text{\bf(G)}}_k(\Gamma)$ for all $\Gamma$ except $\Gamma_0'(4)$. (Proposition~\ref{dimSkGprop})
 \item Then we use the \textbf{(G)} version of \eqref{mainmatrixeq1}, with the row for $\Gamma_0'(4)$ omitted, to determine the $s^{\text{\bf(G)}}_k(\Omega)$. Here, it is necessary to combine some types $\Omega$ which cannot be distinguished by their fixed vector dimensions; see \eqref{skGIIIaeq}, \eqref{skGVIIeq}. This step reduces the number of unknowns to 10, the same as the number of equations. (Theorem~\ref{skGtheorem})
 \item Next we use the $\Gamma_0'(4)$-row of the \textbf{(G)} version of \eqref{mainmatrixeq1} to determine $\dim S^{\text{\bf(G)}}_k(\Gamma_0'(4))$. Since we already have $\dim S^{\text{\bf(P)}}_k(\Gamma_0'(4))$, this gives us $\dim S_k(\Gamma_0'(4))$. (Corollary~\ref{Klingen4corollary})
\end{itemize}
Finally, we will be able to fill in the row for $\dim M_k(\Gamma_0'(4))$ in Table~\ref{dimMktable}, using the codimension formula for $k\geq6$ given in Table~\ref{codimensionstable}, and the low weight results from Appendix~\ref{CDappendix}.
\subsection{Saito-Kurokawa type}\label{SKsec}
Recall from Sect.~\ref{Arthursec} that Arthur packets of type \textbf{(P)} are parametrized by pairs $(\mu,\sigma)$, where $\mu$ is a cuspidal, automorphic representation of $\GL(2,\A)$ with trivial central character, and $\sigma$ is a quadratic Hecke character.
\begin{lemma}\label{Psigmatrivlemma}
 Suppose that the cuspidal, automorphic representation $\pi$ lies in a packet of type \textrm{\bf{(P)}}, parametrized by the pair $(\mu,\sigma)$, where $\mu$ is a cuspidal, automorphic representation of $\GL(2,\A)$ with trivial central character, and $\sigma$ is a quadratic Hecke character. Suppose that also $\pi\in S_k(\Omega)$ for some $\Omega$. Then $\sigma$ is trivial.
\end{lemma}
\begin{proof}
We write $\pi=\otimes\pi_v$ and $\sigma=\otimes\sigma_v$. The local representation $\pi_v$ occurs in \cite[Table~2]{Schmidt2020}, for any place~$v$. Since $\pi_p$ is spherical for $p\geq3$, we see from \cite[Table~2]{Schmidt2020} that $\sigma_p$ is unramified. Since $\pi_2$ occurs in Table~\ref{parahoricrestrictiontable}, inspecting \cite[Table~2]{Schmidt2020} shows that $\sigma_2$ is also unramified. Hence the character $\sigma$, being unramified everywhere, must be trivial. 
\end{proof}
If $\pi$ lies in a packet of type \textbf{(P)}, parametrized by the pair $(\mu,\sigma)$ with trivial $\sigma$ as in the lemma, then we say that ``$\pi$ is a Saito-Kurokawa lift of $\mu$''. Note that a given $\mu$ may have multiple Saito-Kurokawa lifts, depending on the size of the Arthur packet. As the proof of the next result shows, those $\mu$ corresponding to eigenforms in $S_k^{\rm new}(\Gamma_0^{(1)}(N))$ with $N\in\{2,4\}$ admit a unique \emph{holomorphic} Saito-Kurokawa lift.

\begin{theorem}\label{skPtheorem}
 The generating series for the numbers $s_k^{\text{\bf(P)}}(\Omega)$ given in Table~\ref{skPOmegatable} hold. If a representation type $\Omega$ is not listed in Table~\ref{skPOmegatable}, then $s_k^{\text{\bf(P)}}(\Omega)=0$ for all $k$.
\end{theorem}
\begin{proof}
Table~\ref{SKtable} below shows several spaces of elliptic modular newforms, and how an eigenform in one of these spaces Saito-Kurokawa lifts to $\GSp(4,\A)$. The notation $S_{2k-2}^{\pm,\mathrm{new}}(\Gamma_0^{(1)}(N))$ indicates the subspace of $S_{2k-2}^{\mathrm{new}}(\Gamma_0^{(1)}(N))$ spanned by eigenforms with sign $\pm1$ in the functional equation of their $L$-function. If $\mu\cong\otimes\mu_v$ is the cuspidal, automorphic representation of $\GL(2,\A)$ corresponding to an eigenform in one of these spaces, then the sign in the functional equation coincides with the global $\varepsilon$-factor $\varepsilon(1/2,\mu)=\varepsilon_\infty\varepsilon_2$, where $\varepsilon_\infty:=\varepsilon(1/2,\mu_\infty)=(-1)^{k-1}$ and $\varepsilon_2:=\varepsilon(1/2,\mu_2)$. In the $\mu_2$ column of Table~\ref{SKtable} the symbol $\xi$ stands for the unique non-trivial, unramified, quadratic character of $\Q_2^\times$, and $\tau_2$ denotes the unique irreducible, admissible representation of $\GL(2,\Q_2)$ with trivial central character and conductor exponent $2$; it is a depth zero supercuspidal.

Now $\mu_v$, or rather the pair $(\mu_v,1_v)$, where $1_v$ is the trivial character of $\Q_v^\times$, determines a local Arthur packet consisting of one or two representations, for each place $v$. These local packets are explicitly given in \cite[Table~2]{Schmidt2020}, and each one of them contains a ``base point''. The packet is a singleton if and only if $\mu_v$ is not a discrete series representation, in which case the unique representation in the packet is also the base point. Recall from \eqref{parityconditionP} that in order for the global Saito-Kurokawa packet $\mu$ to contain the cuspidal, automorphic representation $\pi\cong\otimes\pi_v$ of $\GSp(4,\A)$, the parity condition $\varepsilon(1/2,\mu)=(-1)^n$ has to be satisfied, where $n$ is the number of places for which $\pi_v$ is not the base point in the local Arthur packet. Since we want $\pi$ to correspond to holomorphic Siegel modular forms, the set $S$ of places where $\pi_v$ is not the base point must include the archimedean place, the reason being that the non-base point in the archimedean local packet is the holomorphic discrete series representation of $\PGSp(4,\R)$ of lowest weight $(k,k)$. Hence the set $S$ must be $\{\infty\}$ if $\varepsilon(1/2,\mu)=-1$ and must be $\{\infty,2\}$ if $\varepsilon(1/2,\mu)=1$. The final column of Table~\ref{SKtable}, which can be read off \cite[Table~2]{Schmidt2020}, shows the type of $\pi_2$, the local component at $p=2$ of the unique cuspidal, automorphic representation $\pi$ in the global packet parametrized by $\mu$ which has the required discrete series representation at the archimedean place.

\begin{table}
 \caption{Some spaces of elliptic cusp forms and their Saito-Kurokawa lifts.}
 \label{SKtable}
\begin{equation}
 \begin{array}{cccccccc}
  \toprule
   \text{space}&k&\varepsilon_\infty&\varepsilon_2&\mu_2&S&\pi_2\\
  \toprule
   S_{2k-2}(\SL(2,\Z))&\text{even}&-1&1&\text{spherical}&\{\infty\}&\text{IIb}\\
  \cmidrule{2-7}
   &\text{odd}&1&1&\text{spherical}&\multicolumn{2}{c}{\text{no lifting}}\\
  \midrule
   S_{2k-2}^{+,\mathrm{new}}(\Gamma_0^{(1)}(2))&\text{even}&-1&-1&\St_{\GL(2)}&\{\infty,2\}&\text{VIb}\\
  \cmidrule{2-7}
   &\text{odd}&1&1&\xi\St_{\GL(2)}&\{\infty,2\}&\text{Va}^*\\
  \midrule
   S_{2k-2}^{-,\mathrm{new}}(\Gamma_0^{(1)}(2))&\text{even}&-1&1&\xi\St_{\GL(2)}&\{\infty\}&\text{Vb}\\
  \cmidrule{2-7}
   &\text{odd}&1&-1&\St_{\GL(2)}&\{\infty\}&\text{VIc}\\
  \midrule
   S_{2k-2}^{+,\mathrm{new}}(\Gamma_0^{(1)}(4))&\text{even}&-1&-1&\tau_2&\{\infty,2\}&\text{XIa}^*\\
  \cmidrule{2-7}
   &\text{odd}&1&1&\multicolumn{3}{c}{\text{no possible }\mu_2}\\
  \midrule
   S_{2k-2}^{-,\mathrm{new}}(\Gamma_0^{(1)}(4))&\text{even}&-1&1&\multicolumn{3}{c}{\text{no possible }\mu_2}\\
  \cmidrule{2-7}
   &\text{odd}&1&-1&\tau_2&\{\infty\}&\text{XIb}\\
  \bottomrule
 \end{array}
\end{equation}
\end{table}

The upshot is that each newform in one of the spaces given in Table~\ref{SKtable} gives rise to a unique ``holomorphic'' cuspidal representation of $\PGSp(4,\A)$, the only exception being that eigenforms in $S_{2k-2}(\SL(2,\Z))$ for odd $k$ cannot be lifted, because it is impossible to satisfy the parity condition. We can thus produce elements of $S_k^{\text{\bf(P)}}(\Omega)$ for those types $\Omega$ listed in the last column of Table~\ref{SKtable}. Note that representations of type XIa$^*$ do not appear in Table~\ref{parahoricrestrictiontable}, and hence those Saito-Kurokawa lifts are not relevant for our purposes.

Conversely, suppose that $\pi\cong\otimes\pi_v$ is an element of $S_k^{\text{\bf(P)}}(\Omega)$ for some $\Omega$. Then, by Lemma \ref{Psigmatrivlemma}, the Arthur packet containing $\pi$ is parametrized by a cuspidal, automorphic representation $\mu\cong\otimes\mu_v$ and the trivial character $\sigma$. Looking at the archimedean parameters in \cite[Table~2]{Schmidt2020}, we see that $\mu$ corresponds to a newform of weight $2k-2$. There can be no ramification outside $2$, so that the level of this newform is a power of $2$. In fact, the level must be $1$, $2$ or $4$, since otherwise a look at the non-archimedean packets in \cite[Table~2]{Schmidt2020} would show that the local component $\mu_2$ would be such that the elements of the local Arthur packet at $p=2$ would not appear in Table~\ref{parahoricrestrictiontable}. Hence $\pi$ is a lift of a newform of one of the spaces appearing in Table~\ref{SKtable}.

This discussion shows that
\begin{align*}
 s_k^{\text{\bf(P)}}(\text{IIb})&=\begin{cases}
                                   \dim S_{2k-2}(\SL(2,\Z))&\text{if $k$ is even},\\
                                   0&\text{if $k$ is odd},\\
                                  \end{cases}\\
 s_k^{\text{\bf(P)}}(\text{VIb})&=\begin{cases}
                                   \dim S_{2k-2}^{+,\text{new}}(\Gamma_0^{(1)}(2))&\text{if $k$ is even},\\
                                   0&\text{if $k$ is odd},\\
                                  \end{cases}\\
 s_k^{\text{\bf(P)}}(\text{Va}^*)&=\begin{cases}
                                   0&\text{if $k$ is even},\\
                                   \dim S_{2k-2}^{+,\text{new}}(\Gamma_0^{(1)}(2))&\text{if $k$ is odd},\\
                                  \end{cases}\\
 s_k^{\text{\bf(P)}}(\text{Vb})&=\begin{cases}
                                   \dim S_{2k-2}^{-,\text{new}}(\Gamma_0^{(1)}(2))&\text{if $k$ is even},\\
                                   0&\text{if $k$ is odd},\\
                                  \end{cases}\\
 s_k^{\text{\bf(P)}}(\text{VIc})&=\begin{cases}
                                   0&\text{if $k$ is even},\\
                                   \dim S_{2k-2}^{-,\text{new}}(\Gamma_0^{(1)}(2))&\text{if $k$ is odd},\\
                                  \end{cases}\\
 s_k^{\text{\bf(P)}}(\text{XIb})&=\begin{cases}
                                   0&\text{if $k$ is even},\\
                                   \dim S_{2k-2}^{-,\text{new}}(\Gamma_0^{(1)}(4))&\text{if $k$ is odd},\\
                                  \end{cases}
\end{align*}
Now the asserted formulas follow from \eqref{SkSL2Zdim}, \eqref{SkGamma02dim}, \eqref{SkGamma04dim}, the dimension formula for $S_k^{\pm,\text{new}}(\Gamma_0^{(1)}(2))$ in \cite[Thm.~2.2]{Martin2018}, and straightforward calculations.

If $\Omega\notin\{\text{IIb, Vb, VIb, VIc, XIb, Va$^*$, XIa$^*$}\}$, then $s_k^{\text{\bf(P)}}(\Omega)=0$, because type $\Omega$ does not appear in local Arthur packets of type \textbf{(P)}; see \cite[Table~2]{Schmidt2020}. Furthermore, $s_k^{\text{\bf(P)}}(\text{XIa}^*)=0$ because the hyperspecial parahoric restriction for representations of type XIa$^*$ is zero. 
\end{proof}

We note that the cases of $s_k(\Omega)$ for $\Omega\in\{\text{IIb, Vb, VIb, VIc}\}$ can be found in \cite[(3.6) and Sect.~3.2]{RoySchmidtYi2021}.
\begin{corollary}\label{dimSkPcorollary}
 The dimension formulas for Saito-Kurokawa cusp forms given in Table~\ref{dimSkPtable} hold.
\end{corollary}
\begin{proof}
This is immediate from Theorem~\ref{skPtheorem} and the following \textbf{(P)} version of \eqref{mainmatrixeq1}.

\begin{equation}\label{dimskPeq}
	\begin{bsmallmatrix}\dim S_k^{\text{\bf(P)}}(\Gamma(2))\\\dim S_k^{\text{\bf(P)}}(\Sp(4,\Z))\\\dim S_k^{\text{\bf(P)}}(K(2))\\\dim S_k^{\text{\bf(P)}}(K(4))\\\dim S_k^{\text{\bf(P)}}(\Gamma_0(2))\\\dim S_k^{\text{\bf(P)}}(\Gamma_0(4))\\\dim S_k^{\text{\bf(P)}}(\Gamma_0^*(4))\\\dim S_k^{\text{\bf(P)}}(\Gamma_0'(2))\\\dim S_k^{\text{\bf(P)}}(\Gamma_0'(4))\\\dim S_k^{\text{\bf(P)}}(M(4))\\\dim S_k^{\text{\bf(P)}}(B(2))\end{bsmallmatrix}
	=\resizebox{0.23\textwidth}{!}{$
		\left[
		\begin{array}{ccccccc} 	
			15&9&5&5&5&1\\
			1&0&0&0&0&0\\
			1&1&0&1&0&0\\
			2&1&0&1&1&0\\
			3&1&1&0&0&0\\
			7&3&3&0&0&0\\
			7&3&3&1&3&1\\
			2&1&0&1&0&0\\
			4&2&0&2&1&0\\
			3&2&0&2&1&0\\
			4&2&1&1&0&0
		\end{array}\right]$}
	\begin{bsmallmatrix}
		s_k^{\text{\bf(P)}}(\text{IIb})\\s_k^{\text{\bf(P)}}(\text{Vb})\\s_k^{\text{\bf(P)}}(\text{VIb})\\s_k^{\text{\bf(P)}}(\text{VIc})\\s_k^{\text{\bf(P)}}(\text{XIb})\\s_k^{\text{\bf(P)}}(\text{Va}^*)\end{bsmallmatrix}.
\end{equation}
\end{proof}
For illustration we have listed $\dim S_k^{\text{\bf(P)}}(\Gamma)$ and $s_k^{\text{\bf(P)}}(\Omega)$ for weights $k\le 20$ in Tables~\ref{SkPlowweightstable}~and ~\ref{SkOmegalowweightstable}.
\subsection{General type}\label{Gsec}
In this section we will determine the numbers $s_k^{\text{\bf(G)}}(\Omega)$. As an application we obtain dimension formulas for the congruence subgroup $\Gamma_0'(4)$.
\begin{proposition}\label{dimSkGprop}
 With the possible exception of $\Gamma=\Gamma_0'(4)$, the generating series for $\dim S_k^{\text{\bf(G)}}(\Gamma)$ given in Table~\ref{dimSkGtable} hold.
\end{proposition}
\begin{proof}
 Recall from Proposition~\ref{MkSkdimprop} that the formulas given in Table~\ref{dimSktable} have been proven except for $\Gamma=\Gamma_0'(4)$. The formulas in Table~\ref{dimSkPtable} have been proven for all $\Gamma$; see Corollary~\ref{dimSkPcorollary}. In view of \eqref{SMFarthurdecomp2}, all we have to do is subtract the formulas in Table~\ref{dimSkPtable} from those in Table~\ref{dimSktable}.
\end{proof}

\begin{theorem}\label{skGtheorem}
 The generating series for the numbers $s_k^{\text{\bf(G)}}(\Omega)$ given in Table~\ref{skGOmegatable} hold. If a representation type $\Omega$ is not listed in Table~\ref{skGOmegatable}, then $s_k^{\text{\bf(G)}}(\Omega)=0$ for all $k$.
\end{theorem}
\begin{proof}
The \textbf{(G)} version of \eqref{mainmatrixeq1}, with appropriate columns combined and the row for $\Gamma_0'(4)$ omitted, is
\begin{equation}\label{dimskGeq}
	\begin{bsmallmatrix}\dim S_k^{\text{\bf(G)}}(\Gamma(2))\\\dim S_k^{\text{\bf(G)}}(\Sp(4,\Z))\\\dim S_k^{\text{\bf(G)}}(K(2))\\\dim S_k^{\text{\bf(G)}}(K(4))\\\dim S_k^{\text{\bf(G)}}(\Gamma_0(2))\\\dim S_k^{\text{\bf(G)}}(\Gamma_0(4))\\\dim S_k^{\text{\bf(G)}}(\Gamma_0^*(4))\\\dim S_k^{\text{\bf(G)}}(\Gamma_0'(2))\\\dim S_k^{\text{\bf(G)}}(M(4))\\\dim S_k^{\text{\bf(G)}}(B(2))\end{bsmallmatrix}
	=\resizebox{0.46\textwidth}{!}{$
		\left[
		\begin{array}{ccccccccccc} 	
		45&30&30&16&22&15&10&15&10&9\\
		1&0&0&0&0&0&0&0&0&0\\
		2&1&0&0&0&0&0&0&0&0\\
		4&2&1&0&1&0&0&1&0&0\\
		4&1&2&0&0&0&0&0&0&0\\
		12&5&8&2&2&4&3&1&1&0\\
		15&8&10&4&6&5&4&7&4&3\\
		4&2&1&0&1&0&0&0&0&0\\
		8&5&3&1&3&0&0&2&1&0\\
		8&4&4&1&2&0&0&0&0&0
	\end{array}\right]$}
	\begin{bsmallmatrix}s_k^{\text{\bf(G)}}(\text{I})\\s_k^{\text{\bf(G)}}(\text{IIa})\\s_k^{\text{\bf(G)}}(\text{IIIa+VIa/b})\\s_k^{\text{\bf(G)}}(\text{IVa})\\s_k^{\text{\bf(G)}}(\text{Va/a}^*)\\s_k^{\text{\bf(G)}}(\text{VII+VIIIa/b})\\s_k^{\text{\bf(G)}}(\text{IXa})\\s_k^{\text{\bf(G)}}(\text{X})\\s_k^{\text{\bf(G)}}(\text{XIa})\\s_k^{\text{\bf(G)}}(\text{sc(16)})\end{bsmallmatrix}.
\end{equation}

The $10\times10$ matrix is invertible, so that we can solve for the $s_k^{\text{\bf(G)}}(\Omega)$.
\end{proof}

\begin{corollary}\label{Klingen4corollary}
 For $\Gamma=\Gamma_0'(4)$, the results given in Tables~\ref{dimMktable}, \ref{dimSktable} and \ref{dimSkGtable} hold.
\end{corollary}
\begin{proof}
The row for $\Gamma_0'(4)$ in the \textbf{(G)} version of \eqref{mainmatrixeq1} is
\begin{equation}\label{Klingen4corollaryeq1}
 \dim S_k^{\text{\bf(G)}}(\Gamma_0'(4))
 =\resizebox{0.35\textwidth}{!}{$
 	\left[
 	\begin{array}{ccccccccccc} 		
  11&7&5&2&5&2&1&3&2&1
  \end{array}\right]$}
 \begin{bsmallmatrix}s_k^{\text{\bf(G)}}(\text{I})\\s_k^{\text{\bf(G)}}(\text{IIa})\\s_k^{\text{\bf(G)}}(\text{IIIa+VIa/b})\\s_k^{\text{\bf(G)}}(\text{IVa})\\s_k^{\text{\bf(G)}}(\text{Va/a}^*)\\s_k^{\text{\bf(G)}}(\text{VII+VIIIa/b})\\s_k^{\text{\bf(G)}}(\text{IXa})\\s_k^{\text{\bf(G)}}(\text{X})\\s_k^{\text{\bf(G)}}(\text{XIa})\\s_k^{\text{\bf(G)}}(\text{sc(16)})\end{bsmallmatrix}.
\end{equation}
The numbers on the right hand side are all known and given in Table~\ref{skGOmegatable}, allowing us to calculate $\dim S_k^{\text{\bf(G)}}(\Gamma_0'(4))$. Since $\dim S_k^{\text{\bf(P)}}(\Gamma_0'(4))$ is already known by Corollary~\ref{dimSkPcorollary}, we obtain $\dim S_k(\Gamma_0'(4))$ by \eqref{SMFarthurdecomp2}. We then obtain $\sum_{k=6}^\infty\dim M_k(\Gamma_0'(4))t^k$ using the codimensions from Table~\ref{codimensionstable}. Evidently $\dim M_0(\Gamma_0'(4))=1$, and $M_k(\Gamma_0'(4))=0$ for $k\in\{1,3,5\}$ by Theorem~\ref{oddweightcodimtheorem}. Finally $\dim M_k(\Gamma_0'(4))$ for $k\in\{2,4\}$ are determined in Appendix~\ref{CDappendix}.
\end{proof}
For illustration we have listed $\dim S_k^{\text{\bf(G)}}(\Gamma)$ and $s_k^{\text{\bf(G)}}(\Omega)$ for weights $k\le 20$ in Tables~\ref{SkGlowweightstable}~and ~\ref{SkOmegalowweightstable}.
\begin{appendix}
\section{Modular forms of Klingen level \texorpdfstring{$4$}{} and small weight}\label{CDappendix}
by Cris Poor and David S.\ Yuen
\subsection{Introduction and notation\label{sectionIntr}}

This appendix proves $\dim M_4(\KLfour)=4$ and~$\dim M_2(\KLfour)=0$.  
The proof proceeds by getting upper and lower bounds that agree.  
The proofs of the upper bounds rely on the known dimensions 
$\dim M_8(\KLfour)=12$ and~$\dim M_8(\middlefour)=8$.  
Proving the nontrivial lower bound relies on constructing Gritsenko lifts of linearly independent Jacobi-Eisenstein series.  
\smallskip

Let $\Jac km$ denote the space of Jacobi forms of weight~$k$ and index~$m$ 
on~$\SL(2,\Z)$, see~\cite{ez85} for definitions.  
Jacobi forms of index zero are identified 
with elliptic modular forms, $\Jac k0=M_k\left( \SL(2,\Z) \right)$, and we will need the 
Eisenstein series $G_k \in  M_k\left( \SL(2,\Z) \right)$ for even~$k \ge 4$, 
$$
G_k(\tau) = \frac12 \zeta(1-k) + \sum_{n \ge 1} \sigma_{k-1}(n) q^n,
$$
for $\tau \in \UHP_1$ and $q=e(\tau)=e^{2 \pi i \tau}$.  For $z \in \C$ and $\zzeta=e(z)$, let 
$$
\Ekone(\tau,z) = \sum_{n,r \in \Z:\, n, 4n-r^2 \ge0}
\dfrac{H(k-1,4n-r^2)}{H(k-1,0)} q^n \zzeta^r \in \Jac k1
$$
be the Jacobi-Eisenstein series of even weight~$k\ge 4$ and index one from 
Eichler--Zagier \cite{ez85}, page~{22}.    
Here the Cohen numbers~$H(r,n)$ for $n \ge 0, r \ge2$  are directly computed as $H(r,0)=\zeta(1-2r)$; 
$H(r,n)=0$ for  $n \in \N$ such that $(-1)^r n \equiv 2,3 \bmod 4$; and, for
 $n \in \N$ such that $(-1)^r n \equiv 0,1 \bmod 4$, as 
   \begin{align*} 
 H(r,n)
&=   {L(1-r, \chi_{D})} \sum_{d | f} \mu(d)  \chi_{D}(d) d^{r-1} 
  \sigma_{2r-1}\left({f}/{d} \right), 
  \end{align*}
where $D$ is the fundamental discriminant of $\Q( \sqrt{ (-1)^r n } )$, 
$(-1)^r  n= D f^2$ for  $f  \in \N$, 
$\mu$ is the M\"{o}bius  function, 
and $\chi_D: \Z  \to \{-1, 0, 1\}$ is the Kronecker symbol.   The $L$-function $L(s,\chi_{D})=\sum_{n \in \N} \frac{\chi_D(n)}{n^s}$ 
is defined by analytic continuation, and its special values are given by 
twisted Bernoulli numbers $B_{k, \chi_D} $  
\begin{align*}
 &\sum_{a=1}^{ |D| } \chi_D(a) \dfrac{ te^{at} }{ e^{|D| t} -1 }
  = \sum_{k=0}^{\infty} B_{k, \chi_D} \frac{t^k}{k!}
   \qquad \text{(page 53, \cite{MR3307736})},\\
 &L(1-k, \chi_D)  = -\dfrac{ B_{k, \chi_D} }{k},  \text{ for $k \in \N$   } 
  \qquad \text{(page 152, \cite{MR3307736})}.
\end{align*}

For $\ell \in \N$, let $V_{\ell}: \Jac km \to \Jac k{m\ell}$ and 
$U_{\ell}: \Jac km \to \Jac k{m\ell^2}$ be the commuting family of index~raising operators 
from~\cite{ez85}, page~{41}.  In particular, for $\phi \in \Jac km$, 
\begin{align*}
\left( \phi \vert V_2 \right)(\tau,z) &= 
2^{k-1} \phi(2\tau,2z) + \frac12 
\left( \phi\left( \frac{\tau+1}{2},z \right)+\phi\left( \frac{\tau}{2},z \right) \right),  \\
\left( \phi \vert U_2 \right)(\tau,z) &= 
2^k\, \phi(\tau,2z).  
\end{align*}
The lower bound $\dim M_4\left( \KLfour \right) \ge 4$ will by proven by 
constructing four linearly independent Gritsenko lifts of Jacobi forms.  
Enough information has already been presented to define the Jacobi forms 
that we will need and to compute their initial Fourier expansions.  
We have normalized their constant terms to be~$1$. 
\begin{align*}
	&\varphi_0 = \Efone = 1 + 
	\left(\zzeta^2+56\zzeta+126+56\zzeta\inv+\zzeta^{-2}\right) q  \\
	\,\,&+\left(126\zzeta^2+576\zzeta+756+576\zzeta\inv+126\zzeta^{-2}\right) q^2 + \cdots \\
	&\varphi_1 = \frac1{16}\Efone \vert U_2 = 1 + 
	\left(\zzeta^4+56\zzeta^2+126+56\zzeta^{-2}+\zzeta^{-4}\right) q  \\
	&\,\,+\left(126\zzeta^4+576\zzeta^2+756+576\zzeta^{-2}+126\zzeta^{-4}\right) q^2 + \cdots \\
	&\varphi_2 = \frac1{9}\Efone \vert V_2 = 1 + 
	\left(14\zzeta^2+64\zzeta+84+64\zzeta^{-1}+14\zzeta^{-2}\right) q  \\
	&\,\,+\left(\zzeta^4+64\zzeta^3+280\zzeta^2+448\zzeta + 574+448\zzeta^{-1}+ \cdots +\zzeta^{-4}\right) q^2 + \cdots \\
	&\varphi_3 = \frac1{81}\Efone \vert V_2\vert V_2 = 1  
	+\left(\frac19\zzeta^4+\frac{64}{9}\zzeta^3+\frac{280}{9}\zzeta^2+\frac{448}{9}\zzeta+\frac{574}{9}+\frac{448}{9}\zzeta^{-1}+ \cdots+\frac19\zzeta^{-4}\right) q\,  \\
+	&\left(\frac{64}{9}\zzeta^5+\frac{686}{9}\zzeta^4+\frac{448}{3}\zzeta^3+320\zzeta^2+\frac{896}{3}\zzeta + \frac{1372}{3}+\frac{896}{3}\zzeta^{-1}+ \cdots +\frac{64}{9}\zzeta^{-5}\right) q^2 + \cdots \\
\end{align*}

\subsection{Proofs\label{sectionproof}}

Let $\Jmero km$ be the $\C$-vector space of meromorphic functions on $\UHP_1\times\C$ 
spanned by~$a/b$ such that $a \in \Jac {k_1}{m_1}$, $b \in \Jac {k_2}{m_2}\setminus\{0\}$, and $k_1-k_2=k$, $m_1-m_2=m$;  
define $M^{\rm mero}_k(\Gamma)$  similarly.  
\begin{lemma}
\label{lemmaone}
A basis for $\Jac 40$ is $G_4$.  
A basis for $\Jac 41$ is $\varphi_0$. 
A basis for $\Jac 42$ is $\varphi_2$. 
A basis for $\Jac 44$ is $\varphi_1, \varphi_3$. 
We have $\varphi_2^2/G_4 \in \Jmero 44 \setminus \Jac 44$.  
\end{lemma}
\begin{proof}
From \cite{ez85}, pages 103-105, we have $\dim \Jac 4m = 1,1,2$ for $m=1,2,4$.  
The Fourier expansions show the linear independence.  
Assume~$\varphi_2^2/G_4 \in \Jac 44$,   
then its Fourier expansion would be given by the  quotient of the series 
for~$\varphi_2^2$ by $G_4=\frac1{240}+q+9q^2+28q^3+\cdots$.  
The formal series for $\varphi_2^2/(240G_4)$ begins: 
\begin{align*}
 1 +& 
\left(28\zzeta^2+128\zzeta-72+128\zzeta^{-1}+28\zzeta^{-2}\right) q + 
(198\zzeta^4+1920\zzeta^3+288\zzeta^2  \\ 
&-17280\zzeta + 31908-17280\zzeta^{-1}+ \cdots +198\zzeta^{-4}) q^2 + \cdots. 
\end{align*}
However, by the Fourier expansions, this is not in the span of $\varphi_1$ and~$\varphi_3$.  
\end{proof}

{\bf Remark.\/}  
Let 
$\vartheta(\tau,z)=\sum_{n \in \Z} (-1)^n\, q^{(2n+1)^2/8} \zzeta^{(2n+1)/2}$ 
define the odd Jacobi theta function; then $\vartheta^8 \in \Jac 44$, and we may check 
$\vartheta^8=\frac98 \left(\varphi_1-\varphi_3\right) $.  
\smallskip

Each Siegel modular form $ f \in M_k\left( \KLN \right)$ has a unique Fourier-Jacobi expansion
\begin{equation}
 f\smallpmat{\tau}{z}{z}{\omega} = \sum_{m=0}^{\infty} \phi_m(\tau,z) e(m\omega)
\end{equation}
for which $\phi_m \in \Jac km$; 
to see this use $\Gamma_0'(N) \cap Q= \Sp(4,\Z) \cap Q$ and Theorem~{6.1} in~\cite{ez85}.  
Setting $\xi=e(\omega)$, we write this more briefly as 
$f = \sum_{m=0}^{\infty} \phi_m \xi^m$.  
The following theorem~\cite{MR1345176} will allow us to obtain paramodular forms as 
Gritsenko lifts of the Jacobi forms~$\varphi_j$.  
In this theorem, $c(0,0;\phi)$ is the constant term of the Fourier expansion of the Jacobi form~$\phi$.  

\begin{theorem}\label{gritlift}
Let $k,N \in \N$.  For $\phi \in \Jac kN$, we have $c(0,0;\phi)=0$ unless $k \ge 4$ is even. 
An injective linear map 
$\Grit: \Jac kN \to M_k\left( K(N)  \right)$ is defined by 
$$
\Grit(\phi) = c(0,0;\phi)G_k + \sum_{m \in \N}  \phi \vert V_m\, \xi^{Nm}. 
$$ 
\end{theorem}

\begin{definition}
Let  $g_j=\Grit(\varphi_j) \in M_4\left( K(N_j)  \right)$ 
for $N_0=1$, $N_1=4$, $N_2=2$, and~$N_3=4$.  
\end{definition}

\begin{lemma}
	\label{lemmatwo}
	The elements $g_0,g_1,g_2,g_3 \in M_4\left( \KLfour \right)$ are linearly independent.  
	Each element of $\Span(  g_0,g_1,g_2,g_3)$ is determined by its Fourier-Jacobi coefficients through index~$4$. 
	The elements $g_1,g_2,g_3 \in M_4\left( \middlefour \right)$ are linearly independent.  
\end{lemma}
\begin{proof}
	Since $\KLfour \subseteq K(N)$ for $N=1,2,4$, we have 
	$g_0,g_1,g_2,g_3 \in M_4\left( \KLfour \right)$.  
	By Theorem~{\ref{gritlift}}, 
	their Fourier-Jacobi expansions through index four are 
	\begin{align}
		\label{sammy}
		g_0&=G_4+ \varphi_0 \xi + 9 \varphi_2 \xi^2 + \varphi_0 \vert V_3 \xi^3 + \varphi_0 \vert V_4 \xi^4 +\cdots \notag \\
		g_1&=G_4+ 0 \xi + 0 \xi^2 + 0 \xi^3 + \varphi_1 \xi^4 +\cdots \\
		g_2&=G_4+ 0 \xi + \varphi_2 \xi^2 + 0 \xi^3 + 9\varphi_3 \xi^4 +\cdots \notag\\
		g_3&=G_4+ 0 \xi + 0 \xi^2 + 0 \xi^3 + \varphi_3 \xi^4 +\cdots. \notag
	\end{align}
	We first show that the subspace of $\Span(  g_0,g_1,g_2,g_3)$ whose 
	Fourier-Jacobi coefficients of index~$0$, $1$, and~$2$ vanish is the one-dimensional space spanned by~$g_1-g_3$; this subspace defined by vanishing conditions is well-defined because Fourier-Jacobi expansions are unique.  
	Let $f= \sum_j c_j g_j \in \Span(  g_0,g_1,g_2,g_3)$ for $c_j \in \C$.  
	The vanishing of the Jacobi coefficient of index zero gives $c_0+c_1+c_2+c_3=0$; 
	of index one, $c_0=0$; and of index two, $9c_0+c_2=0$.  Hence we have $c_0=c_2=0$,  
	$c_3=-c_1$, and $f=\sum_j c_j g_j=c_1(g_1-g_3)=c_1 (\varphi_1-\varphi_3) \xi^4 + \cdots$.  
	Since $\varphi_1$ and~$\varphi_3$ are linearly independent by Lemma~\ref{lemmaone}, 
	if we additionally demand that the fourth Jacobi coefficient~$c_1(\varphi_1-\varphi_3)$ 
	vanishes then $c_1=0$ and $f=0$.  
	Hence $\Span(  g_0,g_1,g_2,g_3)$ is determined by the Fourier-Jacobi coefficients through index~$4$.
	
	On the other hand, $f=0$ implies that the Fourier-Jacobi expansion of $f$ vanishes though
	index~$4$, so that $c_0+c_1+c_2+c_3=0$, $c_0=0$, $9c_0+c_2=0$, and~$c_1=0$, 
	implying $c_0=c_1=c_2=c_3=0$.  Thus the $g_j$ are linearly independent. 
	Since  $\middlefour \subseteq K(N)$ for $N=2,4$, we have 
	$g_1,g_2,g_3 \in M_4\left( \middlefour \right)$. 
	We have already seen their linear independence.   
\end{proof}

{\bf Remark.\/}  
In terms of the global paramodular newform theory of~\cite{robertsschmidt06}, 
the level one Eisenstein series $g_0$ is a newform for $K(1)=\Sp(4,\Z)$, 
and $g_2$ is the oldform above~$g_0$ in~$K(2)$, 
and $g_1,g_3$ are the oldforms above~$g_0$ in~$K(4)$.

\begin{lemma}
	\label{lemmathree}
	Products from $M_4\left( \middlefour \right)$ span a $6$-dimensional space in   
	$M_8\left( \middlefour \right)$.  
\end{lemma}
\begin{proof}
	Since $\dim M_4\left( \middlefour \right)=3$, Lemma~\ref{lemmatwo} shows that 
	$g_1,g_2,g_3$ is a basis.  Thus we need  to show that the six 
	products $g_1^2$, $g_1g_2$, $g_1g_3$, $g_2^2$, $g_2g_3$, and~$g_3^2$ are 
	linearly independent in $M_8\left( \middlefour \right)$. 
	Multiplying the expansions~{(\ref{sammy})}, their Fourier-Jacobi expansions through index four are 
	\begin{align*}
		g_1^2&=G_4^2 + 2G_4\varphi_1 \xi^4 +\cdots \\
		g_1g_2&=G_4^2+  G_4 \varphi_2 \xi^2  + (9G_4\varphi_3 +G_4 \varphi_1)\xi^4 +\cdots \\
		g_1g_3&=G_4^2  + (G_4\varphi_3 +G_4 \varphi_1)\xi^4 +\cdots \\
		g_2^2&=G_4^2 + 2G_4\varphi_2 \xi^2 + (18G_4\varphi_3 +\varphi_2^2) \xi^4 +\cdots \\
		g_2g_3&=G_4^2+  G_4 \varphi_2 \xi^2  + 10G_4\varphi_3 \xi^4 +\cdots \\
		g_3^2&=G_4^2 + 2G_4\varphi_3 \xi^4 +\cdots .
	\end{align*}
	To prove linear independence, 
	let  $\sum_{1 \le i \le j \le 3} c_{ij} g_ig_j=0$ for some $c_{ij} \in \C$.  
Using the linear independence of $\varphi_1$, $\varphi_3$, and~$\varphi_2^2/G_4$ from Lemma~\ref{lemmaone},  
the vanishing of the Fourier-Jacobi coefficients of indices~$0$, $2$, and~$4$, implies 
	$c_{11}=c_{23}+c_{33}$, $c_{12}=-c_{23}$, $c_{13}=-c_{23}-2c_{33}$, and $c_{22}=0$, and   
	$$
	\sum_{1 \le i \le j \le 3} c_{ij}^{\,} g_ig_j=
	(g_1-g_3)\left( c_{23}^{\,}(g_1-g_2)+c_{33}^{\,}(g_1-g_3)  \right).
	$$
	By Lemma~\ref{lemmatwo}, 
	$g_1-g_2$ and $g_1-g_3$ are linearly independent, so 
	we obtain $c_{23}=c_{33}=0$.  
\end{proof}

The following proof is the most interesting. 
It leverages dimensions of spaces of higher weight to deduce 
the dimension of a space of lower weight.

\begin{proposition}
	\label{mainprop}
	We have $\dim M_4\left( \KLfour \right)=4$.  
\end{proposition}
\begin{proof}
	We know $g_0,g_1,g_2,g_3 \in M_4\left( \KLfour \right)$ are linearly independent by Lemma~\ref{lemmatwo}.  
	Suppose by way of contradiction that an $f \in M_4\left( \KLfour \right)$ exists with 
	$f,g_0,g_1,g_2,g_3$ linearly independent.  
	Using the known dimension $\dim M_8\left( \middlefour \right)=8$, 
	let $h_1, \ldots, h_8$ be a basis of $ M_8\left( \middlefour \right)$.  
	The~$14$ elements $fg_1,fg_2,fg_3, g_0g_1,g_0g_2,g_0g_3, h_1, \ldots, h_8$,  
	are in $M_8\left( \KLfour \right)$, which is known to be $12$-dimensional, 
	so that there must be at least two linearly independent relations
	\begin{equation}
		\label{ranktwo}
		fg_4+g_0g_5+h_0=0; \qquad fg_4'+g_0g_5'+h_0'=0,
	\end{equation}
	for some $g_4,g_5,g_4',g_5' \in \Span(g_1,g_2,g_3)=M_4\left( \middlefour \right)$ 
	and some $h_0,h_0' \in M_8\left( \middlefour \right)$.  
	\smallskip
	
	We will show that $g_4$ is not identically zero.  If $g_4$ and $g_5$ were both trivial then 
	$h_0$ would also be trivial, contradicting that the relations~{(\ref{ranktwo})} have rank two. 
	If $g_4\equiv 0$ and $g_5 \not\equiv 0$, then 
	$$
	g_0= -\frac{h_0}{g_5} \in M_4^{\rm mero}\left( \middlefour  \right) \cap M_4\left( \KLfour \right)
	=M_4\left( \middlefour \right) = \Span(g_1,g_2,g_3), 
	$$
	contradicting the linear independence of $g_0,g_1,g_2,g_3$.  
	We will refer to this argument, that a meromorphic form for $\middlefour$ that is also a holomorphic form 
	for $\KLfour$ must be a holomorphic form for $\middlefour$, as the integral closure argument.  
The principle is general: 
For congruence subgroups $\Gamma_1 \subseteq \Gamma_2$ of~$\Sp(2n,\Q)$, 
we have $M_k^{\rm mero}(\Gamma_2) \cap M_k(\Gamma_1) = M_k(\Gamma_2)$.  
To prove this take $ f \in M_k^{\rm mero}(\Gamma_2) \cap M_k(\Gamma_1)$ and $\gamma \in \Gamma_2$.  
We have $f\vert\gamma =f$ on some dense open subset of~$\UHP_n$.  
Since $f \in M_k(\Gamma_1)$ is holomorphic we have $f\vert\gamma =f$ on~$\UHP_n$ and $f \in M_k(\Gamma_2)$.  
	Similarly to~$g_4$, we have $g_4'$ not identically zero.  Since $f,g_1,g_2,g_3$ are linearly independent, 
	the integral closure argument also shows that $g_5,g_5'$ are not identically zero.  
	\smallskip
	
	We will show that $g_4g_5'-g_4'g_5$ is identically zero. If not then 
	$$
	\begin{pmatrix} f \\ g_0 \end{pmatrix} 
	= \dfrac{1}{g_4g_5'-g_4'g_5}  
	\begin{pmatrix} g_5'  &  -g_5 \\ -g_4'  &  g_4 \end{pmatrix} 
	\begin{pmatrix} -h_0 \\ -h_0' \end{pmatrix} 
	\in M_4\left( \middlefour \right) \times M_4\left( \middlefour \right)
	$$
	by the integral closure argument.  However, the linear dependence of $f$ and~$g_0$ on 
	$g_1,g_2,g_3$  contradicts the assumption that  
	$f,g_0,g_1,g_2,g_3$ are linearly independent.
	\smallskip
	
	We use Lemma~{\ref{lemmathree}}.  
	Since $g_4g_5'=g_4'g_5$ for nontrivial 
	$g_4,g_5,g_4',g_5' \in M_4\left( \middlefour \right)$ and the six products
	$g_ig_j$, for $1 \le i \le j \le 3$, are linearly independent, it follows that 
	there exists a unit $\alpha \in \C^{*}$ such that 
	$g_4=\alpha g_4'$ and $g_5 = \alpha g_5'$, {\em or\/}  
	$g_4=\alpha g_5$ and $g_4' = \alpha g_5'$.  
	In the first case the two linear relations~{(\ref{ranktwo})} become
	\begin{equation*}
		\alpha fg_4'+ \alpha g_0g_5'+h_0=0; \qquad  fg_4'+g_0g_5'+h_0'=0,
	\end{equation*}
	so that $h_0=\alpha h_0'$ and the relations are not linearly independent.  
	In the second case we obtain
	\begin{equation*}
		\alpha fg_5+  g_0g_5+h_0=0; \qquad \alpha fg_5'+g_0g_5'+h_0'=0,
	\end{equation*}
	and $\alpha f = -g_0 - h_0/g_5 \in M_4\left( \middlefour \right)$ 
	by the integral closure argument, contradicting the linear independence of 
	$f,g_1,g_2,g_3$.  
	Thus no~$f\in  M_4\left( \KLfour \right)$ with  $f,g_0,g_1,g_2,g_3$ linearly independent can exist.  
\end{proof}

\begin{proposition}
	\label{minorprop}
	We have $\dim M_2\left( \KLfour \right)=0$.  
\end{proposition}
\begin{proof}
	Take $f \in M_2\left( \KLfour \right)$ with Fourier-Jacobi expansion  $f = \sum _{m=0}^{\infty}  \phi_m \xi^m$
	for $\phi_m \in \Jac 2m$.  We have $\dim \Jac 2m \le 0$ for $m \le 2$ 
	by the corollary on page~{103} of~\cite{ez85}.  
	Therefore $f = \sum _{m=3}^{\infty}  \phi_m \xi^m$ has order at least index~$3$ 
	and $f^2 \in M_4\left( \KLfour \right)=\Span(g_0,g_1,g_2,g_3)$ has order at least index~$6$.  
	By Lemma~\ref{lemmatwo} this span is determined by the Fourier-Jacobi coefficients of 
	index through~$4$; thus $f^2=0$ and $f=0$.  
\end{proof}
\pagebreak
\section{Tables}\label{TablesAppendix}
\subsection{History of dimension formulas}

\vspace{-3ex}
\begin{table}[hbt!]
		\caption{History of dimension formulas for $M_k(\Gamma)$ and $S_k(\Gamma)$ for some $\Gamma$. Earlier references appear left of later (relevant) references in the reference column.
		}
		\label{historytable}
		$$
		\renewcommand{\arraystretch}{1}
		\renewcommand{\arraycolsep}{.7cm}
		\begin{array}{ccl}
			\toprule
			\Gamma&\text{weight}&\text{Reference}\\
			\midrule
			\Sp(4, \Z)&k\geq 0 &\text{\cite[Thm.~2]{Igusa1964}},\ \text{\cite[Thm.~6-2]{Hashimoto1983}}\\
			\midrule
			\Gamma(2) &k\geq 0&\text{\cite[Thm.~2]{Igusa1964}},\ \text{\cite[p.~882]{Tsushima1982}}\\
			\midrule
			K(2)&k=1&\text{\cite[Thm.~6.1]{Ibukiyama2007}}\\
			\cmidrule{2-3}
			&k=2&\text{\cite[Sect.~1]{Ibukiyama1984}}\\
			\cmidrule{2-3}
			&k=3&\text{\cite[Thm.~2.1]{Ibukiyama2007}}\\
			\cmidrule{2-3}
			&k=4&\text{\cite[Sect.~2.4]{Ibukiyama2007}}\\
			\cmidrule{2-3}
			&k\geq5&\text{\cite[Thm.~4]{Ibukiyama1985}}\\
			\midrule
			\Gamma_0(2)&k=1&\text{\cite[Thm.~6.1]{Ibukiyama2007}}\\
			\cmidrule{2-3}
			&k=2&\text{\cite[Sect.~1]{Ibukiyama1984}}\\
			\cmidrule{2-3}
			&k=3&\text{\cite[Thm.~2.2]{Ibukiyama2007}}\\
			\cmidrule{2-3}
			&k=4&\text{\cite[Cor.~4.12]{Tsushima1997}},\ \text{\cite[Sect.~2.4]{Ibukiyama2007}}\\
			\cmidrule{2-3}
			&k\geq5&\text{\cite{Ibukiyamathesis},\ \cite[Cor.~4.12]{Tsushima1997}},\ \text{\cite[Thm.~7.4]{Wakatsuki2012}}\\
			\midrule
			\Gamma_0'(2)&k=1&\text{\cite[Thm.~6.1]{Ibukiyama2007}}\\
			\cmidrule{2-3}
			&k=2&\text{\cite[Sect.~1]{Ibukiyama1984}}\\
			\cmidrule{2-3}
			&k=3&\text{\cite[Thm.~2.4]{Ibukiyama2007}}\\
			\cmidrule{2-3}
			&k=4&\text{\cite[Sect.~2.4]{Ibukiyama2007}}\\
			\cmidrule{2-3}
			&k\geq5&\text{\cite{Ibukiyamathesis},\ \cite[Thm.~A.1]{Wakatsuki2013}}\\
			\midrule
			\B(2)&k=1&\text{\cite[Thm.~6.1]{Ibukiyama2007}}\\
			\cmidrule{2-3}
			&k=2&\text{\cite[Sect.~1]{Ibukiyama1984}}\\
			\cmidrule{2-3}
			&k=3&\text{\cite[Thm.~2.3]{Ibukiyama2007}}\\
			\cmidrule{2-3}
			&k=4&\text{\cite[Sect.~2.4]{Ibukiyama2007}}\\
			\cmidrule{2-3}
			&k\geq5&\text{\cite{Ibukiyamathesis},\ \cite[Thm.~A.2]{Wakatsuki2013}}\\   \midrule
			K(4)&k\geq 0&\text{\cite[Thm.~1.1]{PoorYuen2013}}\\
			\midrule
			\Gamma_0(4)&k\geq 0&\text{\cite[Prop.~5.4]{Tsushima2003}}\\
			\cmidrule{2-3}
			&k\geq 5&\text{\cite[Thm.~3.5]{Shukla2017}}\\\midrule
			\Gamma_0^*(4),\:\Gamma_0'(4),\,M(4)&k\geq 0&\text{Tables~\ref{dimMktable} and \ref{dimSktable}}\\
			\bottomrule
		\end{array}
		$$
\end{table}
\clearpage
\subsection{Dimension formulas for all weights}
\begin{table}[hbt!]
 \caption{Dimension formulas for $M_k(\Gamma)$. The second column indicates those cases that follow directly from \cite[Thm.~2]{Igusa1964}. The last column gives references for some other places where these formulas appear in the literature.}
 \label{dimMktable}
\begin{equation*}\setlength{\arraycolsep}{0.7ex}\renewcommand{\arraystretch}{1.9}
 \begin{array}{ccccccccccccccccccccc}
  \toprule
   \Gamma&\!\!\text{Igusa}\!\!\!\!&\sum_{k=0}^\infty\dim M_k(\Gamma)t^k&\text{reference}\\
  \toprule
   \Gamma(2)&\bullet&\frac{(1+t^2)(1+t^4)(1+t^5)}{(1-t^2)^4}&\text{\cite[p.~883]{Tsushima1982}}\\
  \midrule
   \!\!\!\Sp(4,\Z)\!\!\!&\bullet&\frac{1+t^{35}}{(1-t^4)(1-t^6)(1-t^{10})(1-t^{12})}&\text{\cite[p.~402]{Igusa1964}}\\
  \midrule
   K(2)&&\frac{(1+t^{10})(1+t^{12})(1+t^{11})}{(1-t^4)(1-t^6)(1-t^8)(1-t^{12})}&\text{\cite[Prop.~2]{IbukiyamaOnodera1997}}\\
  \midrule
   K(4)&\bullet&\frac{(1+t^{12})(1+t^6+t^7+t^8+t^9+t^{10}+t^{11}+t^{17})}{(1-t^4)^2(1-t^6)(1-t^{12})}&\text{\cite[p.~121]{IbukiyamaPoorYuen2013}}\\
  \midrule
   \Gamma_0(2)&\bullet&\frac{1+t^{19}}{(1-t^2)(1-t^4)^2(1-t^6)}&\!\!\!\!\text{\cite[Thm.~A,C]{Ibukiyama1991}}\\
  \midrule
   \Gamma_0(4)&\bullet&\frac{1+t^4+t^{11}+t^{15}}{(1-t^2)^3(1-t^6)}&\!\!\!\!\text{\cite[Prop.~5.4]{Tsushima2003}}\\
  \midrule
   \Gamma_0^*(4)&\bullet&\frac{(1+t^4+t^6+t^{10})(1+t^5)}{(1-t^2)^3(1-t^6)}\\
  \midrule
   \Gamma_0'(2)&\bullet&\frac{(1+t^6+t^8+t^{10}+t^{12}+t^{18})(1+t^{11})}{(1-t^4)^2(1-t^6)(1-t^{12})}&\text{\cite[Prop.~2]{IbukiyamaOnodera1997}}\\
  \midrule
   \Gamma_0'(4)&&\frac{1 + 2 t^4 + 4 t^6 + t^7 + 5 t^8 + 2 t^9 + 4 t^{10} + 5 t^{11} + 5 t^{12} + 4 t^{13} + 2 t^{14} + 5 t^{15} + t^{16} + 4 t^{17} + 2 t^{19} + t^{23}}{(1-t^4)^2(1-t^6)^2}\\
  \midrule
   M(4)&\bullet&\frac{(1+t^4)(1+2t^6+t^7+3t^8+t^9+t^{10}+2t^{11}+t^{12}+t^{13}+2t^{14}+t^{15}+t^{16}+3t^{17}+t^{18}+2t^{19}+t^{25})}{(1-t^4)^2(1-t^6)(1-t^{12})}\\ \midrule
   B(2)&\bullet&\frac{(1+t^6)(1+t^{11})}{(1-t^2)(1-t^4)^3}&\!\!\!\!\text{\cite[Thm.~B,C]{Ibukiyama1991}}\\
  \bottomrule
 \end{array}
\end{equation*}
\end{table}

We remark that all the numerator polynomials in Table~\ref{dimMktable} are palindromic. By Theorem~4.4 of \cite{Stanley1978}, this is related to the graded algebra $\bigoplus_{k\geq0} M_k(\Gamma)$ being a Gorenstein ring. (For the question of being Cohen-Macaulay, see \cite{Eichler1971}, \cite{Tsuyumine1988a}, \cite{Tsuyumine1988b}.)

\begin{table}[hbt!]
 \caption{Dimension formulas for $S_k(\Gamma)$. The second column indicates those cases that follow directly from \cite[Thm.~2]{Igusa1964}, together with the codimension formulas given in Table~\ref{codimensionstable}. The last column gives references for some other places where these formulas appear in the literature.}
 \label{dimSktable}
\begin{equation*}\setlength{\arraycolsep}{0.7ex}\renewcommand{\arraystretch}{1.9}
 \begin{array}{ccccccccccccccccccccc}
  \toprule
   \Gamma&\!\!\text{Igusa}\!\!\!\!&\sum_{k=0}^\infty\dim S_k(\Gamma)t^k&\text{reference}\\
  \toprule
   \Gamma(2)&\bullet&\frac{t^5(1+5t+t^2+4t^3+t^4-5t^5+t^6)}{(1-t^2)^4}&\text{\cite[p.~882]{Tsushima1982}}\\
  \midrule
   \!\!\!\Sp(4,\Z)\!\!\!&\bullet&\frac{1+t^{35}}{(1-t^4)(1-t^6)(1-t^{10})(1-t^{12})}-\frac1{(1-t^4)(1-t^6)}&\text{\cite[Thm.~3]{Igusa1964}}\\
  \midrule
   K(2)&&\frac{t^8(1+t^{12})(1+t^2+t^3+t^4-t^{12}+t^{13})}{(1-t^4)(1-t^6)(1-t^8)(1-t^{12})}&\text{\cite[Thm.~4]{Ibukiyama1985}}\\
  \midrule
   K(4)&\bullet&\frac{t^7(1+t+t^2+2t^3+t^4+2t^5+t^9+t^{10}+2t^{11}+t^{12}+t^{13}+t^{14}+t^{16}-t^{21}+t^{22})}{(1-t^4)^2(1-t^6)(1-t^{12})}&\text{\cite[Thm.~2]{PoorYuen2013}}\\
  \midrule
   \Gamma_0(2)&\bullet&\frac{t^6(1+t^2-t^8+t^{13})}{(1-t^2)(1-t^4)^2(1-t^6)}&\text{\cite[Thm.~A,C]{Ibukiyama1991}}\\
  \midrule
   \Gamma_0(4)&\bullet&\frac{t^6(3+t^4+t^5-2t^6+t^9)}{(1-t^2)^3(1-t^6)}\\
  \midrule
   \Gamma_0^*(4)&\bullet&\frac{t^5(1+3t+t^3+t^4+2t^5+t^6-t^7-t^9+t^{10})}{(1-t^2)^3(1-t^6)}\\
  \midrule
   \Gamma_0'(2)&\bullet&\frac{t^8(1+t^2+t^3+t^4-t^5-t^6+t^7+2t^8+t^{11}-t^{14}+t^{15}+t^{16}-t^{17}-t^{18}+t^{19})}{(1-t^2)(1-t^4)(1-t^6)(1-t^{12})}&\text{\cite[Thm.~A.1]{Wakatsuki2013}}\\
  \midrule
   \Gamma_0'(4)&&\frac{t^7(1+3t+2t^2+9t^3+5t^4+13t^5+4t^6+6t^7+5t^8+4t^{10}-3t^{11}+2t^{12}-2t^{13}-2t^{15}+t^{16})}{(1-t^4)^2(1-t^6)^2}\\
  \midrule
   M(4)&\bullet&\frac{t^7(1+2t+2t^3+3t^4+4t^5-t^6+4t^8+5t^9+3t^{12}+2t^{13}-2t^{15}+2t^{16}+t^{17}-t^{18}-2t^{19}+t^{20})}{(1-t^2)(1-t^4)(1-t^6)(1-t^{12})}&\\
  \midrule
   B(2)&\bullet&\frac{t^6(1+t^2-t^4+t^5+t^6-t^7-t^8+t^9)}{(1-t^2)^2(1-t^4)^2}&\text{\cite[Thm.~B,C]{Ibukiyama1991}}\\
  \bottomrule
 \end{array}
\end{equation*}
\end{table}
\begin{table}
 \caption{Dimension formulas for cusp forms of Saito-Kurokawa type}\label{dimSkPtable}
 \begin{equation*}
  \renewcommand{\arraystretch}{1.9}
  \begin{array}{cc} 
   \toprule
    \Gamma&\displaystyle\sum_{k=0}^\infty\dim S_k^{\text{\bf(P)}}(\Gamma)t^k\\
   \toprule
    \Gamma(2)&\dfrac{t^5 (1 + t + t^2)(1 + 4 t + 10 t^3 - 5 t^4 + 10 t^5)}{(1-t^4)(1-t^6)}\\
   \midrule
    \Sp(4,\Z)&\dfrac{t^{10}}{(1-t^2)(1-t^{6})}\\
   \midrule
    \mathrm{K}(2)&\dfrac{t^8(1+t^2+t^3+t^4)}{(1-t^4)(1-t^6)}\\
   \midrule
    \mathrm{K}(4)&\dfrac{t^7(1 + t + t^2 + 2 t^3 + t^4 + 2 t^5)}{(1-t^4)(1-t^6)}\\
   \midrule
    \Gamma_0(2)&\dfrac{t^6(1+t^2+2t^4)}{(1-t^2)(1-t^6)}\\
   \midrule
    \Gamma_0(4)&\dfrac{t^6(3+3t^2+4t^4)}{(1-t^2)(1-t^6)}\\
   \midrule
    \Gamma_0^*(4)&\dfrac{t^5(1-t+t^2)(1 + 4t + 5t^2 + 4t^3)}{(1-t^2)(1-t^6)}\\
   \midrule
    \Gamma_0'(2)&\dfrac{t^8(1+t+t^2)(1-t+2t^2)}{(1-t^4)(1-t^6)}\\
   \midrule
    \Gamma_0'(4)&\dfrac{t^7(1 + 2 t + t^2 + 4 t^3 + 2 t^4 + 4 t^5)}{(1-t^4)(1-t^6)}\\
   \midrule
    M(4)&\dfrac{t^7(1+t+t^2)(1 + t - t^2 + 3 t^3)}{(1-t^4)(1-t^6)}\\
   \midrule	
    B(2)&\dfrac{t^6(1+t+t^2)(1 - t + 3 t^2 - 2t^3 + 3 t^4)}{(1-t^4)(1-t^6)}\\
   \bottomrule
  \end{array}
 \end{equation*}
\end{table}
\begin{table}
 \caption{Dimension formulas for cusp forms of general type}\label{dimSkGtable}
 \begin{equation*}
  \renewcommand{\arraystretch}{1.9}
  \begin{array}{cc} 
   \toprule
    \Gamma&\displaystyle\sum_{k=0}^\infty\dim S_k^{\text{\bf(G)}}(\Gamma)t^k\\
   \toprule
    \Gamma(2)&\dfrac{t^8(1+t+t^2)(10 - t + 12t^2 - 5t^3 + 2t^4 + 13t^5 - 16t^6+t^7)}{(1-t^2)^2(1-t^4) \left(1-t^6\right)}\\
   \midrule
    \!\!\!\Sp(4,\Z)\!\!\!&\dfrac{t^{20}(1+t^2+t^4-t^{12}-t^{14}+t^{15})}{(1-t^4)(1-t^6)(1-t^{10})(1-t^{12})}\\
   \midrule
    \mathrm{K}(2)&\dfrac{t^{16}(1+t^3+t^4+t^7+t^8-2t^9-2t^{10}+t^{11})}{(1-t^2)(1-t^4)(1-t^6)(1-t^{12})}\\
   \midrule
    \mathrm{K}(4)&\dfrac{t^{11}(1 + t + t^3 + t^4 + 2t^5 + 2t^8+2t^9+t^{12}+t^{13}-2t^{14}-3t^{15}+t^{16})}{(1-t^2)(1-t^4)(1-t^6)(1-t^{12})}\\
   \midrule
    \Gamma_0(2)&\dfrac{t^{12}(2+2t^2-t^4-2t^6+t^7)}{(1-t^2)(1-t^4)^2(1-t^6)}\\
   \midrule
    \Gamma_0(4)&\dfrac{t^8(3+t^3+3t^4-4t^6+t^7)}{(1-t^2)^3(1-t^6)}\\
   \midrule
    \Gamma_0^*(4)&\dfrac{t^8(4+3t+t^2+t^3+4t^4-t^5-5t^6+t^7)}{(1-t^2)^3(1-t^6)}\\
   \midrule
    \Gamma_0'(2)&\dfrac{t^{12}(1+t^2+t^3+2t^4+t^7+t^8+2t^{11}+t^{12}-2t^{13}-3t^{14}+t^{15})}{(1-t^2)(1-t^4)(1-t^6)(1-t^{12})}\\
   \midrule
    \!\!\!\!\!\!\Gamma_0'(4)&\!\!\!\!\!\!\!\!\dfrac{t^8(1 + t + 5 t^2 + 4 t^3 + 11 t^4 + 6 t^5 + 12 t^6 + 8 t^7 + 8 t^8 + 5 t^9 - t^{10} + t^{11} - \!6 t^{12} - \!2 t^{13} - \!6 t^{14} + t^{15})}{(1-t^4)^2(1-t^6)^2}\!\!\!\!\!\\
   \midrule
    M(4)&\!\!\dfrac{t^{10}(1+2t+4t^2+t^3+3t^4+4t^5+5t^6+4t^9+4t^{10}-t^{12}+3t^{13}+t^{14}-3t^{15}-5t^{16}+t^{17})}{(1-t^2)(1-t^4)(1-t^6)(1-t^{12})}\!\!\\
   \midrule	
    B(2)&\dfrac{t^{10}(1+t+t^2)(1-t+4t^2-2t^3+t^4+3t^5-5t^6+t^7)}{(1-t^2)(1-t^4)^2(1-t^6)}\\
   \bottomrule
  \end{array}
 \end{equation*}
\end{table}
\clearpage
\subsection{Dimensions for low weights}

\vspace{-3ex}
\begin{table}[hbt!]
 \caption{Dimensions for low weights: All modular forms}\vspace{-2ex}
 \label{Mklowweightstable}
\begin{equation*}\setlength{\arraycolsep}{1ex}
 \begin{array}{ccccccccccccccccccccc}
  \toprule
   &\multicolumn{20}{c}{\dim M_k(\Gamma)\text{ for }k=\ldots}\\
   \Gamma&1&2&3&4&5&6&7&8&9&10&11&12&13&14&15&16&17&18&19&20\\
  \toprule
   \Gamma(2)&0&5&0&15&1&35&5&69&15&121&35&195&69&295&121&425&195&589&295&791\\
  \midrule
   \!\!\!\Sp(4,\Z)\!\!\!&0&0&0&1&0&1&0&1&0&2&0&3&0&2&0&4&0&4&0&5\\
  \midrule
   K(2)&0&0&0&1&0&1&0&2&0&2&1&5&0&3&1&7&1&7&2&10\\
  \midrule
   K(4)&0&0&0&2&0&2&1&4&1&5&3&10&3&9&6&17&7&19&12&27\\
  \midrule
   \Gamma_0(2)&0&1&0&3&0&4&0&7&0&9&0&14&0&17&0&24&0&29&1&38\\
  \midrule
   \Gamma_0(4)&0&3&0&7&0&14&0&24&0&38&1&57&3&81&7&111&14&148&24&192\\
  \midrule
   \Gamma_0^*(4)&0&3&0&7&1&15&3&27&7&45&15&71&27&105&45&149&71&205&105&273\\
  \midrule
   \Gamma_0'(2)&0&0&0&2&0&2&0&4&0&5&1&10&0&9&2&17&2&19&4&26\\
  \midrule
   \Gamma_0'(4)&0&0&0&4&0&6&1&12&2&20&7&36&10&46&22&75&32&98&50&133\\
  \midrule
   M(4)&0&0&0&3&0&3&1&8&1&10&5&21&5&23&13&41&16&49&28&71\\
  \midrule
   B(2)&0&1&0&4&0&5&0&11&0&14&1&24&1&30&4&45&5&55&11&76\\
  \bottomrule
 \end{array}
\end{equation*}
\end{table}

\vspace{-2ex}
\begin{table}[hbt!]
 \caption{Dimensions for low weights: All cusp forms}\vspace{-2ex}
 \label{Sklowweightstable}
\begin{equation*}\setlength{\arraycolsep}{1ex}
 \begin{array}{ccccccccccccccccccccc}
  \toprule
   &\multicolumn{20}{c}{\dim S_k(\Gamma)\text{ for }k=\ldots}\\
   \Gamma&1&2&3&4&5&6&7&8&9&10&11&12&13&14&15&16&17&18&19&20\\
  \toprule
   \Gamma(2)&0&0&0&0&1&5&5&24&15&61&35&120&69&205&121&320&195&469&295&656\\
  \midrule
   \!\!\!\Sp(4,\Z)\!\!\!&0&0&0&0&0&0&0&0&0&1&0&1&0&1&0&2&0&2&0&3\\
  \midrule
   K(2)&0&0&0&0&0&0&0&1&0&1&1&2&0&2&1&4&1&4&2&7\\
  \midrule
   K(4)&0&0&0&0&0&0&1&1&1&2&3&4&3&5&6&10&7&12&12&19\\
  \midrule
   \Gamma_0(2)&0&0&0&0&0&1&0&2&0&4&0&7&0&10&0&15&0&20&1&27\\
  \midrule
   \Gamma_0(4)&0&0&0&0&0&3&0&9&0&19&1&34&3&54&7&80&14&113&24&153\\
  \midrule
   \Gamma_0^*(4)&0&0&0&0&1&3&3&10&7&23&15&44&27&73&45&112&71&163&105&226\\
  \midrule
   \Gamma_0'(2)&0&0&0&0&0&0&0&1&0&2&1&4&0&5&2&10&2&12&4&18\\
  \midrule
   \Gamma_0'(4)&0&0&0&0&0&0&1&3&2&9&7&19&10&30&22&53&32&74&50&106\\
  \midrule
   M(4)&0&0&0&0&0&0&1&2&1&4&5&10&5&14&13&27&16&35&28&54\\
  \midrule
   B(2)&0&0&0&0&0&1&0&3&0&6&1&12&1&18&4&29&5&39&11&56\\
  \bottomrule
 \end{array}
\end{equation*}
\end{table}

\clearpage
\begin{table}[hbt!]
 \caption{Dimensions for low weights: Cusp forms of Saito-Kurokawa type}
 \label{SkPlowweightstable}
\begin{equation*}\setlength{\arraycolsep}{1ex}
 \begin{array}{ccccccccccccccccccccc}
  \toprule
   &\multicolumn{20}{c}{\dim S_k^{\text{\bf(P)}}(\Gamma)\text{ for }k=\ldots}\\
   \Gamma&1&2&3&4&5&6&7&8&9&10&11&12&13&14&15&16&17&18&19&20\\
  \toprule
   \Gamma(2)&0&0&0&0&1&5&5&14&6&20&11&29&11&34&16&44&17&49&21&58\\
  \midrule
   \!\!\!\Sp(4,\Z)\!\!\!&0&0&0&0&0&0&0&0&0&1&0&1&0&1&0&2&0&2&0&2\\
  \midrule
   K(2)&0&0&0&0&0&0&0&1&0&1&1&2&0&2&1&3&1&3&1&4\\
  \midrule
   K(4)&0&0&0&0&0&0&1&1&1&2&2&3&2&3&3&5&3&5&4&6\\
  \midrule
   \Gamma_0(2)&0&0&0&0&0&1&0&2&0&4&0&5&0&6&0&8&0&9&0&10\\
  \midrule
   \Gamma_0(4)&0&0&0&0&0&3&0&6&0&10&0&13&0&16&0&20&0&23&0&26\\
  \midrule
   \Gamma_0^*(4)&0&0&0&0&1&3&3&6&4&10&5&13&7&16&8&20&9&23&11&26\\
  \midrule
   \Gamma_0'(2)&0&0&0&0&0&0&0&1&0&2&1&3&0&3&1&5&1&5&1&6\\
  \midrule
   \Gamma_0'(4)&0&0&0&0&0&0&1&2&1&4&3&6&2&6&4&10&4&10&5&12\\
  \midrule
   M(4)&0&0&0&0&0&0&1&2&1&3&3&5&2&5&4&8&4&8&5&10\\
  \midrule
   B(2)&0&0&0&0&0&1&0&3&0&5&1&7&0&8&1&11&1&12&1&14\\
  \bottomrule
 \end{array}
\end{equation*}
\end{table}

\begin{table}[hbt!]
 \caption{Dimensions for low weights: Cusp forms of general type}\vspace{-2ex}
 \label{SkGlowweightstable}
\begin{equation*}\setlength{\arraycolsep}{1ex}
 \begin{array}{ccccccccccccccccccccc}
  \toprule
   &\multicolumn{20}{c}{\dim S_k^{\text{\bf(G)}}(\Gamma)\text{ for }k=\ldots}\\
   \Gamma&1&2&3&4&5&6&7&8&9&10&11&12&13&14&15&16&17&18&19&20\\
  \toprule
   \Gamma(2)&0&0&0&0&0&0&0&10&9&41&24&91&58&171&105&276&178&420&274&598\\
  \midrule
   \!\!\!\Sp(4,\Z)\!\!\!&0&0&0&0&0&0&0&0&0&0&0&0&0&0&0&0&0&0&0&1\\
  \midrule
   K(2)&0&0&0&0&0&0&0&0&0&0&0&0&0&0&0&1&0&1&1&3\\
  \midrule
   K(4)&0&0&0&0&0&0&0&0&0&0&1&1&1&2&3&5&4&7&8&13\\
  \midrule
   \Gamma_0(2)&0&0&0&0&0&0&0&0&0&0&0&2&0&4&0&7&0&11&1&17\\
  \midrule
   \Gamma_0(4)&0&0&0&0&0&0&0&3&0&9&1&21&3&38&7&60&14&90&24&127\\
  \midrule
   \Gamma_0^*(4)&0&0&0&0&0&0&0&4&3&13&10&31&20&57&37&92&62&140&94&200\\
  \midrule
   \Gamma_0'(2)&0&0&0&0&0&0&0&0&0&0&0&1&0&2&1&5&1&7&3&12\\
  \midrule
   \Gamma_0'(4)&0&0&0&0&0&0&0&1&1&5&4&13&8&24&18&43&28&64&45&94\\
  \midrule
   M(4)&0&0&0&0&0&0&0&0&0&1&2&5&3&9&9&19&12&27&23&44\\
  \midrule
   B(2)&0&0&0&0&0&0&0&0&0&1&0&5&1&10&3&18&4&27&10&42\\
  \bottomrule
 \end{array}
\end{equation*}
\end{table}
\clearpage
\subsection{Number of automorphic representations}

\vspace{-1ex}
\begin{table}[hbt!]
 \caption{Number of automorphic representations: Saito-Kurokawa type}\label{skPOmegatable}\vspace{-1ex}
 \begin{equation*}
  \setlength{\arraycolsep}{2.0ex}\renewcommand{\arraystretch}{1.9}
  \begin{array}{cc|cc|cc} 
   \toprule
    \Omega&\displaystyle\sum_{k=0}^\infty s_k^{\text{\bf(P)}}(\Omega)t^k&\Omega&\displaystyle\sum_{k=0}^\infty s_k^{\text{\bf(P)}}(\Omega)t^k&\Omega&\displaystyle\sum_{k=0}^\infty s_k^{\text{\bf(P)}}(\Omega)t^k\\
   \toprule
    \text{IIb}&\dfrac{t^{10}}{(1-t^2)(1-t^{6})}&\text{VIb}&\dfrac{t^6+t^8-t^{12}}{(1-t^4)(1-t^6)}&\text{XIb}&\dfrac{t^7}{(1-t^2)(1-t^6)}\\
   \midrule
    \text{Vb}&\dfrac{t^8}{(1-t^4)(1-t^6)}&\text{VIc}&\dfrac{t^{11}}{(1-t^4)(1-t^6)}&\text{Va}^*&\dfrac{t^5}{(1-t^4)(1-t^6)}\\
   \bottomrule
  \end{array}
 \end{equation*}
\end{table}

\vspace{-2ex}
\begin{table}[hbt!]
 \caption{Number of automorphic representations: general type}\label{skGOmegatable}\vspace{-3ex}
 \begin{equation*}
  \renewcommand{\arraystretch}{1.9}
  \begin{array}{cc} 
   \toprule
    \Omega&\displaystyle\sum_{k=0}^\infty s_k^{\text{\bf(G)}}(\Omega)t^k\\
   \toprule
    \text{I}&\dfrac{t^{20}(1+t^2+t^4-t^{12}-t^{14}+t^{15})}{(1-t^4)(1-t^6)(1-t^{10})(1-t^{12})}\\
   \midrule
    \text{IIa}&\dfrac{t^{16}(1+t^2+t^3-t^4-t^6)}{(1-t^4)^2(1-t^5)(1-t^6)}\\
   \midrule
    \!\!\text{IIIa+VIa/b}\!\!&\dfrac{t^{12}(1 + 2 t^2 + 2 t^4 - t^5 + 2 t^6 - 2 t^7 + t^8 - 2 t^9 + t^{10} - 2 t^{11} + t^{14} + t^{16} + t^{17} - t^{18})}{(1-t^4)(1-t^5)(1-t^6)(1-t^{12})}\\
   \midrule
    \text{IVa}&\dfrac{t^{10}(1 + t^2 + t^3 + t^4 - t^8 + t^9 + 2 t^{10} + 
 2 t^{11} + t^{12} + t^{13} - t^{14} - t^{15} - t^{16} - t^{17} + t^{20})}{(1-t^4)(1-t^5)(1-t^6)(1-t^{12})}\\
   \midrule
    \text{Va/a}^*&\dfrac{t^{15}(1 + t^2 - t^5 - t^7 + t^{10})}{(1-t^4)(1-t^5)(1-t^6)(1-t^{12})}\\
   \midrule
    \!\!\text{VII+VIIIa/b}\!\!&\dfrac{t^{10}(1+t^2-t^6+t^7)}{(1-t^4)^2(1-t^6)^2}\\
   \midrule
    \text{IXa}&\dfrac{t^8(1+t^{11})}{(1-t^2)(1-t^4)(1-t^6)(1-t^{12})}\\
   \midrule
    \text{X}&\dfrac{t^{11}(1+t^8+t^9-t^{12})}{(1-t^2)(1-t^4)(1-t^6)(1-t^{12})}\\
   \midrule
    \text{XIa}&\dfrac{t^{12}(1 + t^3 + t^4 - t^7 - t^8 + t^{11})}{(1-t^2)(1-t^4)(1-t^6)(1-t^{12})}\\
   \midrule
    \text{sc(16)}&\dfrac{t^9}{(1-t^2)(1-t^4)^2(1-t^5)}\\
   \bottomrule
  \end{array}
 \end{equation*}
\end{table}

\begin{table}[hbt!]
 \caption{Number of automorphic representations for low weights}\vspace{-2ex}
 \label{SkOmegalowweightstable}
\begin{equation*}\setlength{\arraycolsep}{1ex}
 \begin{array}{cccccccccccccccccccccc}
  \toprule
   &&\multicolumn{20}{c}{s_k^{\text{\bf(P)}}(\Omega)\text{ or }s_k^{\text{\bf(G)}}(\Omega)\text{ for }k=\ldots}\\
   &\Omega&1&2&3&4&5&6&7&8&9&10&11&12&13&14&15&16&17&18&19&20\\
  \toprule
   \text{\bf(P)}&\text{IIb}&0&0&0&0&0&0&0&0&0&1&0&1&0&1&0&2&0&2&0&2\\
  \cmidrule{2-22}
   &\text{Vb}&0&0&0&0&0&0&0&1&0&0&0&1&0&1&0&1&0&1&0&2\\
  \cmidrule{2-22}
   &\text{VIb}&0&0&0&0&0&1&0&1&0&1&0&1&0&2&0&1&0&2&0&2\\
  \cmidrule{2-22}
   &\text{VIc}&0&0&0&0&0&0&0&0&0&0&1&0&0&0&1&0&1&0&1&0\\
  \cmidrule{2-22}
   &\text{XIb}&0&0&0&0&0&0&1&0&1&0&1&0&2&0&2&0&2&0&3&0\\
  \cmidrule{2-22}
   &\text{Va}^*&0&0&0&0&1&0&0&0&1&0&1&0&1&0&1&0&2&0&1&0\\
  \midrule
   \text{\bf(G)}&\text{I}&0&0&0&0&0&0&0&0&0&0&0&0&0&0&0&0&0&0&0&1\\
  \cmidrule{2-22}
   &\text{IIa}&0&0&0&0&0&0&0&0&0&0&0&0&0&0&0&1&0&1&1&1\\
  \cmidrule{2-22}
   &\text{IIIa+VIa/b}&0&0&0&0&0&0&0&0&0&0&0&1&0&2&0&3&0&5&0&6\\
  \cmidrule{2-22}
   &\text{IVa}&0&0&0&0&0&0&0&0&0&1&0&1&1&2&1&2&2&3&4&6\\
  \cmidrule{2-22}
   &\text{Va/a}^*&0&0&0&0&0&0&0&0&0&0&0&0&0&0&1&0&1&0&1&0\\
  \cmidrule{2-22}
   &\text{VII+VIIIa/b}&0&0&0&0&0&0&0&0&0&1&0&1&0&2&0&3&1&5&0&5\\
  \cmidrule{2-22}
   &\text{IXa}&0&0&0&0&0&0&0&1&0&1&0&2&0&3&0&4&0&5&1&8\\
  \cmidrule{2-22}
   &\text{X}&0&0&0&0&0&0&0&0&0&0&1&0&1&0&2&0&3&0&5&1\\
  \cmidrule{2-22}
   &\text{XIa}&0&0&0&0&0&0&0&0&0&0&0&1&0&1&1&3&1&4&1&5\\
  \cmidrule{2-22}
   &\text{sc(16)}&0&0&0&0&0&0&0&0&1&0&1&0&3&1&3&1&6&3&7&3\\
  \bottomrule
 \end{array}
\end{equation*}
\end{table}

\end{appendix}

	\bibliographystyle{plain}
	\bibliography{dimKl4.bib}
	
	\vspace{5ex}
	\noindent Department of Mathematics, Lafayette College, Easton, Pennsylvania, USA.

\noindent E-mail address: {\tt royma@lafayette.edu}	
	
	\vspace{2ex}
	\noindent Department of Mathematics, University of North Texas, Denton, TX 76203-5017, USA.
	
	\noindent E-mail address: {\tt ralf.schmidt@unt.edu}
	
	\vspace{2ex}
	\noindent School of Mathematical Sciences, Xiamen University, Xiamen, Fujian 361005, China.
	
	\noindent E-mail address: {\tt yishaoyun926@xmu.edu.cn}
\end{document}